\documentclass[11pt,a4paper]{article}

\usepackage{cmap}
\usepackage[T1]{fontenc}
\input{glyphtounicode}
\usepackage[utf8]{inputenc}
\usepackage[margin=2.6cm]{geometry}
\usepackage{amsmath,amssymb,amsthm,mathtools}
\usepackage{booktabs,array,enumitem}
\usepackage{pgfplots}
\pgfplotsset{compat=1.18}
\usepackage{aliascnt}
\usepackage{xcolor}
\definecolor{linkblue}{RGB}{0,60,170}
\usepackage[colorlinks=true,linkcolor=linkblue,citecolor=linkblue,urlcolor=linkblue]{hyperref}
\hypersetup{
  pdftitle={The Type-II Error of Test Supermartingales:
            e-Power versus the Chernoff-Stein Exponent},
  pdfauthor={Patrick Forre},
  pdfsubject={Anytime-valid hypothesis testing},
  pdfkeywords={e-value, e-variable, test supermartingale, safe testing,
               anytime-valid inference, type-II error, error exponent,
               Renyi divergence, Chernoff-Stein lemma}}
\usepackage[capitalize,noabbrev]{cleveref}

\allowdisplaybreaks
\newcommand{\doi}[1]{\href{https://doi.org/#1}{\texttt{doi:\detokenize{#1}}}}
\newcommand{\E}{\mathbb{E}}
\newcommand{\Prob}{\mathbb{P}}
\newcommand{\R}{\mathbb{R}}
\newcommand{\N}{\mathbb{N}}
\newcommand{\Xc}{\mathcal{X}}
\newcommand{\Bs}[1]{\mathcal{B}_{#1}}
\newcommand{\Fc}{\mathcal{F}}
\newcommand{\Ev}{\mathcal{E}}
\newcommand{\KL}{\mathrm{KL}}
\newcommand{\KLrev}{\overline{\mathrm{KL}}}
\DeclareMathOperator*{\argmin}{arg\,min}
\newcommand{\Bc}{\mathcal{B}}
\newcommand{\Ren}{\mathrm{D}}
\newcommand{\Var}{\operatorname{Var}}
\newcommand{\evar}{\textsf{e}-variable}
\newcommand{\evars}{\textsf{e}-variables}
\newcommand{\epower}{\textsf{e}-power}
\newcommand{\etm}{test supermartingale}
\newcommand{\eproc}{\textsf{e}-process}
\newcommand{\Hc}{\mathcal{H}}

\newcommand{\lp}{\left(}
\newcommand{\rp}{\right)}
\newcommand{\lB}{\left[}
\newcommand{\rB}{\right]}
\newcommand{\lC}{\left\{}
\newcommand{\st}{\;\middle|\;}
\newcommand{\rC}{\right\}}

\newtheorem{theorem}{Theorem}[section]
\newcommand{\aliasthm}[3]{%
  \newaliascnt{#1}{theorem}%
  \theoremstyle{#3}\newtheorem{#1}[#1]{#2}%
  \aliascntresetthe{#1}%
  \crefname{#1}{#2}{#2s}\Crefname{#1}{#2}{#2s}}

\theoremstyle{plain}
\aliasthm{proposition}{Proposition}{plain}
\aliasthm{lemma}{Lemma}{plain}
\aliasthm{corollary}{Corollary}{plain}
\aliasthm{definition}{Definition}{definition}
\aliasthm{assumption}{Assumption}{definition}
\aliasthm{construction}{Construction}{definition}
\aliasthm{example}{Example}{definition}
\aliasthm{remark}{Remark}{remark}
\crefname{theorem}{Theorem}{Theorems}
\Crefname{theorem}{Theorem}{Theorems}

\title{\bfseries The Type-II Error of Test Supermartingales:\\[3pt]
\textsf{e}-Power versus the Chernoff--Stein Exponent}
\author{Patrick Forr\'e\\[+10pt]
\small{AI4Science Lab}\\[0pt]
\small{Korteweg-de Vries Institute for Mathematics}\\[0pt]
\small{University of Amsterdam}\\[0pt]
\small{\texttt{p.d.forre@uva.nl}}}
\date{}

\begin{document}
\maketitle

\begin{abstract}
\noindent
In safe hypothesis testing with \etm{}s, Ville's inequality provides
anytime-valid type-I error guarantees for every significance level
$\alpha\in(0,1]$, if one rejects the null hypothesis whenever the wealth
process first exceeds $\frac1\alpha$. Due to an inherent asymmetry, the type-II
error behaves differently. We prove two things about the latter, for a simple
null and alternative.

First, the mean growth rate $\E_{P_1}[\log E]$, the \epower{}, that Kelly
betting and growth-rate-optimal \evars{} maximise, bounds nothing on its own.
For every level $c>0$, every $\alpha$ and horizon $t$ we construct \evars{} of
conditional \epower{} exactly $c$ whose probability of not rejecting by $t$ is
arbitrarily close to one. It forces eventual rejection, but no finite-horizon
guarantee follows.

Second, the quantity that does control the type-II error is the
\emph{Chernoff--Stein exponent of an \evar{}},
$\Lambda(E)=\sup_{s\ge0}\lC-\log\E_{P_1}[E^{-s}]\rC$, whose range is exactly
determined: $\sup_E\Lambda(E)=\KL(P_0\|P_1)$, the classical Chernoff--Stein
exponent, and so the ceiling of its own per-\evar{} form. One conditional
application of H\"older's inequality per step gives it, for every \etm{} on an
arbitrary filtered space, with no independence or product structure; the
i.i.d.\ case adds that it is matched, and attained by nothing.

The \epower{} has its own ceiling, $\KL(P_1\|P_0)$, and that one \emph{is}
attained, $P_0$-a.s.\ uniquely, by the likelihood ratio $R$. The two optima are
the same divergence in opposite arguments, at opposite ends of the flattened
family $R^{\beta}/\E_{P_0}[R^{\beta}]$: the ceiling as $\beta\downarrow0$, $R$
at $\beta=1$. Which $\beta$ is best is settled by the horizon, exactly: $R$ is
optimal at $t=\log\frac1\alpha/\KL(P_1\|P_0)$ alone, beaten by sharpening
$(\beta>1)$ below it and by flattening above.


(secondary).
\end{abstract}

\tableofcontents

\section{Introduction}
\label{sec:intro}

\subsection*{The asymmetry}

Fix a filtered space $(\Omega,\Fc,(\Fc_t)_{t\in\N_0})$ carrying a null
$\Prob_0$ and an alternative $\Prob_1$ (\Cref{def:setting}); the case to keep in
mind, and the one in force wherever a lower bound is claimed, is the
\emph{i.i.d.\ model} (\Cref{def:iid}), in which one observes $X_1,X_2,\dots$
i.i.d.\ from a one-observation law $P_i$, so that $\Prob_i=P_i^{\otimes\N}$ and
$\Fc_t=\sigma(X_1,\dots,X_t)$. Now bet: choose at each step a \emph{conditional
\evar{}} $E_t$, a nonnegative statistic with
$\E_{\Prob_0}[E_t\mid\Fc_{t-1}]\le1$, and accumulate wealth
$W_t=\prod_{n\le t}E_n$. Reject at the \emph{rejection time}
\[
   \tau_\alpha\ :=\ \inf\lC t\in\N\st W_t>\tfrac1\alpha\rC ,
\]
the first moment at which the wealth passes $1/\alpha$. Ville's inequality then
gives $\Prob_0(\tau_\alpha<\infty)\le\alpha$ --- type-I error controlled
uniformly over all times, for any betting sequence whatsoever
(\Cref{thm:typeI}). This is the whole appeal of the framework, and it costs the
statistician nothing.

The type-II error is the opposite. It is not controlled by the structure, it is
not a single number but a function of a horizon --- we write
$\gamma_t(\alpha)=\Prob_1(W_t\le\frac1\alpha)$ for the fixed-horizon error and
$\bar\gamma_t(\alpha)=\Prob_1(\tau_\alpha>t)$ for the sequential one
(\Cref{rem:noinfty}) --- and what
it depends on is not what the design literature optimises. This paper is about
that gap. Its two results, and the dichotomy they produce, are stated below;
they are proved in \Cref{sec:nfl,sec:ceiling}, and
\Cref{sec:horizon} turns the second one into a design rule.

\subsection*{First: \texorpdfstring{\epower{}}{e-power} alone bounds nothing}

The standard summary of a betting design is its \emph{\epower{}}, the growth
rate $\E_{P_1}[\log E]$ of a single \evar{} $E$ --- an element of the class
$\Ev$ of nonnegative statistics of one observation with $\E_{P_0}[E]\le1$
\cite{VW21} --- that Kelly betting \cite{Kel56,Bre61},
growth-rate-optimal \evars{} \cite{GHK24} and the reverse information
projection \cite{LRR25} all maximise. Since $\E_{\Prob_1}[\log W_t]\ge tc$ when
the conditional \epower{} is at least $c$, the wealth grows linearly in
expectation, and it is tempting to conclude that it therefore crosses $1/\alpha$
before long.

It does not follow, and the failure is total rather than technical.

\begin{quote}
\textbf{\Cref{thm:nfl}.} \emph{For every $c>0$, $\alpha\in(0,1)$,
$\eta\in(0,1)$ and $t\in\N$ there are hypotheses $(\Prob_0,\Prob_1)$ and a
betting sequence with
$\E_{\Prob_1}[\log E_n\mid\Fc_{n-1}]=c$ exactly, for every $n$, and with
$\Prob_1(\tau_\alpha>t)\ge1-\eta$. Consequently no function
$\Phi(c,\alpha,t)<1$ bounds the type-II error of every design of conditional
\epower{} $c$, uniformly over testing problems.}
\end{quote}

The construction (\Cref{con:nfl}) is one catastrophic first bet, made deep
enough that the remaining $t-1$ bets of the same \epower{} cannot climb back to
the threshold. A lower bound on a conditional mean permits an arbitrarily heavy
lower tail, and that is all the room the counterexample needs.

The theorem does not contradict consistency, and the distinction matters.

\begin{quote}
\textbf{\Cref{prp:powerone}.} \emph{For a fixed $E\in\Ev$ with
$\E_{P_1}[\log E]>0$, the i.i.d.\ product it
generates rejects eventually: $\tau_\alpha<\infty$ $\Prob_1$-almost surely, for
every $\alpha$.}
\end{quote}

What fails in \Cref{thm:nfl} is therefore not consistency but \emph{uniformity}
over the class that \epower{} alone defines --- and uniformity is exactly what a
sample-size calculation needs (\Cref{rem:alone}). The
missing ingredient is any quantitative control of the lower tail of $\log E_t$;
\Cref{rem:cramer} sorts out which tail hypothesis buys which rate, and the rest
of the paper takes the sharpest of them.

\subsection*{Second: the exponent has an exact ceiling}

The quantity that does control the error is the \emph{tilted exponent}
$\Lambda_s(E)=-\log\E_{P_1}[E^{-s}]$ and its optimised form, the
\emph{Chernoff--Stein exponent} $\Lambda(E)=\sup_{s\ge0}\Lambda_s(E)$ --- so named
because at $E=R$ it is the Chernoff information of the pair
(\Cref{cor:chernoffinfo}). It lower-bounds the type-II exponent
of the i.i.d.\ product generated by $E$ always, and equals it when the \epower{}
is positive and $0$ is interior to the domain of the relevant rate function
(\Cref{prp:exponent}). Its exact range over the whole class is the
second result. Three quantities enter. Write
$R=\frac{dP_1}{dP_0}$ for the likelihood ratio,
$Z_\beta=\E_{P_0}[R^{\beta}]$ for its \emph{Hellinger integral} of order
$\beta\in(0,1]$, and
\[
   \lambda(\beta)\ :=\ -\tfrac1\beta\log Z_\beta
   \ =\ \Ren_{1-\beta}(P_0\|P_1)
\]
for the associated R\'enyi divergence of order $1-\beta$, a decreasing function
of $\beta$ (\Cref{def:renyi}). The \emph{flattened} designs
$R_\beta=R^{\beta}/Z_\beta$ are the \evars{} obtained by raising $R$ to a
power and renormalising; $\beta=1$ recovers $R$ itself, and $\beta>1$, legitimate
whenever $Z_\beta<\infty$, \emph{sharpens} rather than flattens
(\Cref{def:flattened}). Only $\beta\in(0,1]$ arises in the ceiling, where
$\beta=\frac1{1+s}$ is tied to the tilt; the horizon question of
\Cref{sec:horizon} uses the whole range.

\begin{quote}
\textbf{\Cref{thm:upper}, \Cref{prp:upper-seq}, \Cref{thm:ceiling}.}
\emph{For every \evar{} $E$ and every $s>0$, with $\beta=\frac1{1+s}$,}
\[
   \Lambda_s(E)\ \le\ \lambda(\beta),
   \qquad\text{with equality iff }E=R^{\beta}/Z_\beta ,
\]
\emph{(equality $P_0$-a.s.), and the same bound holds step by step for an
arbitrary test supermartingale. Consequently}
\[
   \sup_{\mathcal W}\ \Lambda^{(t)}\ =\ \lim_{\beta\downarrow0}\lambda(\beta)
   \ =\ \KL(P_0\|P_1)\qquad\text{for every }t\in\N,
\]
\emph{the supremum being over all test supermartingales, $\Lambda^{(t)}$
denoting the optimised $t$-step exponent of \Cref{def:exponent}, and approached
along the flattened family as $\beta\downarrow0$; and if $P_0$ and $P_1$ are
mutually absolutely continuous $(P_0\sim P_1)$, $\KL(P_0\|P_1)<\infty$ and
$\log R$ is not $P_0$-a.s.\ constant,
it is attained by none of them at any horizon.}
\end{quote}

The upper bound is a single application of H\"older's inequality
(\Cref{thm:upper}); what makes it hold for adaptive designs is that it can be
applied conditionally on the past, once per step, and the steps multiply
(\Cref{prp:upper-seq}). That is also what makes it general. Nothing in the
argument uses the product structure, and on an arbitrary filtered space
\Cref{cor:ceiling} reads
\[
   \Lambda^{(t)}\ \le\ \frac1t\sum_{n\le t}\ \operatorname*{ess\,sup}\
   \KLrev_n ,
   \qquad
   \KLrev_n=-\E_{\Prob_0}\lB\log R_n\mid\Fc_{n-1}\rB ,
\]
the average of the per-step conditional relative entropies, with $R_n$ the
one-step likelihood ratio --- the per-step forms of the Chernoff--Stein exponent
named below. A design attaining this bound in the limit exists in the same generality
(\Cref{lem:flat-cond}). What the i.i.d.\ model supplies is not the bound and not
the design but three things about them: that the essential suprema are
superfluous and the average is a single number, $\KL(P_0\|P_1)$; that the bound
is therefore matched, and not attained; and that it binds the type-II error and
not merely the exponent. \Cref{rem:generalceiling} says which hypothesis each of
the three needs, and \Cref{app:general} supplies them. The value $\KL(P_0\|P_1)$ is the Chernoff--Stein
exponent, the best available to \emph{any} level-$\alpha$ test when
$\KL(P_0\|P_1)<\infty$, so the class of test supermartingales is
exponent-optimal as a class while no member of it attains the ceiling at any
finite horizon (\Cref{prp:stein-ceiling}, \Cref{rem:stein}). The lower bound is attained inside
the smallest natural subclass, the i.i.d.\ products of a single \evar{}, so
adaptivity buys nothing \emph{in the exponent} (\Cref{rem:scope}) --- though it
buys a constant, and \Cref{sec:adaptive} says how much.

\subsection*{The dichotomy, and what to bet}

Both scalars have ceilings, and this is where the two halves meet.

\begin{quote}
\textbf{\Cref{thm:two}.} \emph{$\sup_E\E_{P_1}[\log E]=\KL(P_1\|P_0)$, attained
at $E=R$ and, up to $P_0$-null modification, only there; while
$\sup_E\Lambda(E)=\KL(P_0\|P_1)$, attained by nobody --- under the
non-degeneracy hypotheses above.}
\end{quote}

The same divergence in opposite arguments, one optimum attained and one not.
The two are not even comparable as numbers (\Cref{rem:notcomparable}). Along the
family the tension is visible in one curve: the \epower{}
$c_\beta=\beta(\KL(P_1\|P_0)+\lambda(\beta))$ rises to its ceiling at $\beta=1$
and falls away on both sides, while $\lambda(\beta)$ decreases throughout, from
its ceiling at $\beta=0$ (\Cref{lem:family}, \Cref{fig:two}).

Which end one wants is settled by the horizon, and here the answer is exact.
Every member of the family performs the same likelihood-ratio test with a
different threshold: writing $G_t=\sum_{n\le t}\log R(X_n)$ and
$B_t(\beta)=\frac{\ell}{\beta}-t\lambda(\beta)$ with $\ell=\log\frac1\alpha$,
the type-II error of $R_\beta$ is \emph{exactly} $\Prob_1(G_t\le B_t(\beta))$
(\Cref{prp:exactfamily}). Minimising the error over the family is therefore
minimising $B_t$, and that has a clean answer.

\begin{quote}
\textbf{\Cref{thm:betastar}.} \emph{Under the non-degeneracy hypotheses above,
$B_t$ has, whenever one exists, a unique minimiser $\beta^{\ast}(t)$, and
$\beta^{\ast}(t)=1$ --- the
growth-optimal design is the horizon-optimal one within the flattened family ---
if and only if $t=\ell/\KL(P_1\|P_0)$; below that horizon $R$ is beaten by
sharpening $(\beta>1)$ and above it by flattening. If moreover $\E_{P_0}[R^{-\delta}]<\infty$ for some
$\delta>0$ and $V_0=\Var_{P_0}(\log R)\in(0,\infty)$ is the variance of the
log-likelihood-ratio under the null, then as $t\to\infty$ one
has $\beta^{\ast}(t)=\sqrt{2\ell/(tV_0)}(1+o(1))$ and the design
$R_{\beta^{\ast}(t)}$ satisfies}
\[
   \gamma_t(\alpha)\ \le\ \exp\lp-t\,\KL(P_0\|P_1)+\sqrt{2\ell tV_0}+O(1)\rp .
\]
\end{quote}

The crossover $\ell/\KL(P_1\|P_0)$ is exactly the number of observations $R$
needs to accumulate $\ell$ nats of evidence on average. So the growth-optimal
\evar{} is the right design at exactly the horizon at which it typically
rejects, and is beaten on either side of it (\Cref{rem:crossover}). The price of the near-optimal
exponent is paid in horizon: a tilt of order $\epsilon$ buys an exponent within
$\epsilon$ of the ceiling and costs a horizon of order $\epsilon^{-1}$ before
the test has any power at all (\Cref{lem:floor}, \Cref{cor:price}). In the
Gaussian location model everything is closed form, and $R$ realises exactly a
quarter of the ceiling (\Cref{ex:gauss}).

\subsection*{Scope}

The null and the alternative are simple and both are known; $\KL(P_0\|P_1)$ is
an oracle quantity and the designs approaching it are built from $R$. The
composite case, which is what growth-rate-optimal \evars{} exist for, is not
treated (\Cref{rem:oracle}). Neither is the exact second-order behaviour beyond
the $\sqrt t$ term, nor misspecification (\Cref{rem:scopefinal}). The
observations are i.i.d.\ wherever a lower bound is claimed; the upper bounds are
proved without that assumption, and \Cref{app:general} states what each of the
remaining results becomes, and at what cost, when it is dropped. Results quoted
from the literature are collected in \Cref{app:quoted}, restated in the notation
used here; \Cref{sec:notation} is an index of that notation.

\section{Setup and notation}
\label{sec:setup}

We fix notation once and use it without further comment. Every map is given
with its domain and codomain. Throughout, $\log$ is the natural logarithm,
$\log 0:=-\infty$, $e^{-\infty}:=0$, $0\cdot\infty:=0$, and for
$x\in[0,\infty]$ we set $x^{0}:=1$ and, for $s>0$, $0^{-s}:=+\infty$ and
$(+\infty)^{-s}:=0$; also $\log(+\infty):=+\infty$, $e^{+\infty}:=+\infty$ and
$0\cdot(-\infty):=0$, so that $0\log0=0$. With these conventions $x\mapsto x^{-s}$ is nonincreasing
on all of $[0,\infty]$ for every $s\ge0$.

\subsection{The filtered setting}

\begin{definition}[Setting]\label{def:setting}
Let $(\Omega,\Fc)$ be a measurable space carrying a filtration
$(\Fc_t)_{t\in\N_0}$, where $\N_0:=\N\cup\{0\}$, with $\Fc_0=\{\emptyset,\Omega\}$
and $\Fc_\infty:=\sigma\lp\bigcup_{t\in\N}\Fc_t\rp\subseteq\Fc$, and let
$\Prob_0,\Prob_1$ be probability measures on $(\Omega,\Fc)$, the \emph{null} and
the \emph{alternative}. Nothing is assumed about how the filtration arises. We
write $\Bs{S}$ for the $\sigma$-algebra carried by a space $S$.
\end{definition}

\begin{definition}[The i.i.d.\ model]\label{def:iid}
The \emph{i.i.d.\ model} generated by a pair $P_0,P_1$ of probability measures on
a measurable space $(\Xc,\Bs{\Xc})$, the \emph{observation space}, is the
instance of \Cref{def:setting} in which
\[
   \Omega:=\Xc^{\N},\qquad
   X_n:\Omega\to\Xc,\quad X_n(\omega):=\omega_n\ \ (n\in\N),
\]
$\Fc:=\Bs{\Xc}^{\otimes\N}$, $\Fc_t:=\sigma(X_1,\dots,X_t)$ and
$\Prob_i:=P_i^{\otimes\N}$. We call $P_0,P_1$ the \emph{one-observation} laws.
\Cref{def:setting} is in force throughout; the i.i.d.\ model is assumed only in
those statements that name it. A statement is therefore i.i.d.-only exactly when
the one-observation laws $P_0,P_1$ --- as opposed to the path measures
$\Prob_0,\Prob_1$ --- occur in it, and we say so as well.
\end{definition}

\begin{assumption}[Standing]\label{ass:standing}
For every $t\in\N$, $\Prob_1\ll\Prob_0$ on $\Fc_t$ and
\[
   \KL\lp\Prob_1|_{\Fc_t}\,\big\|\,\Prob_0|_{\Fc_t}\rp\ <\ \infty ,
   \qquad\text{where}\qquad
   \KL(Q\|P):=\begin{cases}
      \E_{Q}\lB\log\frac{dQ}{dP}\rB, & Q\ll P,\\
      +\infty, & \text{otherwise.}
   \end{cases}
\]
\end{assumption}

\begin{remark}[What the standing assumption says in the i.i.d.\ model]
\label{rem:one-obs}
In the i.i.d.\ model the chain rule gives
$\KL(\Prob_1|_{\Fc_t}\|\Prob_0|_{\Fc_t})=t\,\KL(P_1\|P_0)$, so
\Cref{ass:standing} says exactly
\[
   P_1\ll P_0\qquad\text{and}\qquad \KL(P_1\|P_0)<\infty ,
\]
and we then write $R:=\frac{dP_1}{dP_0}:\Xc\to[0,\infty)$ for a fixed version of
the density, so that $\KL(P_1\|P_0)=\E_{P_1}[\log R]$ and $\E_{P_0}[R]=1$. Note
that the assumption is imposed at each \emph{finite} $t$ only: the path
divergence $\KL(\Prob_1|_{\Fc_t}\|\Prob_0|_{\Fc_t})$ diverges as $t\to\infty$
unless $P_0=P_1$, which is why every divergence between $P_0$ and $P_1$ below is
per observation, and every divergence in the general setting is per step and
conditional on the past.
\end{remark}

\begin{definition}[Likelihood ratios]\label{def:ratios}
Under \Cref{ass:standing} let
\[
   L_t:=\frac{d\Prob_1|_{\Fc_t}}{d\Prob_0|_{\Fc_t}}\qquad(t\in\N_0) ,
\]
a nonnegative $\Prob_0$-martingale with $L_0=1$, and define the \emph{one-step
ratio}
\[
   R_t:\Omega\to[0,\infty),\qquad
   R_t:=\frac{L_t}{L_{t-1}}\ \ \text{on }\lC L_{t-1}>0\rC ,
   \qquad R_t:=1\ \ \text{elsewhere} .
\]
Because
$\Prob_1(L_{t-1}=0)=\E_{\Prob_0}[L_{t-1}\mathbf 1_{\{L_{t-1}=0\}}]=0$, the
convention off $\lC L_{t-1}>0\rC$ affects nothing asserted $\Prob_1$-a.s. In the
i.i.d.\ model $L_t=\prod_{n\le t}R(X_n)$, and $R_t=R(X_t)$ $\Prob_1$-a.s.\ ---
but not $\Prob_0$-a.s., since $\lC L_{t-1}=0\rC$ may carry $\Prob_0$-mass. Every
identification below is therefore stated under $\Prob_1$.
\end{definition}

\begin{lemma}[Change of measure, one step]\label{lem:bayes}
Let $t\in\N$ and let $Y:\Omega\to[0,\infty]$ be $\Fc_t$-measurable. Then
\[
   \E_{\Prob_1}\lB Y\mid\Fc_{t-1}\rB
   \ =\ \E_{\Prob_0}\lB R_t\,Y\mid\Fc_{t-1}\rB
   \qquad\Prob_1\text{-a.s.}
\]
\end{lemma}

\begin{proof}
First, $L_t=0$ $\Prob_0$-a.s.\ on $\lC L_{t-1}=0\rC$: by the martingale property
$\E_{\Prob_0}[L_t\mathbf 1_{\{L_{t-1}=0\}}]
=\E_{\Prob_0}[L_{t-1}\mathbf 1_{\{L_{t-1}=0\}}]=0$ and $L_t\ge0$. Hence
$L_{t-1}R_tY=L_tY$ $\Prob_0$-a.s., both sides vanishing on
$\lC L_{t-1}=0\rC$. Now let $A\in\Fc_{t-1}$. Using the density $L_{t-1}$ on
$\Fc_{t-1}$, then the pull-out property for the nonnegative
$\Fc_{t-1}$-measurable factor $\mathbf 1_AL_{t-1}$, then the identity just
proved, then the density $L_t$ on $\Fc_t$:
\[
   \E_{\Prob_1}\lB\mathbf 1_A\,\E_{\Prob_0}[R_tY\mid\Fc_{t-1}]\rB
   =\E_{\Prob_0}\lB\mathbf 1_AL_{t-1}R_tY\rB
   =\E_{\Prob_0}\lB\mathbf 1_AL_tY\rB
   =\E_{\Prob_1}\lB\mathbf 1_AY\rB .
\]
The variable $\E_{\Prob_0}[R_tY\mid\Fc_{t-1}]$ is $\Fc_{t-1}$-measurable and
nonnegative, so this identity for every $A\in\Fc_{t-1}$ is the defining property
of $\E_{\Prob_1}[Y\mid\Fc_{t-1}]$.
\end{proof}

\subsection{\texorpdfstring{\evars}{e-variables} and \texorpdfstring{\etm s}{test supermartingales}}

\begin{definition}[\evar{}]\label{def:evar}
In the i.i.d.\ model (\Cref{def:iid}), a map $E:\Xc\to[0,\infty]$ is an
\emph{\evar{} for $P_0$} if it is
$\Bs{\Xc}$-measurable and
\[
   \E_{P_0}[E]\ \le\ 1 .
\]
We write
\[
   \Ev\ :=\ \lC E:\Xc\to[0,\infty]\ \text{measurable}\st\E_{P_0}[E]\le1\rC .
\]
\end{definition}

\begin{definition}[Conditional \evar{}]\label{def:cond-evar}
Let $t\in\N$. A map $E_t:\Omega\to[0,\infty]$ is a \emph{conditional \evar{} for
$\Prob_0$ at time $t$} if it is $\Fc_t$-measurable and
\[
   \E_{\Prob_0}\lB E_t\mid\Fc_{t-1}\rB\ \le\ 1
   \qquad \Prob_0\text{-a.s.}
\]
A sequence $(E_t)_{t\in\N}$ of such maps is called a \emph{betting sequence}.
\end{definition}

\begin{definition}[Wealth process, \etm{}]\label{def:wealth}
Given a betting sequence $(E_t)_{t\in\N}$, define
\[
   W:\N_0\times\Omega\to[0,\infty],\qquad
   W_0:=1,\qquad W_t:=\prod_{n=1}^{t}E_n .
\]
We call $(W_t)_{t\in\N_0}$ the \emph{\etm{}} of the sequence
\cite{SSVV11,RGVS23}. In the i.i.d.\ model, when $E_t=E(X_t)$ for a single
$E\in\Ev$ and all $t$, we call $(W_t)_{t\in\N_0}$ the \emph{i.i.d.\ product} generated by
$E$ and write $W_t=W_t(E)$.
\end{definition}

\begin{lemma}[The wealth process is a supermartingale]\label{lem:supmg}
$(W_t)_{t\in\N_0}$ is a nonnegative $\Prob_0$-supermartingale with $W_0=1$, and
$E_t<\infty$ $\Prob_0$-a.s.\ for every $t$.
\end{lemma}

\begin{proof}
$W_t=W_{t-1}E_t$ in $[0,\infty]$, with the convention $0\cdot\infty:=0$. Since
$W_{t-1}$ is nonnegative and $\Fc_{t-1}$-measurable, the pull-out property for
nonnegative extended-real variables gives
$\E_{\Prob_0}[W_t\mid\Fc_{t-1}]=W_{t-1}\E_{\Prob_0}[E_t\mid\Fc_{t-1}]\le W_{t-1}$,
and induction from $W_0=1$ gives $\E_{\Prob_0}[W_t]\le1<\infty$. Finally
$\E_{\Prob_0}[E_t\mid\Fc_{t-1}]\le1$ forces $E_t<\infty$ $\Prob_0$-a.s.
\end{proof}

\begin{definition}[Rejection rule and rejection time]\label{def:reject}
Fix $\alpha\in(0,1)$ and put $\ell:=\log\frac1\alpha\in(0,\infty)$. The test
rejects at the \emph{rejection time}
\[
   \tau_\alpha:\Omega\to\N\cup\{\infty\},\qquad
   \tau_\alpha:=\inf\lC t\in\N\st W_t>\tfrac1\alpha\rC ,
   \qquad \inf\emptyset:=\infty .
\]
\end{definition}

The following is the reason for the framework, and the only place where a
probability under the null is computed. It is Ville's inequality
\cite{Vil39} (restated as \Cref{thm:app-ville}) applied to $(W_t)_{t\in\N_0}$.

\begin{theorem}[Type-I guarantee]\label{thm:typeI}
For every betting sequence and every $\alpha\in(0,1)$,
$\Prob_0(\tau_\alpha<\infty)\le\alpha$.
\end{theorem}

\begin{proof}
$(W_t)_{t\in\N_0}$ is a nonnegative $\Prob_0$-supermartingale with $W_0=1$, and
$\{\tau_\alpha<\infty\}=\{\sup_{t\in\N}W_t>\frac1\alpha\}
\subseteq\{\sup_{t\in\N_0}W_t\ge\frac1\alpha\}$. Apply Ville's
inequality \cite{Vil39}, \cite[Lem.~1]{How20} in the form of
\Cref{thm:app-ville}, with $x=\frac1\alpha$.
\end{proof}

\subsection{The type-II functionals}

\begin{definition}[Type-II error]\label{def:typeII}
Fix $\alpha\in(0,1)$ and a betting sequence. Define
\begin{align}
   \gamma_\bullet(\alpha)&:\N_0\to[0,1], &
   \gamma_t(\alpha)&:=\Prob_1\lp W_t\le\tfrac1\alpha\rp ,
   \label{eq:gamma}\\
   \bar\gamma_\bullet(\alpha)&:\N_0\to[0,1], &
   \bar\gamma_t(\alpha)&:=\Prob_1\lp\tau_\alpha>t\rp
   =\Prob_1\lp\max_{0\le n\le t}W_n\le\tfrac1\alpha\rp .
   \label{eq:gammabar}
\end{align}
We call $\gamma_t$ the \emph{fixed-horizon} and $\bar\gamma_t$ the
\emph{sequential} type-II error.
\end{definition}

\begin{lemma}[Ordering]\label{lem:order}
$\bar\gamma_t(\alpha)\le\gamma_t(\alpha)$ for every $t\in\N_0$ and
$\alpha\in(0,1)$.
\end{lemma}

\begin{proof}
$\{\tau_\alpha>t\}=\bigcap_{n=0}^{t}\{W_n\le\frac1\alpha\}\subseteq
\{W_t\le\frac1\alpha\}$.
\end{proof}

\begin{remark}[There is no type-II error ``at $t=\infty$'']\label{rem:noinfty}
The type-I error is a single number because
$\Prob_0(\sup_{t}W_t>\frac1\alpha)\le\alpha$ involves a supremum over all
times. The formal analogue $\bar\gamma_\infty(\alpha)=\Prob_1(\tau_\alpha=\infty)$
carries no information for the sequences this paper is about: by
\Cref{prp:powerone} below it is $0$ for every i.i.d.\ product generated by an
$E\in\Ev$ of positive \epower{}. (For a general betting sequence it need not
vanish --- $E_t\equiv1$ gives $\tau_\alpha=\infty$ a.s.\ --- but then no
horizon-free statement is available either.) The type-II error is therefore
irreducibly a function of a horizon, and all statements below are indexed by
$t$.
\end{remark}

\subsection{The two scalars attached to an \texorpdfstring{\evar}{e-variable}}

\begin{definition}[\epower{}]\label{def:epower}
The \emph{\epower{}} of an $E\in\Ev$ \cite{VW21} is the quantity
\[
   \E_{P_1}\lB\log E\rB\ \in\ [-\infty,\infty] ,
\]
defined whenever $\E_{P_1}[(\log E)_+]<\infty$ or $\E_{P_1}[(\log E)_-]<\infty$.
Under \Cref{ass:standing} the positive part is always integrable
(\Cref{lem:epower-welldef}), so the \epower{} is a well-defined element of
$[-\infty,\infty)$ for every $E\in\Ev$.
\end{definition}

\begin{remark}[No symbol for the \epower{}]\label{rem:nosymbol}
The \epower{} is written out as $\E_{P_1}[\log E]$ throughout rather than
abbreviated, and its conditional version, used in \Cref{sec:nfl}, as
$\E_{\Prob_1}[\log E_t\mid\Fc_{t-1}]$.
\end{remark}

\begin{definition}[The Chernoff--Stein exponent of an \evar{}]\label{def:exponent}
For $E\in\Ev$ define
\[
   \Lambda_\bullet(E):[0,\infty)\to[-\infty,\infty],\qquad
   \Lambda_s(E):=-\log\E_{P_1}\lB E^{-s}\rB ,
\]
and
\[
   \Lambda:\Ev\to[0,\infty],\qquad
   \Lambda(E):=\sup_{s\in[0,\infty)}\Lambda_s(E) .
\]
Note $\Lambda_0(E)=0$, so $\Lambda(E)\ge0$. We call $\Lambda_s(E)$ the
\emph{tilted exponent} at tilt $s$, and $\Lambda(E)$ the \emph{Chernoff--Stein
exponent of $E$}: it is the type-II error exponent that $E$ achieves in the
Chernoff--Stein regime --- level $\alpha$ fixed, horizon $t\to\infty$
(\Cref{prp:exponent}) --- and the classical Chernoff--Stein exponent
$\KL(P_0\|P_1)$ is its supremum over the class, which is \Cref{thm:ceiling}. The
two levels of the name are related by exactly that theorem. At $E=R$ the value
is the Chernoff information of the pair (\Cref{cor:chernoffinfo}). The name
records the regime, not an identity: \Cref{prp:exponent} proves
$\Lambda(E)\le\liminf_t-\frac1t\log\gamma_t(\alpha)$ always, with equality under
a hypothesis; whether the inequality can be strict we do not know.
For a betting sequence with wealth
process $(W_t)_{t\in\N_0}$ define, in parallel, the \emph{$t$-step exponent}
\[
   \Lambda^{(t)}_\bullet:[0,\infty)\to[-\infty,\infty],\qquad
   \Lambda^{(t)}_s:=-\frac1t\log\E_{\Prob_1}\lB W_t^{-s}\rB
   \qquad (t\in\N) ,
\]
and its optimised form
\[
   \Lambda^{(t)}:=\sup_{s\in[0,\infty)}\Lambda^{(t)}_s\ \in[0,\infty] .
\]
\end{definition}

\begin{lemma}[The two exponents agree on i.i.d.\ products]\label{lem:iidexp}
In the i.i.d.\ model, if $(W_t)_{t\in\N_0}$ is the i.i.d.\ product generated by a
single $E\in\Ev$, then
$\Lambda^{(t)}_s=\Lambda_s(E)$ for every $s\in[0,\infty)$ and every $t\in\N$.
\end{lemma}

\begin{proof}
Under $\Prob_1=P_1^{\otimes\N}$ the factors $E(X_1),\dots,E(X_t)$ are
independent with common law that of $E$ under $P_1$. Moreover
$\E_{P_0}[E]\le1$ forces $P_0(E=\infty)=0$, hence $P_1(E=\infty)=0$ by
\Cref{ass:standing}, so $W_t^{-s}=\prod_{n\le t}E(X_n)^{-s}$ holds
$\Prob_1$-a.s.\ (the identity can fail only where one factor is $0$ and another
$\infty$). Tonelli then gives
$\E_{\Prob_1}[W_t^{-s}]=(\E_{P_1}[E^{-s}])^{t}$ in $(0,\infty]$; take
$-\frac1t\log$.
\end{proof}

\begin{lemma}[\epower{} is well defined and bounded]\label{lem:epower-welldef}
Under \Cref{ass:standing}, in the i.i.d.\ model, for every $E\in\Ev$,
\[
   \E_{P_1}\lB(\log E)_+\rB\ \le\ \KL(P_1\|P_0)+1\ <\ \infty ,
   \qquad\text{and}\qquad
   \E_{P_1}[\log E]\ \le\ \KL(P_1\|P_0) ,
\]
with equality in the second bound if and only if $E=R$ $P_0$-a.s.
\end{lemma}

\begin{proof}
Young's inequality $xy\le x\log x-x+e^{y}$ holds for all $x\ge0$, $y\in\R$ (fix
$x$ and minimise the right-hand side over $y$). Since $\E_{P_0}[E]\le1$ gives
$P_0(E=\infty)=0$, we may apply it pointwise on the $P_0$-full set
$\{E<\infty\}$, with $x=R$ and $y=(\log E)_+\in\R$, and integrate against $P_0$:
\[
   \E_{P_0}\lB R\,(\log E)_+\rB
   \ \le\ \E_{P_0}\lB R\log R-R\rB+\E_{P_0}\lB e^{(\log E)_+}\rB .
\]
The left side is $\E_{P_1}[(\log E)_+]$. On the right,
$\E_{P_0}[R\log R-R]=\KL(P_1\|P_0)-1$, and $e^{(\log E)_+}=\max(1,E)\le1+E$, so
$\E_{P_0}[e^{(\log E)_+}]\le1+\E_{P_0}[E]\le2$. Hence
$\E_{P_1}[(\log E)_+]\le(\KL(P_1\|P_0)-1)+2$, which is the first claim.
For the second, $\log$ is concave, so by Jensen's inequality applied to the
probability measure $P_1$ and the map $E/R$ (defined $P_1$-a.s.\ since $R>0$
$P_1$-a.s.),
\[
   \E_{P_1}[\log E]-\KL(P_1\|P_0)=\E_{P_1}\lB\log\tfrac ER\rB
   \le\log\E_{P_1}\lB\tfrac ER\rB=\log\E_{P_0}\lB E\mathbf 1_{\{R>0\}}\rB\le0 .
\]
For the equality case, equality in Jensen's inequality for the strictly concave
$\log$ forces $E/R$ to be $P_1$-a.s.\ equal to a constant $c$, and equality in
the last step forces $\E_{P_0}[E\mathbf 1_{\{R>0\}}]=1$. On $\{R>0\}$ the
measures $P_0$ and $P_1$ are equivalent --- if $A\subseteq\{R>0\}$ has
$P_1(A)=\int_AR\,dP_0=0$ then $P_0(A)=0$ --- so $E=cR$ $P_0$-a.s.\ on
$\{R>0\}$, whence $1=\E_{P_0}[E\mathbf 1_{\{R>0\}}]=c\,\E_{P_0}[R]=c$. Finally
$\E_{P_0}[E]\le1=\E_{P_0}[E\mathbf 1_{\{R>0\}}]$ gives
$\E_{P_0}[E\mathbf 1_{\{R=0\}}]=0$, so $E=0=R$ $P_0$-a.s.\ there. Conversely
$E=R$ gives equality throughout.
\end{proof}

\begin{remark}[What the two scalars are for]\label{rem:two-scalars}
$\E_{P_1}[\log E]$ is a statement about the \emph{mean} of $\log E$;
$\Lambda(E)$ about its \emph{lower tail}. \Cref{sec:nfl} shows the first
bounds nothing on its own, and \Cref{sec:ceiling} determines the exact range of
the second.
\end{remark}

\subsection{Why \texorpdfstring{$\Lambda$}{Lambda} is the exponent}
\label{sec:whyexponent}

We record the elementary reason $\Lambda^{(t)}$ deserves the name. The first
statement holds for every \etm{}; the second is a matching lower bound on
$\gamma_t(\alpha)$, and there independence is used.

\begin{lemma}[Chernoff bound; any \etm{}]\label{lem:chernoff}
For every betting sequence, every $\alpha\in(0,1)$, every $s\in[0,\infty)$ and
every $t\in\N$,
\begin{equation}\label{eq:chernoff}
   \bar\gamma_t(\alpha)\ \le\ \gamma_t(\alpha)
   \ \le\ \exp\lp s\ell-t\,\Lambda^{(t)}_s\rp .
\end{equation}
\end{lemma}

\begin{proof}
Fix $s\in[0,\infty)$. Since $x\mapsto x^{-s}$ is nonincreasing on $[0,\infty]$
(with the conventions of \Cref{sec:setup}),
\[
   \lC W_t\le\tfrac1\alpha\rC\subseteq\lC W_t^{-s}\ge\alpha^{s}\rC ,
\]
so Markov's inequality applied to the nonnegative variable $W_t^{-s}$ gives
$\gamma_t(\alpha)\le\alpha^{-s}\E_{\Prob_1}[W_t^{-s}]
=\alpha^{-s}e^{-t\Lambda^{(t)}_s}$, and $\alpha^{-s}=e^{s\ell}$.
\Cref{lem:order} gives the first inequality.
\end{proof}

\begin{proposition}[The Chernoff--Stein exponent is the type-II exponent of an
i.i.d.\ product]
\label{prp:exponent}
Work in the i.i.d.\ model. Let $E\in\Ev$ and let $(W_t)_{t\in\N_0}$ be the
i.i.d.\ product generated by $E$. Then
\eqref{eq:chernoff} reads
$\bar\gamma_t(\alpha)\le\gamma_t(\alpha)\le\exp(s\ell-t\Lambda_s(E))$, so that
$\liminf_{t\to\infty}-\frac1t\log\gamma_t(\alpha)\ge\Lambda(E)$. If in
addition $\E_{P_1}|\log E|<\infty$, $\E_{P_1}[\log E]>0$, and $0$ lies in the interior
of $\lC x\in\R\st I(x)<\infty\rC$, where $I:\R\to[0,\infty]$ is the Cram\'er rate
function of $\log E$ under $P_1$ (\Cref{thm:app-cramer}, from
\cite{Cra38}, \cite[Thm.~2.2.3]{DZ10}), then
\begin{equation}\label{eq:cramerlimit}
   \lim_{t\to\infty}-\frac1t\log\gamma_t(\alpha)\ =\ \Lambda(E) .
\end{equation}
Statement \eqref{eq:cramerlimit} concerns the fixed-horizon error. For the
sequential error only the inequality
$\liminf_{t}-\frac1t\log\bar\gamma_t(\alpha)\ \ge\ \Lambda(E)$
is claimed here.
\end{proposition}

\begin{proof}
The displayed bound is \Cref{lem:chernoff} combined with \Cref{lem:iidexp};
taking $-\frac1t\log$ and then the supremum over $s$ gives the $\liminf$
statement.

For \eqref{eq:cramerlimit}, write $Y_n:=\log E(X_n)$ for $n\in\N$, and put
$S_t:=\sum_{n\le t}Y_n$, $\bar Y_t:=t^{-1}S_t$ for $t\in\N$, so that
$\gamma_t(\alpha)=\Prob_1(S_t\le \ell)=\Prob_1(\bar Y_t\le \ell/t)$; let $I$ be the
Cram\'er rate function of $Y_1$ under $P_1$. We record three properties of $I$:
it is convex, it vanishes at $\E_{P_1}[Y_1]=\E_{P_1}[\log E]>0$, hence it is
nonincreasing on $(-\infty,\E_{P_1}[\log E]]$; and it is finite in a neighbourhood of
$0$, hence continuous at $0$, by the interiority hypothesis.

\emph{Upper bound on $\gamma_t$.} Fix $t_1\in\N$ with $\ell/t_1<\E_{P_1}[\log E]$. For
$t\ge t_1$ we have $\ell/t\le \ell/t_1$, so
$\{\bar Y_t\le \ell/t\}\subseteq\{\bar Y_t\in F\}$ with $F:=(-\infty,\ell/t_1]$
closed. Since $I$ is nonincreasing on $(-\infty,\E_{P_1}[\log E]]\supseteq F$,
$\inf_{x\in F}I(x)=I(\ell/t_1)$, so the large-deviation upper bound of
Cram\'er's theorem \cite{Cra38}, \cite[Thm.~2.2.3]{DZ10}
(\Cref{thm:app-cramer}) gives
\[
   \liminf_{t\to\infty}-\tfrac1t\log\gamma_t(\alpha)\ \ge\ I\lp \ell/t_1\rp .
\]

\emph{Lower bound on $\gamma_t$.} Since $\ell/t>0$, the open set
$G:=(-\infty,0)$ satisfies $\{\bar Y_t\in G\}\subseteq\{\bar Y_t\le \ell/t\}$, so
the large-deviation lower bound of \Cref{thm:app-cramer} gives
\[
   \limsup_{t\to\infty}-\tfrac1t\log\gamma_t(\alpha)\ \le\ \inf_{x<0}I(x)
   \ =\ \lim_{x\uparrow0}I(x)\ =\ I(0) ,
\]
the first equality by monotonicity of $I$ on $(-\infty,\E_{P_1}[\log E]]$ and the
second by continuity at $0$.

\emph{Conclusion.} Combining and then letting $t_1\to\infty$, so that
$I(\ell/t_1)\to I(0)$ by continuity, squeezes both outer quantities to $I(0)$.
Finally $I(0)=\sup_{\theta\in\R}\{-\log\E_{P_1}[e^{\theta\log E}]\}
=\sup_{s\ge0}\Lambda_s(E)=\Lambda(E)$: the substitution $\theta=-s$ turns
$\theta<0$ into $s>0$ and $-\log\E_{P_1}[e^{\theta\log E}]$ into
$\Lambda_s(E)$, while for $\theta>0$ Jensen gives
$\log\E_{P_1}[e^{\theta\log E}]\ge\theta\E_{P_1}[\log E]>0$, so positive $\theta$
contribute a value below $0=\Lambda_0(E)$.
\end{proof}

\subsection{Notation index}
\label{sec:notation}

\begin{center}
\footnotesize
\setlength{\tabcolsep}{4pt}
\begin{tabular}{@{}lll@{}}
\toprule
Symbol & Meaning & Defined in\\
\midrule
$(\Omega,\Fc,(\Fc_t)_{t\in\N_0})$ & filtered space & \Cref{def:setting}\\
$\Prob_0,\Prob_1$ & null and alternative & \Cref{def:setting}\\
$(\Xc,\Bs{\Xc})$; $X_n$; $P_0,P_1$ & observation space, coordinates,
   one-observation laws & \Cref{def:iid}\\
$L_t$; $R_t$ & likelihood-ratio martingale; one-step ratio
   & \Cref{def:ratios}\\
$R=\frac{dP_1}{dP_0}:\Xc\to[0,\infty)$ & likelihood ratio, i.i.d.\ model
   & \Cref{rem:one-obs}\\
$Z_{\beta,t}$; $\lambda_t(\beta)$ & conditional Hellinger integral; profile
   & \Cref{def:condZ}\\
$\KLrev_n$; $\bar\lambda_n(0^{+})$ & conditional relative entropy of step $n$;
   its essential supremum & \Cref{cor:ceiling}\\
$\KL(Q\|P)$; $\Ren_a(Q\|P)$ & KL and R\'enyi divergence
   & \Cref{ass:standing}; \Cref{def:renyi}\\
$\Ev$ & the set of \evars{} for $P_0$ & \Cref{def:evar}\\
$E$; $E_t$ & an \evar{}; the bet at time $t$
   & \Cref{def:evar}, \Cref{def:cond-evar}\\
$W_t$ & wealth after $t$ bets & \Cref{def:wealth}\\
$\alpha$; $\ell=\log\frac1\alpha$ & level; log-threshold & \Cref{def:reject}\\
$\tau_\alpha$ & rejection time & \Cref{def:reject}\\
$\gamma_t(\alpha)$; $\bar\gamma_t(\alpha)$
   & type-II error, fixed-horizon; sequential & \Cref{def:typeII}\\
$\E_{P_1}[\log E]$ & \epower{} (no separate symbol) & \Cref{def:epower}\\
$s$ & Chernoff parameter (the tilt of the bound) & \Cref{def:exponent}\\
$\Lambda_s(E)$; $\Lambda(E)$ & tilted exponent at $s$; Chernoff--Stein exponent
   of an \evar{} & \Cref{def:exponent}\\
$\Lambda^{(t)}_s$; $\Lambda^{(t)}$ & $t$-step exponent of a \etm{}
   & \Cref{def:exponent}\\
$\beta$; $\beta_{\max}$ & flattening parameter, $\beta=\frac1{1+s}$ in
   \Cref{sec:ceiling}, in $\Bc$ in \Cref{sec:horizon}
   & \Cref{def:Zbeta}, \Cref{def:flattened}\\
$Z_\beta=\E_{P_0}[R^{\beta}]$ & Hellinger integral & \Cref{def:Zbeta}\\
$\lambda(\beta)=-\frac1\beta\log Z_\beta$ & flattening profile, $=\Ren_{1-\beta}(P_0\|P_1)$;
   bounds $\Lambda$
   & \Cref{def:Zbeta}\\
$R_\beta=R^{\beta}/Z_\beta$ & flattened likelihood ratio, $R_1=R$
   & \Cref{def:flattened}\\
$\mathcal W$; $\mathcal W_{\mathrm{iid}}$ & all \etm s; the i.i.d.\ ones
   & \Cref{thm:ceiling}\\
$c_\beta$ & \epower{} $\E_{P_1}[\log R_\beta]$ of the flattened design
   & \Cref{lem:family}\\
$G_t$; $B_t(\beta)$ & log-likelihood-ratio sum; threshold & \Cref{def:GB}\\
$g(\beta)=\beta\kappa'(\beta)-\kappa(\beta)$ & horizon-optimality function
   & \Cref{thm:betastar}\\
$\beta^{\ast}(t)$; $V_0$ & horizon-optimal tilt; $\Var_{P_0}(\log R)$
   & \Cref{thm:betastar}\\
$C(P_0,P_1)$ & Chernoff information; $=\Lambda(R)$ when $\KL(P_0\|P_1)<\infty$
   & \Cref{cor:chernoffinfo}\\
$\gamma^{\star}_t(\alpha)$ & least type-II error at level $\alpha$
   & \Cref{thm:app-stein}\\
\bottomrule
\end{tabular}
\end{center}

\noindent
A bar denotes the sequential, running-maximum version of a quantity
($\bar\gamma$); a star denotes an optimiser ($s^{\ast},\beta^{\ast}$). The
symbols $\eta,u,\tilde u,M,q,\tilde q,\Phi$ belong to \Cref{con:nfl} alone, and
$\rho,\varphi,\varkappa,c_1,c_2,g_n,h_\beta,\theta,\pi,\sigma,b_t,r,r',\delta,
\epsilon,x,y,U_t,C_t,I$ are local to the proof, remark or quoted theorem in
which they appear. The letter $c$
denotes an \epower{} level throughout (\Cref{con:nfl}, \Cref{lem:floor}, and
$c_\beta$ above), and $\kappa(\beta):=\log Z_\beta$ is used from
\Cref{lem:family} onwards. In \Cref{ex:gauss} only, $D:=\Delta^{2}/2$; it is
unrelated to the upright $\Ren$ of \Cref{def:renyi}. The
$\sigma$-algebra of a space $S$ is written $\Bs{S}$; in the i.i.d.\ model
$\Fc_\infty=\Bs{\Omega}=\Bs{\Xc}^{\otimes\N}$, while in general
\Cref{def:setting} asks only $\Fc_\infty\subseteq\Fc$.

\section{\texorpdfstring{\epower{}}{e-power} alone does not bound the type-II error}
\label{sec:nfl}

\subsection{The heuristic, and the exact point at which it fails}

\begin{remark}[The heuristic]\label{rem:heuristic}
Let $(E_t)_{t\in\N}$ be a betting sequence with conditional \epower{} bounded below by
$c>0$, i.e.
\begin{equation}\label{eq:cond-epower}
   \E_{\Prob_1}\lB\log E_t\mid\Fc_{t-1}\rB\ \ge\ c
   \qquad\Prob_1\text{-a.s., for every }t\in\N .
\end{equation}
Then $\log W_t=\sum_{n\le t}\log E_n$ has $\E_{\Prob_1}[\log W_t]\ge tc\to\infty$,
which is the design principle behind Kelly betting \cite{Kel56,Bre61} and
growth-rate-optimal \evars{} \cite{GHK24,LRR25}. The inference
``therefore $W_t$ exceeds $\frac1\alpha$ with high probability'' is what fails.
\end{remark}

The failure is not subtle, and it can be located precisely. Condition
\eqref{eq:cond-epower} constrains a \emph{conditional mean}. The quantity to be
controlled is a \emph{probability},
\[
   \bar\gamma_t(\alpha)=\Prob_1\lp\max_{n\le t}\log W_n\le \ell\rp ,
\]
and a lower bound on a mean permits arbitrarily heavy lower tails. Concretely:
$\log W_t$ may be very large in expectation because it is enormous on an event
of small probability, while on the complement it is so negative that the
partial maxima never reach $\ell$. \Cref{con:nfl} realises exactly this, with the
crash placed at the first step.

\subsection{The construction}

The first bet must differ from the later ones. A betting sequence may of course
use a different function at each step --- \Cref{def:cond-evar} asks only for
$\Fc_n$-measurability --- but the particular constants below were chosen for two
different one-observation laws, and it is simplest to state the construction
directly on the path space, with $\Prob_0,\Prob_1$ products of non-identical
factors. \Cref{def:setting} allows this: the i.i.d.\ model of \Cref{def:iid} is
assumed only where a statement names it, and no statement in this section
outside \Cref{rem:quantifier}, \Cref{prp:powerone} and \Cref{rem:cramer} does.

\begin{construction}[A betting sequence with prescribed \epower{} and no power]
\label{con:nfl}
Let $c\in(0,\infty)$, $\alpha\in(0,1)$, $\eta\in(0,1)$ and $t\in\N$ be given,
and put $\ell:=\log\frac1\alpha$. Define constants
\begin{align}
   \tilde u&:=2c+1 , & M&:=(t-1)\tilde u+\ell , &
   u&:=\frac{c+(1-\eta)M}{\eta} ,
   \label{eq:nfl-constants}\\
   q&:=\frac{1-e^{-M}}{e^{u}-e^{-M}} , &
   \tilde q&:=\frac{1-e^{-1}}{e^{\tilde u}-e^{-1}} . & &
   \label{eq:nfl-probs}
\end{align}
Take $\Xc:=\{0,1\}$ with $\Bs{\Xc}:=2^{\Xc}$, $\Omega=\Xc^{\N}$, and let
$\Prob_0,\Prob_1$ be the laws on $(\Omega,\Bs{\Omega})$ under which the
coordinates $X_1,X_2,\dots$ are independent with
\[
\begin{array}{lll}
   \Prob_1(X_1=1)=\eta , &\qquad
   \Prob_0(X_1=1)=q , &\\[2pt]
   \Prob_1(X_n=1)=\tfrac12 , &\qquad
   \Prob_0(X_n=1)=\tilde q , &\qquad n\ge2 .
\end{array}
\]
Define the betting sequence $E_n:\Omega\to(0,\infty)$, $n\in\N$, by
\begin{equation}\label{eq:nfl-bets}
   \log E_1:=\begin{cases} u, & X_1=1,\\ -M, & X_1=0,\end{cases}
   \qquad
   \log E_n:=\begin{cases} \tilde u, & X_n=1,\\ -1, & X_n=0,\end{cases}
   \quad n\ge2 .
\end{equation}
\end{construction}

The constants are admissible, by \Cref{lem:nfl-admissible} below. The
construction is a single catastrophic first bet, made deep enough that the
remaining $t-1$ bets cannot climb back to the threshold, together with $t-1$
ordinary bets of the same \epower{}. We verify the four required properties in
turn.

\begin{lemma}[Step 1: the constants are admissible]\label{lem:nfl-admissible}
$\tilde u>1>0$, $M\ge\ell>0$, $u>0$, and $q,\tilde q\in(0,1)$.
\end{lemma}

\begin{proof}
$\tilde u=2c+1>1$ since $c>0$; $M=(t-1)\tilde u+\ell\ge \ell>0$ since $t\ge1$;
$u=(c+(1-\eta)M)/\eta>0$ since $c>0$, $M>0$, $\eta\in(0,1)$. For $q$: the
numerator $1-e^{-M}$ lies in $(0,1)$ because $M>0$, and the denominator
satisfies $e^{u}-e^{-M}>1-e^{-M}>0$ because $u>0$; hence $q\in(0,1)$. The same
argument with $(\tilde u,1)$ in place of $(u,M)$ gives $\tilde q\in(0,1)$.
\end{proof}

\begin{lemma}[Step 2: each $E_n$ is a conditional \evar{}]\label{lem:nfl-evar}
For every $n\in\N$, $E_n$ is $\Fc_n$-measurable and
$\E_{\Prob_0}[E_n\mid\Fc_{n-1}]=1$ $\Prob_0$-a.s.
\end{lemma}

\begin{proof}
Measurability is clear from \eqref{eq:nfl-bets}. Under $\Prob_0$ the
coordinates are independent, so $\E_{\Prob_0}[E_n\mid\Fc_{n-1}]
=\E_{\Prob_0}[E_n]$. For $n=1$,
\[
   \E_{\Prob_0}[E_1]=q\,e^{u}+(1-q)e^{-M} .
\]
Substituting $q$ from \eqref{eq:nfl-probs},
\[
   q\lp e^{u}-e^{-M}\rp+e^{-M}
   =\frac{1-e^{-M}}{e^{u}-e^{-M}}\lp e^{u}-e^{-M}\rp+e^{-M}
   =\lp1-e^{-M}\rp+e^{-M}=1 .
\]
For $n\ge2$ the identical computation with $(\tilde u,1,\tilde q)$ gives
$\E_{\Prob_0}[E_n]=\tilde q e^{\tilde u}+(1-\tilde q)e^{-1}=1$.
\end{proof}

\begin{lemma}[Step 3: the conditional \epower{} equals $c$ at every step]
\label{lem:nfl-epower}
For every $n\in\N$, $\E_{\Prob_1}[\log E_n\mid\Fc_{n-1}]=c$ $\Prob_1$-a.s.
\end{lemma}

\begin{proof}
Again by independence the conditional expectation is the unconditional one. For
$n=1$, using $u=(c+(1-\eta)M)/\eta$ from \eqref{eq:nfl-constants},
\[
   \E_{\Prob_1}[\log E_1]=\eta u-(1-\eta)M
   =\eta\cdot\frac{c+(1-\eta)M}{\eta}-(1-\eta)M=c .
\]
For $n\ge2$, using $\tilde u=2c+1$,
\[
   \E_{\Prob_1}[\log E_n]=\tfrac12\tilde u-\tfrac12
   =\tfrac12(2c+1)-\tfrac12=c . \qedhere
\]
\end{proof}

\begin{lemma}[Step 4: after a first-step crash the threshold is unreachable]
\label{lem:nfl-crash}
On the event $\{X_1=0\}$ one has $\max_{0\le n\le t}\log W_n\le \ell$, hence
$\tau_\alpha>t$.
\end{lemma}

\begin{proof}
Work on $\{X_1=0\}$. First, $\log W_0=0\le \ell$ since $\ell>0$. Second,
$\log W_1=\log E_1=-M<0\le \ell$. Third, for $2\le n\le t$, each of the increments
$\log E_2,\dots,\log E_n$ is at most $\tilde u$ by \eqref{eq:nfl-bets}, so
\[
   \log W_n=-M+\sum_{m=2}^{n}\log E_m\ \le\ -M+(n-1)\tilde u
   \ \le\ -M+(t-1)\tilde u\ =\ -\ell\ \le\ \ell ,
\]
the last two steps by $n\le t$ and by $M=(t-1)\tilde u+\ell$ from
\eqref{eq:nfl-constants}. Hence $\max_{0\le n\le t}\log W_n\le \ell$, i.e.\
$W_n\le\frac1\alpha$ for all $n\le t$, i.e.\ $\tau_\alpha>t$.
\end{proof}

\begin{theorem}[\epower{} alone implies no type-II bound]\label{thm:nfl}
Let $c\in(0,\infty)$, $\alpha\in(0,1)$, $\eta\in(0,1)$ and $t\in\N$. The pair
$(\Prob_0,\Prob_1)$ and the betting sequence of \Cref{con:nfl} --- both of which
depend on $(c,\alpha,\eta,t)$ --- satisfy
\begin{enumerate}[label=(\roman*),leftmargin=2.4em]
\item $E_n$ is a conditional \evar{} for $\Prob_0$ at time $n$, with
      $\E_{\Prob_0}[E_n\mid\Fc_{n-1}]=1$, for every $n\in\N$;
\item $\E_{\Prob_1}[\log E_n\mid\Fc_{n-1}]=c$ $\Prob_1$-a.s.\ for every
      $n\in\N$;
\item $\bar\gamma_t(\alpha)\ \ge\ 1-\eta$, and therefore also
      $\gamma_t(\alpha)\ge1-\eta$.
\end{enumerate}
Consequently there is no function $\Phi:(0,\infty)\times(0,1)\times\N\to[0,1)$
such that
\begin{equation}\label{eq:nophi}
   \bar\gamma_t(\alpha)\ \le\ \Phi(c,\alpha,t)
   \qquad\text{for every pair of hypotheses and every betting sequence
   satisfying \eqref{eq:cond-epower}:}
\end{equation}
no bound on the type-II error is a function of the \epower{} level $c$, the
significance level $\alpha$ and the horizon $t$ alone, \emph{uniformly over
testing problems}.
\end{theorem}

\begin{remark}[What \Cref{thm:nfl} does and does not quantify over]
\label{rem:quantifier}
The pair $(\Prob_0,\Prob_1)$ of \Cref{con:nfl} is built from
$(c,\alpha,\eta,t)$: the alternative puts mass $\eta$ on $\{X_1=1\}$, and the
null probability $q$ of \eqref{eq:nfl-probs} depends on $\eta$ through $u$. Driving $\eta$ to $0$ against a candidate $\Phi$ therefore varies the
testing problem as well as the design, and \eqref{eq:nophi} is a statement about
the class of \emph{all} testing problems.

This is the right quantifier for a sample-size rule, which must be written down
before the data and cannot depend on the unknown alternative. It is not the
strongest conceivable statement, and the stronger one is false: a
\emph{model-dependent} bound $\Phi(c,\alpha,t;P_0,P_1)$ is not excluded by
\Cref{thm:nfl}, and in some models it exists. Work in the i.i.d.\ model
(\Cref{def:iid}) and suppose $R\le R_{\max}$ $P_0$-a.s.\ and $c>\ell$. Since $\Fc_0$ is trivial, $E_1\in\Ev$, and with
$A:=\{E_1>\frac1\alpha\}$ and $\pi:=P_1(A)$ we have $\log E_1\le\ell$ on
$A^{c}$, while Jensen and $\E_{P_1}[E_1]=\E_{P_0}[RE_1]\le R_{\max}$ give
$\E_{P_1}[\log E_1\mathbf 1_A]\le\pi\log(R_{\max}/\pi)$. Hence
\[
   c\ \le\ \pi\log\frac{R_{\max}}{\pi}+(1-\pi)\ell ,
\]
whose right-hand side tends to $\ell$ as $\pi\downarrow0$; so $c>\ell$ forces
$\pi\ge\pi_0(c,\alpha,R_{\max})>0$, and since
$\{\tau_\alpha>t\}\subseteq\{\tau_\alpha>1\}$,
\[
   \bar\gamma_t(\alpha)\ \le\ \bar\gamma_1(\alpha)=1-\pi\ \le\ 1-\pi_0\ <\ 1
   \qquad\text{for every }t .
\]
The class is nonempty whenever $c\le\KL(P_1\|P_0)$ (take $E_n=R(X_n)$), so the
two conditions $\ell<c\le\KL(P_1\|P_0)$ are jointly satisfiable --- for instance
$P_0=\mathrm{Bern}(10^{-3})$, $P_1=\mathrm{Bern}(0.9)$, $\alpha=0.05$, where
$R_{\max}=900$, $\ell=2.996$ and $\KL(P_1\|P_0)=5.892$. The regime is a
degenerate one --- it requires the \epower{} of a single observation to exceed
$\log\frac1\alpha$, i.e.\ one expects to reject before seeing two data points ---
but it shows that \Cref{thm:nfl} is a statement about uniformity over models and
not only over designs.
\end{remark}

\begin{proof}
(i) is \Cref{lem:nfl-evar}, (ii) is \Cref{lem:nfl-epower}. For (iii),
\Cref{lem:nfl-crash} gives $\{X_1=0\}\subseteq\{\tau_\alpha>t\}$, and
$\Prob_1(X_1=0)=1-\eta$ by construction; \Cref{lem:order} gives the second
inequality. For the final claim, suppose such a $\Phi$ existed and fix
$c,\alpha,t$. Choosing $\eta:=\frac12\lp1-\Phi(c,\alpha,t)\rp\in(0,1)$ yields a
sequence with conditional \epower{} exactly $c$ and
$\bar\gamma_t(\alpha)\ge1-\eta=\frac12\lp1+\Phi(c,\alpha,t)\rp>\Phi(c,\alpha,t)$,
a contradiction.
\end{proof}

\subsection{A worked instance}

\begin{example}\label{ex:nfl}
Take $c=0.1$, $\alpha=0.05$, $\eta=\frac12$, $t=10$. Then
$\ell=\log20=2.995732$, and \eqref{eq:nfl-constants}--\eqref{eq:nfl-probs} give
\[
   \tilde u=1.2,\qquad
   M=9\cdot1.2+2.995732=13.795732,\qquad
   u=\frac{0.1+0.5\cdot13.795732}{0.5}=13.995732 ,
\]
\[
   q=\frac{1-e^{-13.795732}}{e^{13.995732}-e^{-13.795732}}=8.350842\cdot10^{-7},
   \qquad
   \tilde q=\frac{1-e^{-1}}{e^{1.2}-e^{-1}}=0.2141158 .
\]
In words: at the first step the bet pays $e^{13.996}\approx1.2\cdot10^{6}$ on an
event of null probability $8.4\cdot10^{-7}$ and alternative probability
$\frac12$, and otherwise loses all but $e^{-13.796}\approx10^{-6}$ of the
wealth. Every step has \epower{} exactly $0.1$. The four requirements read
\[
   \E_{\Prob_0}[E_1]=1,\quad \E_{\Prob_0}[E_n]=1,\quad
   \E_{\Prob_1}[\log E_n]=0.1,\quad
   -M+(t-1)\tilde u=-2.995732=-\ell ,
\]
and enumeration of all $2^{10}$ paths gives
\[
   \bar\gamma_{10}(0.05)\ =\ 0.500000\ =\ 1-\eta
\]
exactly: on $\{X_1=1\}$ the test rejects already at $t=1$, and on $\{X_1=0\}$ it
never rejects by $t=10$. So a test whose \epower{} promises
$\E_{\Prob_1}[\log W_{10}]=1.0$ nats of evidence has power exactly $\frac12$ at
this horizon --- and by \Cref{thm:nfl} the same construction drives the power to
$0$ for any fixed $c$, $\alpha$ and $t$.
\end{example}

\subsection{What the theorem does not say}

\begin{proposition}[Asymptotic power one]\label{prp:powerone}
Work in the i.i.d.\ model. Let $E\in\Ev$ with $\E_{P_1}[\log E]>0$, and let
$(W_t)_{t\in\N_0}$
be the i.i.d.\ product generated by $E$. Then $\tau_\alpha<\infty$
$\Prob_1$-a.s.\ for every $\alpha\in(0,1)$.
\end{proposition}

\begin{proof}
By \Cref{lem:epower-welldef} $\E_{P_1}[(\log E)_+]<\infty$, so
$\E_{P_1}[\log E]>0$ already forces $\log E\in L^{1}(P_1)$. By the strong law of
large numbers, $t^{-1}\log W_t\to\E_{P_1}[\log E]>0$
$\Prob_1$-a.s., so $\log W_t\to+\infty$ $\Prob_1$-a.s., and in particular
$\log W_t>\ell$ for some $t$.
\end{proof}

\begin{remark}[What the word ``alone'' is doing]\label{rem:alone}
\Cref{thm:nfl} is a statement about a \emph{class}: the class of betting
sequences whose conditional \epower{} is at least $c$ admits no type-II bound
depending only on $(c,\alpha,t)$. It is not the statement that \epower{}
carries no information. A positive \epower{} does control the asymptotics
(\Cref{prp:powerone}: the test rejects eventually, almost surely), and it does
control the \emph{mean} of $\log W_t$ by construction. The two results are not
in conflict, and the difference between them is the difference between a
pointwise limit and a uniform bound: \Cref{prp:powerone} fixes one \evar{} and
lets $t\to\infty$, while \Cref{thm:nfl} fixes $t$ and varies the sequence, which
it is free to do because the crash depth $M=(t-1)\tilde u+\ell$ may depend on
$t$. What fails is not consistency but \emph{uniformity over the class defined
by \epower{} alone} --- and uniformity is what a sample-size calculation
requires. The missing ingredient is any control of the lower tail, and
\Cref{rem:whatworks} and \Cref{sec:ceiling} supply it. Note finally that the
hypothesis of \Cref{thm:nfl} is the favourable one, \epower{} bounded
\emph{below} by a positive constant: were it negative there would be nothing to
discuss, not even \Cref{prp:powerone}, so the theorem is not an artefact of a
degenerate case.
\end{remark}

\begin{remark}[What does control the type-II error]\label{rem:whatworks}
Inspection of \Cref{con:nfl} shows the free parameter is the depth $M$ of the
crash, which \eqref{eq:cond-epower} does not constrain. Any hypothesis bounding
the lower tail of $\log E_t$ restores a bound. Suppose, for some
$s_0>0$ and $c_0>0$,
\[
   \E_{\Prob_1}\lB E_t^{-s_0}\mid\Fc_{t-1}\rB\ \le\ e^{-c_0}
   \qquad\Prob_1\text{-a.s., for every }t\in\N .
\]
The tower-property induction of \Cref{prp:upper-seq}\ref{it:agg-exact}, run with $\le$ in
place of $\ge$, gives $\E_{\Prob_1}[W_t^{-s_0}]\le e^{-tc_0}$, i.e.\
$\Lambda^{(t)}_{s_0}\ge c_0$, and \Cref{lem:chernoff} then gives
$\gamma_t(\alpha)\le\exp(s_0\ell-tc_0)\to0$. The exponent $\Lambda^{(t)}$ of
\Cref{def:exponent} is the sharpest such summary, and \Cref{sec:ceiling}
determines exactly how large it can be.
\end{remark}

\begin{remark}[Which hypothesis buys which rate]\label{rem:cramer}
For an i.i.d.\ product the exponential case is settled exactly, and in both
directions, by \Cref{prp:exponent}: under its hypotheses
$\lim_{t}-\frac1t\log\gamma_t(\alpha)=\Lambda(E)$, so the fixed-horizon error decays
exponentially if and only if $\Lambda(E)>0$, i.e.\ if and only if
$\E_{P_1}[E^{-s}]<1$ for some $s>0$. Given positive \epower{} this is in turn
equivalent to \emph{Cram\'er's moment condition},
$\E_{P_1}[E^{-s_0}]<\infty$ for some
$s_0>0$: that condition makes $\varphi(s):=\E_{P_1}[E^{-s}]$ finite on
$[0,s_0]$, by Lyapunov's inequality $\varphi(s)\le\varphi(s_0)^{s/s_0}$, with
$\varphi(0)=1$; and the difference quotients
$\frac{\varphi(s)-1}{s}=\E_{P_1}[\frac{e^{-s\log E}-1}{s}]$ decrease as
$s\downarrow0$ and are dominated above by their value at $s_0$, so monotone
convergence gives the right derivative
$\varphi'(0^{+})=-\E_{P_1}[\log E]<0$, and $\varphi(s)<1$ for small $s>0$. It is a hypothesis on the lower tail: by Markov's inequality it
gives $\Prob_1(E\le\epsilon)\le\epsilon^{s_0}\E_{P_1}[E^{-s_0}]$, a
polynomially small crash probability, and conversely
$\Prob_1(E\le\epsilon)\le K\epsilon^{a}$ gives finiteness for every $s<a$. The
likelihood ratio satisfies it for free with $s_0=1$, since
$\E_{P_1}[R^{-1}]=P_0(R>0)\le1$; it constrains the \emph{approximation} to $R$
that one bets with, not the testing problem.

What Cram\'er's condition is not is necessary for a bound. If
$\E_{P_1}[\log E]=:m>0$ and $\sigma^{2}:=\Var_{P_1}(\log E)<\infty$ --- which
permits a lower tail with no exponential moment whatsoever --- then Chebyshev's
inequality applied to $\log W_t$ gives, for every $t>\ell/m$,
\[
   \gamma_t(\alpha)\ \le\ \frac{t\,\sigma^{2}}{\lp t\,m-\ell\rp^{2}}
   \ =\ \frac{\sigma^{2}}{m^{2}\,t}\,(1+o(1)) ,
\]
a bound of order $1/t$ which holds \emph{even in the regime where Cram\'er's
moment condition fails and $\Lambda(E)=0$}. So \Cref{thm:nfl} separates
\emph{no} hypothesis on the lower tail from \emph{some} hypothesis, while
Cram\'er's moment condition separates exponential from sub-exponential decay;
under the additional hypothesis $\sigma^{2}<\infty$ the sub-exponential rate is
$O(1/t)$. The two dichotomies are different, and only the second is an
equivalence.

Which \emph{quantitative} bound each hypothesis buys is the subject of the
companion paper \cite{LADDER}: Cram\'er's condition, a crash tail with one of
three further moment conditions, a conditional variance with a wealth floor, a
sub-Gaussian or bounded-tilt lower tail, a bounded log-increment, and i.i.d.\
increments form a partial order of hypotheses, each supplying a majorant for
the conditional cumulant generating function of $\log E_n$ and hence, by one
Legendre transform, an explicit exponent. That paper works with an arbitrary
filtration, asks only that the certified drift and the majorant be predictable,
and needs neither a product structure nor the results proved here; conversely
nothing below uses it.
The two papers meet at exactly one point, its top rung: there the majorant may
be taken to be the truth, the exponent is $\Lambda(E)$, and \Cref{sec:ceiling}
is what bounds $\Lambda(E)$ itself.
\end{remark}

\section{The ceiling: matching upper and lower bounds}
\label{sec:ceiling}

The question of this section is the exact range of the exponent
$\Lambda^{(t)}_s$ of \Cref{def:exponent} over the class of all \etm s. The
engine is a single application of H\"older's inequality, made once per step and
conditionally on the past (\Cref{thm:upper-cond}); \Cref{prp:upper-seq}
aggregates the steps, and \Cref{lem:flat} shows the resulting bound is met
inside the i.i.d.\ subclass. \Cref{thm:ceiling} assembles the three statements.
The upper half needs only the filtered setting of \Cref{def:setting}; the
lower half needs the i.i.d.\ model, and \Cref{rem:whatgeneralises} says why.

\subsection{The R\'enyi quantities}

\begin{definition}[Hellinger integrals and the flattening profile]\label{def:Zbeta}
In the i.i.d.\ model define
\[
   Z:\R\to(0,\infty],\qquad Z_\beta:=\E_{P_0}\lB R^{\beta}\rB ,
\]
and
\[
   \lambda:\R\setminus\{0\}\to[-\infty,\infty],\qquad
   \lambda(\beta):=-\frac1\beta\log Z_\beta .
\]
\end{definition}

\begin{definition}[Flattened likelihood ratios]\label{def:flattened}
In the i.i.d.\ model put
\[
   \Bc\ :=\ \lC\beta>0\st Z_\beta<\infty\rC ,
   \qquad
   \beta_{\max}\ :=\ \sup\Bc\ \in[1,\infty] ,
\]
so that $\Bc$ is an interval with $(0,1]\subseteq\Bc\subseteq(0,\beta_{\max}]$ ---
the first inclusion because $Z_\beta\le1$ on $(0,1]$ by \Cref{lem:Zf}(i) --- and
for $\beta\in\Bc$ define
\[
   R_\beta:\Xc\to[0,\infty),\qquad R_\beta:=\frac{R^{\beta}}{Z_\beta} ,
\]
so that $R_1=R$ and $\E_{P_0}[R_\beta]=1$, whence $R_\beta\in\Ev$ for every
$\beta\in\Bc$.
\end{definition}

\begin{definition}[R\'enyi divergence]\label{def:renyi}
Let $Q,P$ be probability measures on $(\Xc,\Bs{\Xc})$, let $\mu$ be any
$\sigma$-finite measure with $Q,P\ll\mu$, and put $q:=\frac{dQ}{d\mu}$,
$p:=\frac{dP}{d\mu}$. For $a\in(0,1)$ define
\[
   \Ren_a(Q\|P):=\frac{1}{a-1}\log\int q^{a}p^{1-a}\,d\mu\ \in[0,\infty] .
\]
The value does not depend on the choice of $\mu$, and no absolute continuity
between $Q$ and $P$ is required; when $Q\ll P$ it reduces to
$\frac{1}{a-1}\log\E_{P}[(\frac{dQ}{dP})^{a}]$.
\end{definition}

\begin{remark}[Why the dominating-measure form]\label{rem:renyigenerality}
\Cref{ass:standing} assumes $P_1\ll P_0$ but not the converse, so
$\frac{dP_0}{dP_1}$ need not exist; the identity
$\lambda(\beta)=\Ren_{1-\beta}(P_0\|P_1)$ of \Cref{lem:Zf}(iii) would be
meaningless under a definition requiring $P_0\ll P_1$. The model
$P_0=\mathrm{Bern}(\frac12)$, $P_1=\delta_1$ of \Cref{rem:eqcase} satisfies
\Cref{ass:standing} and has $P_0\not\ll P_1$.
\end{remark}

\begin{lemma}[Basic properties of $Z$ and $\lambda$]\label{lem:Zf}
Under \Cref{ass:standing}, in the i.i.d.\ model:
\begin{enumerate}[label=(\roman*),leftmargin=2.4em]
\item $Z_\beta\in(0,1]$ for every $\beta\in(0,1]$ --- so on that range $Z$ and
      $\lambda$ take the values announced in \Cref{def:Zbeta} --- and $Z_1=1$, so
      $\lambda(1)=0$ and $\lambda(\beta)\in[0,\infty)$ for every
      $\beta\in(0,1]$ --- only $\lim_{\beta\downarrow0}\lambda(\beta)$ can be
      infinite;
\item $\lambda$ is nonincreasing on $(0,1]$;
\item for $\beta\in(0,1)$, $\lambda(\beta)=\Ren_{1-\beta}(P_0\|P_1)$, the
      R\'enyi divergence of \Cref{def:renyi} of order $1-\beta$ (skew symmetry,
      \cite[Prop.~2]{vEH14}; see \Cref{thm:app-renyi});
\item if moreover $\log R$ is not $P_0$-a.s.\ constant, then $\lambda$ is strictly
      decreasing on $(0,1]$.
\end{enumerate}
\end{lemma}

\begin{proof}
(i) $Z_1=\E_{P_0}[R]=1$. For $\beta\in(0,1)$ the map $x\mapsto x^{\beta}$ is
concave on $[0,\infty)$, so Jensen gives
$Z_\beta=\E_{P_0}[R^{\beta}]\le(\E_{P_0}[R])^{\beta}=1$; and $Z_\beta>0$ because
$P_0(R>0)>0$. Hence $\log Z_\beta\le0$ and $\lambda(\beta)\ge0$.

(ii) Let $0<\beta'\le\beta\le1$ and put $\rho:=\beta'/\beta\in(0,1]$. Applying
Jensen to the concave map $x\mapsto x^{\rho}$,
\[
   Z_{\beta'}=\E_{P_0}\lB\lp R^{\beta}\rp^{\rho}\rB
   \ \le\ \lp\E_{P_0}\lB R^{\beta}\rB\rp^{\rho}=Z_\beta^{\rho} ,
\]
so $\log Z_{\beta'}\le\rho\log Z_\beta$ and, dividing by $-\beta'=-\rho\beta<0$,
$\lambda(\beta')\ge \lambda(\beta)$.

(iii) With $\mu$, $p_0$, $p_1$ as in \Cref{def:renyi} and $\beta\in(0,1)$,
\[
   \int p_0^{1-\beta}p_1^{\beta}\,d\mu
   =\int_{\{p_0>0\}}p_0\,R^{\beta}\,d\mu=Z_\beta ,
\]
the integrand vanishing where $p_0=0$ because $1-\beta>0$; hence
$\Ren_{1-\beta}(P_0\|P_1)=\frac{1}{-\beta}\log Z_\beta=\lambda(\beta)$. No
absolute continuity of $P_0$ with respect to $P_1$ is used. (This is the skew-symmetry identity
$\Ren_a(P\|Q)=\frac{a}{1-a}\Ren_{1-a}(Q\|P)$ of \cite[Prop.~2]{vEH14}, written
out.)

(iv) This is (ii) with strict Jensen. Let $0<\beta'<\beta\le1$, so
$\rho=\beta'/\beta\in(0,1)$ and $x\mapsto x^{\rho}$ is \emph{strictly} concave
on $[0,\infty)$. The variable $R^{\beta}$ is $P_0$-integrable
($\E_{P_0}[R^{\beta}]=Z_\beta\le1$) and is not $P_0$-a.s.\ constant, because
$\log R$ is not; hence Jensen's inequality is strict,
$Z_{\beta'}<Z_\beta^{\rho}$. Taking $\log$ and dividing by $-\beta'=-\rho\beta<0$
gives $\lambda(\beta')>\lambda(\beta)$.
\end{proof}

\subsection{The conditional R\'enyi quantities}

Nothing in \Cref{def:Zbeta} needs the product structure once it is read one step
at a time and conditionally on the past.

\begin{definition}[Conditional Hellinger integrals and profile]\label{def:condZ}
For $t\in\N$ and $\beta\in(0,1]$ define the $\Fc_{t-1}$-measurable variables
\[
   Z_{\beta,t}:=\E_{\Prob_0}\lB R_t^{\beta}\mid\Fc_{t-1}\rB ,
   \qquad
   \lambda_t(\beta):=-\frac1\beta\log Z_{\beta,t} ,
\]
with $R_t$ as in \Cref{def:ratios}. In the i.i.d.\ model
$Z_{\beta,t}=Z_\beta$ and $\lambda_t(\beta)=\lambda(\beta)$ $\Prob_1$-a.s.\ for
every $t$; in particular both are $\Prob_1$-a.s.\ equal to constants.
\end{definition}

\begin{lemma}[Basic properties, conditionally]\label{lem:Zf-cond}
Under \Cref{ass:standing}, for every $t\in\N$ and $\Prob_0$-a.s.\ --- hence also
$\Prob_1$-a.s.:
\begin{enumerate}[label=(\roman*),leftmargin=2.4em]
\item $Z_{\beta,t}\in(0,1]$ for every $\beta\in(0,1]$, and $Z_{1,t}=1$; hence
      $\lambda_t(1)=0$ and $\lambda_t(\beta)\in[0,\infty)$ for every
      $\beta\in(0,1]$;
\item $\beta\mapsto\lambda_t(\beta)$ is nonincreasing on $(0,1]$, so that
      \[
         \lambda_t(0^{+}):=\lim_{\beta\downarrow0}\lambda_t(\beta)
         \ =\ \sup_{\beta\in(0,1]}\lambda_t(\beta)\ \in[0,\infty]
      \]
      exists.
\end{enumerate}
\end{lemma}

\begin{proof}
Each $Z_{\beta,t}$ is a conditional expectation, hence defined up to a null set
for each fixed $\beta$ separately. Fix versions for $\beta\in(0,1]\cap\mathbb{Q}$ and
read (i) and (ii) as holding, off one null set, simultaneously for rational
$\beta$; this is all that is used, since \Cref{cor:ceiling} fixes a single
$\beta$ and $\lambda_t(0^{+})$ may be defined as the supremum over rational
$\beta$, which by (ii) is the same limit.

(i) $Z_{1,t}=\E_{\Prob_0}[R_t\mid\Fc_{t-1}]=1$ $\Prob_0$-a.s.: on
$\lC L_{t-1}>0\rC$ this is the $\Prob_0$-martingale property of $L$, and on
$\lC L_{t-1}=0\rC$ it is the convention $R_t:=1$ of \Cref{def:ratios}. For
$\beta\in(0,1)$ the map $x\mapsto x^{\beta}$ is concave on $[0,\infty)$, so
conditional Jensen gives $Z_{\beta,t}\le(\E_{\Prob_0}[R_t\mid\Fc_{t-1}])^{\beta}=1$
$\Prob_0$-a.s.; and $Z_{\beta,t}>0$ $\Prob_0$-a.s., since $Z_{\beta,t}=0$ on an
$\Fc_{t-1}$-set of positive $\Prob_0$-measure would force $R_t=0$ $\Prob_0$-a.s.\
there, contradicting $\E_{\Prob_0}[R_t\mid\Fc_{t-1}]=1$. Therefore
$\log Z_{\beta,t}\le0$ and $\lambda_t(\beta)\ge0$.

(ii) Let $0<\beta'\le\beta\le1$ and $\rho:=\beta'/\beta\in(0,1]$. Conditional
Jensen for the concave $x\mapsto x^{\rho}$ gives
$Z_{\beta',t}=\E_{\Prob_0}[(R_t^{\beta})^{\rho}\mid\Fc_{t-1}]
\le Z_{\beta,t}^{\rho}$, so $\log Z_{\beta',t}\le\rho\log Z_{\beta,t}$ and,
dividing by $-\beta'=-\rho\beta<0$, $\lambda_t(\beta')\ge\lambda_t(\beta)$.
\end{proof}

\begin{remark}[Two consequences]\label{rem:Zf-cond}
By \Cref{lem:Zf-cond}(i), $\lC0<Z_{\beta,t}<\infty\rC$ has full $\Prob_0$-measure,
so conventions off it are harmless under either measure. And
$\beta\mapsto\log Z_{\beta,t}$ is a conditional cumulant generating function of
$\log R_t$ under $\Prob_0$, hence convex on $\R$, with the value $0$ at
$\beta=1$; it also takes the value $0$ at $\beta=0$ whenever
$\Prob_0(R_t>0\mid\Fc_{t-1})=1$.
\end{remark}

\subsection{The upper bound}

The engine is H\"older's inequality, and H\"older's inequality holds
conditionally. We therefore prove the bound where it lives --- at one step of an
arbitrary \etm{}, conditionally on the past --- and read the one-\evar{}
statement off as the i.i.d.\ case.

\begin{theorem}[Upper bound: one conditional application of H\"older]
\label{thm:upper-cond}
Let $(E_n)_{n\in\N}$ be a betting sequence, let $n\in\N$ and $s\in(0,\infty)$,
and set $\beta:=\frac{1}{1+s}\in(0,1)$. Then
\begin{equation}\label{eq:upper-cond}
   \E_{\Prob_1}\lB E_n^{-s}\mid\Fc_{n-1}\rB\ \ge\ e^{-\lambda_n(\beta)}
   \qquad\Prob_1\text{-a.s.}
\end{equation}
\end{theorem}

\begin{proof}
\emph{Step 1 (factorisation).} On $\lC0<E_n<\infty\rC$, using
$1-\beta-s\beta=1-\beta-(1-\beta)=0$,
\begin{equation}\label{eq:factor-cond}
   E_n^{\,1-\beta}\lp R_nE_n^{-s}\rp^{\beta}
   \ =\ R_n^{\beta}\,E_n^{\,1-\beta-s\beta}
   \ =\ R_n^{\beta} .
\end{equation}
We dispose of the degenerate cases. By \Cref{lem:supmg} $E_n<\infty$
$\Prob_0$-a.s., hence $\Prob_1$-a.s.\ by \Cref{ass:standing}. And on
$\lC E_n=0\rC$ we have $E_n^{-s}=+\infty$, so on the event
$\lC\Prob_1(E_n=0\mid\Fc_{n-1})>0\rC$ the left-hand side of
\eqref{eq:upper-cond} is $+\infty$ and there is nothing to prove; on the
complement $\Prob_1(E_n=0\mid\Fc_{n-1})=0$. That is a $\Prob_1$-statement and
the conditional expectation below is taken under $\Prob_0$, so the two must be
linked: by \Cref{lem:bayes},
$\Prob_1(E_n=0\mid\Fc_{n-1})=\E_{\Prob_0}[R_n\mathbf 1_{\{E_n=0\}}\mid\Fc_{n-1}]$,
so on that complement $R_n\mathbf 1_{\{E_n=0\}}=0$ \emph{$\Prob_0$-a.s.} Hence
$\Prob_0$-a.s.\ there either $E_n>0$, where \eqref{eq:factor-cond} is the
displayed algebra, or $R_n=0$, where both of its sides vanish (using
$\beta>0$). So \eqref{eq:factor-cond} may be used inside
$\E_{\Prob_0}[\,\cdot\mid\Fc_{n-1}]$.

\emph{Step 2 (conditional H\"older).} The exponents $\frac1{1-\beta}$ and
$\frac1\beta$ are conjugate. Conditional H\"older (\Cref{cor:app-holder}) under
$\E_{\Prob_0}[\,\cdot\mid\Fc_{n-1}]$, applied to \eqref{eq:factor-cond}, gives, $\Prob_1$-a.s.,
\[
   Z_{\beta,n}
   =\E_{\Prob_0}\lB E_n^{\,1-\beta}\lp R_nE_n^{-s}\rp^{\beta}\mid\Fc_{n-1}\rB
   \ \le\ \lp\E_{\Prob_0}\lB E_n\mid\Fc_{n-1}\rB\rp^{1-\beta}
          \lp\E_{\Prob_0}\lB R_nE_n^{-s}\mid\Fc_{n-1}\rB\rp^{\beta} .
\]

\emph{Step 3 (the two conditional constraints).} The first factor is at most
$1$, by \Cref{def:cond-evar} and $1-\beta>0$. The second is
$(\E_{\Prob_1}[E_n^{-s}\mid\Fc_{n-1}])^{\beta}$, by \Cref{lem:bayes} applied to
$Y=E_n^{-s}$. Hence
$Z_{\beta,n}\le(\E_{\Prob_1}[E_n^{-s}\mid\Fc_{n-1}])^{\beta}$ $\Prob_1$-a.s.

\emph{Step 4 (rearranging).} Both sides are positive $\Prob_1$-a.s.\ by
\Cref{lem:Zf-cond}(i); taking $\log$ and multiplying by $-\frac1\beta<0$
reverses the inequality and gives
$\lambda_n(\beta)\ge-\log\E_{\Prob_1}[E_n^{-s}\mid\Fc_{n-1}]$, which is
\eqref{eq:upper-cond}.
\end{proof}

In the i.i.d.\ model $\lambda_n(\beta)=\lambda(\beta)$ is $\Prob_1$-a.s.\ constant and
\eqref{eq:upper-cond} is a statement about a single \evar{}. The next theorem
records that case together with its equality analysis, which is where the
product structure is genuinely used: the conditional H\"older equality case
would carry an $\Fc_{n-1}$-measurable proportionality factor, and identifying it
requires the argument below.

\begin{theorem}[Upper bound: one application of H\"older]\label{thm:upper}
Let $E\in\Ev$ and $s\in(0,\infty)$, and set
\[
   \beta:=\frac{1}{1+s}\in(0,1),\qquad\text{equivalently}\qquad
   s=\frac{1-\beta}{\beta} .
\]
Then
\begin{equation}\label{eq:upper}
   \Lambda_s(E)\ \le\ \lambda(\beta) ,
\end{equation}
with equality if and only if $E=R^{\beta}/Z_\beta$ $P_0$-a.s.
\end{theorem}

\begin{proof}
Fix a $\sigma$-finite measure $\mu$ dominating $P_0$ (for instance $\mu=P_0$)
and write $p_0:=\frac{dP_0}{d\mu}$, $p_1:=\frac{dP_1}{d\mu}$, so that
$R=p_1/p_0$ on $\{p_0>0\}$ and $Z_\beta=\int p_0^{1-\beta}p_1^{\beta}\,d\mu$.

\emph{Step 1 (factorisation).} On $\{0<E<\infty\}$ we have, using
$1-\beta-s\beta=1-\beta-(1-\beta)=0$,
\begin{equation}\label{eq:factor}
   \lp p_0E\rp^{1-\beta}\lp p_1E^{-s}\rp^{\beta}
   \ =\ p_0^{1-\beta}p_1^{\beta}\,E^{\,1-\beta-s\beta}
   \ =\ p_0^{1-\beta}p_1^{\beta} .
\end{equation}
We dispose of two degenerate cases first. If $P_1(E=0)>0$ then
$\E_{P_1}[E^{-s}]=\infty$, so $\Lambda_s(E)=-\infty$ and \eqref{eq:upper} is
trivial; \emph{assume henceforth $P_1(E=0)=0$}. And $P_0(E=\infty)>0$ would give
$\E_{P_0}[E]=\infty>1$, contradicting $E\in\Ev$, so $P_0(E=\infty)=0$.
Consequently $p_1=0$ $\mu$-a.e.\ on $\{E=0\}$ and $p_0=0$ $\mu$-a.e.\ on
$\{E=\infty\}$, and on both of those sets the two sides of \eqref{eq:factor}
vanish (using $1-\beta>0$ and $\beta>0$). Hence \eqref{eq:factor} holds
$\mu$-a.e.\ on $\Xc$.

\emph{Step 2 (H\"older).} The exponents $\frac{1}{1-\beta}$ and $\frac1\beta$
are conjugate, since $(1-\beta)+\beta=1$. Applying H\"older's inequality
\cite[Thm.~3.5]{Rud87} (\Cref{thm:app-holder}) to \eqref{eq:factor},
\begin{equation}\label{eq:holderstep}
\begin{split}
   Z_\beta&=\int\lp p_0E\rp^{1-\beta}\lp p_1E^{-s}\rp^{\beta}d\mu
   \ \le\ \lp\int p_0E\,d\mu\rp^{1-\beta}
          \lp\int p_1E^{-s}\,d\mu\rp^{\beta}\\[2pt]
   &=\lp\E_{P_0}[E]\rp^{1-\beta}\lp\E_{P_1}[E^{-s}]\rp^{\beta} .
\end{split}
\end{equation}

\emph{Step 3 (using the constraint).} $\E_{P_0}[E]\le1$ and $1-\beta>0$, so
$(\E_{P_0}[E])^{1-\beta}\le1$ and
\begin{equation}\label{eq:afterconstraint}
   Z_\beta\ \le\ \lp\E_{P_1}[E^{-s}]\rp^{\beta} .
\end{equation}

\emph{Step 4 (rearranging).} Both sides of \eqref{eq:afterconstraint} are
positive; taking $\log$ and multiplying by $-\frac1\beta<0$ reverses the
inequality:
\[
   \lambda(\beta)=-\frac1\beta\log Z_\beta
   \ \ge\ -\log\E_{P_1}[E^{-s}]=\Lambda_s(E) ,
\]
which is \eqref{eq:upper}.

\emph{Step 5 (equality).} Equality in \eqref{eq:holderstep} holds if and only
if $p_0E$ and $p_1E^{-s}$ are proportional $\mu$-a.e., say
$c_1\,p_0E=c_2\,p_1E^{-s}$ with $c_1,c_2\ge0$ not both zero. We use this on all
of $\Xc$, not only where $p_1>0$. On $\{p_0>0,\,p_1>0\}$ we have
$E\in(0,\infty)$ $\mu$-a.e.\ (Step 1), so $c_1,c_2>0$ there and
$E^{1+s}=\frac{c_2}{c_1}\,\frac{p_1}{p_0}=\frac{c_2}{c_1}R$, i.e.\
$E=\varkappa R^{\beta}$ with $\varkappa:=(c_2/c_1)^{\beta}$, using
$\frac{1}{1+s}=\beta$. On $\{p_0>0,\,p_1=0\}$ the right-hand side vanishes, so
$p_0E=0$ and hence $E=0=R^{\beta}$ there. Therefore
\[
   E\ =\ \varkappa\,R^{\beta}\qquad P_0\text{-a.s.\ on all of }\Xc ,
\]
and only now does the constraint enter: equality in Step 3 requires
$\E_{P_0}[E]=1$, which reads $\varkappa Z_\beta=1$, i.e.\
$\varkappa=Z_\beta^{-1}$. Conversely $E=R^{\beta}/Z_\beta$ gives equality
throughout.
\end{proof}

\Cref{thm:upper} concerns a single \evar{}. The next proposition lifts it to an
arbitrary \etm{}, at no cost in the constant: the bound is applied once per
step, conditionally on the past, and the steps multiply.

\begin{proposition}[Upper bound for an arbitrary \etm{}]\label{prp:upper-seq}
Let $(E_n)_{n\in\N}$ be any betting sequence, with wealth process
$(W_t)_{t\in\N_0}$, and let $s\in(0,\infty)$, $\beta=\frac1{1+s}$. Then:
\begin{enumerate}[label=(\alph*),leftmargin=2.4em]
\item\label{it:agg-exact} for every $t\in\N$,
      \begin{equation}\label{eq:upper-seq-exact}
         \E_{\Prob_1}\lB W_t^{-s}\exp\lp\sum_{n\le t}\lambda_n(\beta)\rp\rB
         \ \ge\ 1 ;
      \end{equation}
\item\label{it:agg-det} if $\bar\lambda_n(\beta)\in[0,\infty]$ are constants
      with $\lambda_n(\beta)\le\bar\lambda_n(\beta)$ $\Prob_1$-a.s., then
      \begin{equation}\label{eq:upper-seq-det}
         \Lambda^{(t)}_s\ \le\ \frac1t\sum_{n\le t}\bar\lambda_n(\beta)
         \qquad\text{for every }t\in\N ;
      \end{equation}
\item\label{it:agg-iid} in the i.i.d.\ model, $\Lambda^{(t)}_s\le\lambda(\beta)$
      for every $t\in\N$, with equality for a given $t$ if and only if
      $E_n=R_\beta(X_n)$ $\Prob_1$-a.s.\ for every $n\le t$, with $R_\beta$ as in
      \Cref{def:flattened} --- equivalently $\Prob_0$-a.s., when $P_0\sim P_1$.
\end{enumerate}
\end{proposition}

\begin{proof}
\emph{\ref{it:agg-exact}.} Put
$M_t:=W_t^{-s}\exp(\sum_{n\le t}\lambda_n(\beta))$, so $M_0=1$. By
\Cref{lem:supmg} and \Cref{ass:standing}, $E_n<\infty$ $\Prob_1$-a.s., so
$W_n^{-s}=W_{n-1}^{-s}E_n^{-s}$ holds $\Prob_1$-a.s.\ (the identity can fail
only where some $E_n=\infty$). Each $\lambda_k(\beta)$ with $k\le n$ is
$\Fc_{n-1}$-measurable, and $W_{n-1}^{-s}$ is nonnegative and
$\Fc_{n-1}$-measurable, so the pull-out property and \Cref{thm:upper-cond} give
\[
   \E_{\Prob_1}\lB M_n\mid\Fc_{n-1}\rB
   \ =\ M_{n-1}\,e^{\lambda_n(\beta)}\,
        \E_{\Prob_1}\lB E_n^{-s}\mid\Fc_{n-1}\rB
   \ \ge\ M_{n-1}
   \qquad\Prob_1\text{-a.s.}
\]
Taking expectations and inducting from $\E_{\Prob_1}[M_0]=1$ gives
\eqref{eq:upper-seq-exact}.

\emph{\ref{it:agg-det}.} From $\lambda_n(\beta)\le\bar\lambda_n(\beta)$ we get
$M_t\le W_t^{-s}\exp(\sum_{n\le t}\bar\lambda_n(\beta))$ $\Prob_1$-a.s., so
\eqref{eq:upper-seq-exact} gives
$\E_{\Prob_1}[W_t^{-s}]\ge\exp(-\sum_{n\le t}\bar\lambda_n(\beta))$; take
$-\frac1t\log$.

\emph{\ref{it:agg-iid}.} In the i.i.d.\ model $\lambda_n(\beta)=\lambda(\beta)$
$\Prob_1$-a.s.\ for every $n$, so \ref{it:agg-det} applies with
$\bar\lambda_n(\beta)=\lambda(\beta)$ and yields
$\Lambda^{(t)}_s\le\lambda(\beta)$. For the equality case, note first that
$\E_{\Prob_1}[E_n^{-s}\mid\Fc_{n-1}]\ge e^{-\lambda(\beta)}$ is
\Cref{thm:upper-cond}, and that by \Cref{lem:flat}
$\Lambda_s(R_\beta)=\lambda(\beta)$; so if $E_n=R_\beta(X_n)$ $\Prob_1$-a.s.\ for
all $n\le t$ then every inequality in the induction above is an equality and
$\E_{\Prob_1}[W_t^{-s}]=e^{-t\lambda(\beta)}$. Conversely suppose
$\Lambda^{(t)}_s=\lambda(\beta)$, i.e.\
$\E_{\Prob_1}[W_t^{-s}]=e^{-t\lambda(\beta)}\in(0,\infty)$. Running the
induction from index $n$ to index $t$ gives
$\E_{\Prob_1}[W_n^{-s}]\le e^{(t-n)\lambda(\beta)}\E_{\Prob_1}[W_t^{-s}]<\infty$
for every $n\le t$, so every term in the chain is finite and positive and each
inequality must be an equality. Since $\E_{\Prob_0}[E_n\mid\Fc_{n-1}]\le1$
forces $E_n<\infty$ $\Prob_0$-a.s., and $\Prob_1\ll\Prob_0$ on $\Fc_t$, we have
$W_{n-1}^{-s}>0$ $\Prob_1$-a.s.; hence equality at the $n$-th step forces
$\E_{\Prob_1}[E_n^{-s}\mid\Fc_{n-1}]=e^{-\lambda(\beta)}$ $\Prob_1$-a.s. Now
write $E_n=g_n(X_1,\dots,X_n)$ with $g_n$ measurable and, for
$y\in\Xc^{n-1}$, put $E_n^{y}(x):=g_n(y,x)$. Because $\Prob_i=P_i^{\otimes\N}$,
Fubini's theorem gives
\begin{equation}\label{eq:fubini}
   \E_{\Prob_i}\lB h(E_n)\mid\Fc_{n-1}\rB
   =\E_{P_i}\lB h\lp E_n^{Y}\rp\rB\Big|_{Y=(X_1,\dots,X_{n-1})}
   \qquad\Prob_i\text{-a.s.}
\end{equation}
for every measurable $h:[0,\infty]\to[0,\infty]$. With $i=0$ and
$h=\mathrm{id}$ this says $E_n^{y}\in\Ev$ for $P_0^{\otimes(n-1)}$-a.e.\ $y$,
hence for $P_1^{\otimes(n-1)}$-a.e.\ $y$ since $P_1\ll P_0$; with $i=1$ and
$h(x)=x^{-s}$ it turns the displayed equality into
$\Lambda_s(E_n^{y})=\lambda(\beta)$ for $P_1^{\otimes(n-1)}$-a.e.\ $y$. The
equality case of \Cref{thm:upper} identifies $E_n^{y}=R_\beta$ $P_0$-a.s., hence
also $P_1$-a.s.; that is, $E_n=R_\beta(X_n)$ $\Prob_1$-a.s.
\end{proof}

\begin{remark}[The equality case is genuinely only $\Prob_1$-a.s.]
\label{rem:eqcase}
Under \Cref{ass:standing} alone the conclusion of \Cref{prp:upper-seq} cannot be
strengthened to $\Prob_0$-a.s. Take $\Xc=\{0,1\}$,
$P_0=\mathrm{Bern}(\frac12)$ and $P_1=\delta_1$, so that $P_1\ll P_0$ and
$\KL(P_1\|P_0)=\log2<\infty$; here $R=2\cdot\mathbf 1_{\{1\}}$,
$Z_\beta=2^{\beta-1}$ and $R_\beta=R$. With $t=2$, $E_1:=R_\beta(X_1)$ and
\[
   E_2:=R_\beta(X_2)\mathbf 1_{\{X_1=1\}}+\mathbf 1_{\{X_1=0\}} ,
\]
which is a legitimate betting sequence, one has $W_2=4$ $\Prob_1$-a.s., hence
$\E_{\Prob_1}[W_2^{-s}]=4^{-s}=e^{-2\lambda(\beta)}$ and $\Lambda^{(2)}_s=\lambda(\beta)$:
equality holds, although $\Prob_0(E_2\ne R_\beta(X_2))=\frac12$. Under
$P_0\sim P_1$ the two conclusions coincide.
\end{remark}

\begin{corollary}[The ceiling]\label{cor:ceiling}
For every betting sequence, every $s\in[0,\infty)$ and every $t\in\N$,
\begin{equation}\label{eq:ceiling}
   \Lambda^{(t)}_s\ \le\ \frac1t\sum_{n\le t}\ \bar\lambda_n(0^{+}),
   \qquad
   \bar\lambda_n(0^{+}):=\operatorname*{ess\,sup}_{\Prob_1}\ \lambda_n(0^{+}) ,
\end{equation}
with $\lambda_n(0^{+})$ as in \Cref{lem:Zf-cond}(ii); the same bound therefore
holds for $\Lambda^{(t)}$. The quantity being averaged is a divergence:
\begin{equation}\label{eq:lambda0-is-kl}
   \lambda_n(0^{+})\ =\ \KLrev_n\ :=\ -\E_{\Prob_0}\lB\log R_n\mid\Fc_{n-1}\rB ,
\end{equation}
the conditional relative entropy of $\Prob_0$ with respect to $\Prob_1$ over
step $n$ (\Cref{lem:limit-cond}), so \eqref{eq:ceiling} caps the exponent by the
average of the first $t$ per-step conditional Chernoff--Stein exponents. In the
i.i.d.\ model the right-hand side is
$\lim_{\beta\downarrow0}\lambda(\beta)=\KL(P_0\|P_1)$, so
\begin{equation}\label{eq:ceiling-iid}
   \Lambda^{(t)}_s\ \le\ \KL(P_0\|P_1) ,
\end{equation}
and in particular $\Lambda(E)\le\KL(P_0\|P_1)$ for every $E\in\Ev$.
\end{corollary}

\begin{proof}
For $s=0$ the left-hand side is $0$ and the right-hand side is nonnegative by
\Cref{lem:Zf-cond}(i). For $s>0$ put $\beta=\frac1{1+s}\in(0,1)$. By
\Cref{lem:Zf-cond}(ii), $\lambda_n(\beta)\le\lambda_n(0^{+})\le\bar\lambda_n(0^{+})$
$\Prob_1$-a.s., and the right-hand bound is a constant; so
\Cref{prp:upper-seq}\ref{it:agg-det} applies with
$\bar\lambda_n(\beta)=\bar\lambda_n(0^{+})$ and gives \eqref{eq:ceiling}. Taking
the supremum over $s$ gives the statement for $\Lambda^{(t)}$.

In the i.i.d.\ model $\lambda_n(0^{+})=\lim_{\beta\downarrow0}\lambda(\beta)$ is
the same constant for every $n$, and that limit exists in $[0,\infty]$ by
\Cref{lem:Zf}(ii) and equals $\KL(P_0\|P_1)$ by \Cref{lem:limit}; the last claim
is \Cref{lem:iidexp}. Identity \eqref{eq:lambda0-is-kl} is \Cref{lem:limit-cond}.
\end{proof}

\begin{remark}[What the general ceiling settles, and what it does not]
\label{rem:generalceiling}
\Cref{cor:ceiling} holds on any filtered space: no product structure, no
stationarity, no identically distributed anything. Three things about it are
settled elsewhere, and it is worth saying where.

\emph{It is a divergence, up to an essential supremum.} By
\eqref{eq:lambda0-is-kl} each $\lambda_n(0^{+})$ is a conditional relative
entropy, so \eqref{eq:ceiling} is the Ces\`aro average of their essential
suprema, and in particular
$\frac1t\sum_{n\le t}\bar\lambda_n(0^{+})
\ge\frac1t\KL(\Prob_0|_{\Fc_t}\|\Prob_1|_{\Fc_t})$ by \Cref{lem:compensator}.
The two sides agree when the $\KLrev_n$ are $\Prob_1$-a.s.\ constant
(\Cref{ass:homog}), and then the average is the predictable compensator of
$-\frac1t\log L_t$ exactly. Whether it converges as $t\to\infty$ is a property
of the model, not of the method: it does for stationary ergodic pairs, where the
limit is the relative entropy rate, and it can fail for independent but
non-identically distributed ones (\Cref{rem:rate}).

\emph{It is tight, but not for free.} A witness exists in full generality: the
predictable flattened design $R_{\beta,n}$ of \Cref{def:flat-cond} attains the
per-step bound with no extra hypothesis (\Cref{lem:flat-cond}). What is needed
for the per-step values to aggregate into a deterministic total is that they be
non-random (\Cref{ass:homog}); under that hypothesis
\Cref{thm:ceiling-gen} matches \eqref{eq:ceiling}, and under its further
non-degeneracy hypotheses shows the value is not attained. In the i.i.d.\ model
\Cref{ass:homog} is automatic, the total is $t\KL(P_0\|P_1)$, and the matched
statement is \Cref{thm:ceiling}. Without \Cref{ass:homog} the sharpness of
\eqref{eq:ceiling} is open: the essential suprema may be strict, and we know of
no matching example.

\emph{It bounds the exponent, not yet the error.} \eqref{eq:ceiling} caps
$\Lambda^{(t)}$, which enters the error only through \Cref{lem:chernoff}. That
the same ceiling binds the type-II error itself is \Cref{prp:stein-ceiling} in
the i.i.d.\ model, and \Cref{prp:stein-gen} in general, under an
information-spectrum hypothesis on $-\frac1t\log L_t$ (\Cref{ass:spectrum});
\Cref{prp:rate} identifies its constant with the same average, under a
conditional-variance bound.

\emph{And the dichotomy generalises for free.} The \epower{} side of
\Cref{thm:two} needs no extra hypothesis at all: by \Cref{lem:epower-cond} the
conditional \epower{} of any bet is at most the conditional relative entropy
$\KL_n$ in the opposite direction, with equality if and only if $E_n=R_n$. So on
any filtered space the two ceilings are the two conditional divergences, one
attained step by step by the likelihood ratio and one not.

\Cref{app:general} carries all of this out.
\end{remark}

\begin{proposition}[The ceiling binds the error, not only the bound]
\label{prp:stein-ceiling}
Work in the i.i.d.\ model (\Cref{def:iid}) and assume $\KL(P_0\|P_1)<\infty$.
For every betting sequence and every
$\alpha\in(0,1)$,
\begin{equation}\label{eq:stein-ceiling}
   \limsup_{t\to\infty}\ -\frac1t\log\bar\gamma_t(\alpha)
   \ \le\ \KL(P_0\|P_1),
\end{equation}
and the same holds with $\gamma_t(\alpha)$ in place of $\bar\gamma_t(\alpha)$.
\end{proposition}

\begin{proof}
Fix $t$ and put $C_t:=\{\tau_\alpha>t\}$. Then $C_t\in\Fc_t=\Bs{\Xc}^{\otimes t}$,
and $P_0^{\otimes t}(C_t^{c})=\Prob_0(\tau_\alpha\le t)\le
\Prob_0(\tau_\alpha<\infty)\le\alpha$ by \Cref{thm:typeI}, so $C_t$ is an
admissible acceptance region in the Chernoff--Stein lemma, whose value
$\gamma^{\star}_t(\alpha)$ is the least type-II error of any test of level
$\alpha$ at horizon $t$
\cite{Chernoff56}, \cite[Thm.~11.8.3]{CT06}, \cite{HK89}
(\Cref{thm:app-stein}). Its type-II error is
$P_1^{\otimes t}(C_t)=\bar\gamma_t(\alpha)$, whence
$\bar\gamma_t(\alpha)\ge\gamma^{\star}_t(\alpha)$ and
$-\frac1t\log\bar\gamma_t(\alpha)\le-\frac1t\log\gamma^{\star}_t(\alpha)$. Let
$t\to\infty$ and apply \Cref{thm:app-stein}; the range-free form quoted there
(valid for every $\alpha\in(0,1)$ by the strong converse of \cite{HK89}) is what
lets $\alpha$ be arbitrary. For $\gamma_t(\alpha)$ take
instead $C_t:=\{W_t\le\frac1\alpha\}\in\Fc_t$, whose type-I error
$\Prob_0(W_t>\frac1\alpha)\le\alpha\,\E_{\Prob_0}[W_t]\le\alpha$ by Markov.
\end{proof}

\begin{lemma}[The limit of $\lambda$ at $0$; i.i.d.\ model]\label{lem:limit}
$\displaystyle\lim_{\beta\downarrow0}\lambda(\beta)=\KL(P_0\|P_1)$, in $[0,\infty]$.
\end{lemma}

\begin{proof}
For $\beta\in(0,1]$ define $h_\beta:\Xc\to\R$ by
$h_\beta:=\frac{R^{\beta}-1}{\beta}$. Three observations.

\emph{(a) Monotone decrease.} For fixed $x\in[0,\infty)$ the map
$\beta\mapsto\frac{x^{\beta}-1}{\beta}$ is nondecreasing on $(0,1]$. For $x>0$
it is the difference quotient at $0$ of the convex map
$\beta\mapsto e^{\beta\log x}$, hence nondecreasing for every value of
$\log x\in\R$; for $x=0$ it equals $-\frac1\beta$, which is nondecreasing. In
both cases it decreases to $\log x\in[-\infty,\infty)$ as $\beta\downarrow0$.
Hence $h_\beta\downarrow\log R$ pointwise.

\emph{(b) An integrable majorant.} By (a), $h_\beta\le h_1=R-1$ for all
$\beta\in(0,1]$, and $\E_{P_0}[|R-1|]\le2<\infty$.

\emph{(c) Passing to the limit.} By monotone convergence for the decreasing
family $(h_\beta)_{\beta\in(0,1]}$ dominated above by the integrable $R-1$,
\[
   \lim_{\beta\downarrow0}\E_{P_0}\lB h_\beta\rB
   =\E_{P_0}\lB\log R\rB=-\KL(P_0\|P_1)\in[-\infty,0] ,
\]
the last equality because $\KL(P_0\|P_1)=\E_{P_0}[\log\frac{dP_0}{dP_1}]
=-\E_{P_0}[\log R]$ when $P_0\ll P_1$, and both sides are $+\infty$ otherwise.

Now $\E_{P_0}[h_\beta]=\frac{Z_\beta-1}{\beta}$ and
$\lambda(\beta)=-\frac{\log Z_\beta}{\beta}$. If $\KL(P_0\|P_1)<\infty$ then
$Z_\beta\to1$, so $\log Z_\beta=(Z_\beta-1)(1+o(1))$ and
$\lambda(\beta)=-\frac{Z_\beta-1}{\beta}(1+o(1))\to\KL(P_0\|P_1)$. If
$\KL(P_0\|P_1)=\infty$ then $\frac{Z_\beta-1}{\beta}\to-\infty$; since
$\log y\le y-1$ for $y>0$ we get $\lambda(\beta)\ge-\frac{Z_\beta-1}{\beta}\to\infty$.
\end{proof}

\subsection{The lower bound}

From here to the end of \Cref{sec:horizon} we work in the i.i.d.\ model of
\Cref{def:iid}; \Cref{rem:whatgeneralises} records what this costs.

\begin{lemma}[$R_\beta$ is an \evar{} attaining the bound]\label{lem:flat}
For every $\beta\in(0,1]$: $R_\beta\in\Ev$ with $\E_{P_0}[R_\beta]=1$, and with
$s=\frac{1-\beta}{\beta}$,
\begin{equation}\label{eq:flat-attains}
   \Lambda_{s}(R_\beta)\ =\ \lambda(\beta),
   \qquad\text{hence}\qquad
   \Lambda(R_\beta)\ \ge\ \lambda(\beta) .
\end{equation}
\end{lemma}

\begin{proof}
$\E_{P_0}[R_\beta]=\E_{P_0}[R^{\beta}]/Z_\beta=1$, so $R_\beta\in\Ev$. For
\eqref{eq:flat-attains} compute directly:
\[
   \E_{P_1}\lB R_\beta^{-s}\rB
   =Z_\beta^{s}\,\E_{P_1}\lB R^{-\beta s}\rB
   =Z_\beta^{s}\,\E_{P_0}\lB R^{1-\beta s}\rB ,
\]
using the change of measure $\E_{P_1}[\varphi(R)]=\E_{P_0}[R\,\varphi(R)]$,
valid for every measurable $\varphi:[0,\infty)\to[0,\infty]$. Since
$\beta s=1-\beta$ we have
$1-\beta s=\beta$, so $\E_{P_0}[R^{1-\beta s}]=Z_\beta$ and
\[
   \E_{P_1}\lB R_\beta^{-s}\rB=Z_\beta^{s+1}=Z_\beta^{1/\beta} ,
\]
because $s+1=\frac{1-\beta}{\beta}+1=\frac1\beta$. Taking $-\log$,
$\Lambda_s(R_\beta)=-\frac1\beta\log Z_\beta=\lambda(\beta)$. The inequality follows
from $\Lambda(R_\beta)=\sup_{s'\ge0}\Lambda_{s'}(R_\beta)\ge\Lambda_s(R_\beta)$.
\end{proof}

\begin{corollary}[The exponent of the likelihood ratio]\label{cor:chernoffinfo}
Under \Cref{ass:standing} with $\KL(P_0\|P_1)<\infty$,
\begin{equation}\label{eq:chernoffinfo}
   \Lambda(R)\ =\ \sup_{\beta\in(0,1]}\beta\,\lambda(\beta)
   \ =\ -\inf_{\beta\in(0,1]}\log Z_\beta\ =\ C(P_0,P_1) ,
\end{equation}
the Chernoff information of the pair. Moreover, if $\log R$ is not $P_0$-a.s.\
constant then $\Lambda(R_\beta)>\lambda(\beta)$ strictly, for every
$\beta\in(0,1]$.
\end{corollary}

\begin{proof}
For $s\ge0$, the change of measure $\E_{P_1}[\varphi(R)]=\E_{P_0}[R\,\varphi(R)]$
gives $\E_{P_1}[R^{-s}]=\E_{P_0}[R^{1-s}]=Z_{1-s}$ when $s\in[0,1)$, so
\[
   \Lambda_s(R)=-\log Z_{1-s} ;
\]
valid for every $s\ge0$ because $\KL(P_0\|P_1)<\infty$ gives $P_0(R>0)=1$, so
that $R>0$ holds $P_0$-a.s.\ as well as $P_1$-a.s. Substituting $\beta=1-s$
turns $\sup_{s\ge0}$ into $\sup_{\beta\le1}$ of $-\log Z_\beta$. The map
$\beta\mapsto\log Z_\beta$ is convex with $Z_0=Z_1=1$, so $-\log Z_\beta\le0$
for $\beta\le0$ while $-\log Z_\beta\ge0$ on $[0,1]$; the supremum is therefore
attained on $(0,1]$ and equals $\sup_{\beta\in(0,1]}\beta\lambda(\beta)$. The
last equality is the definition of the Chernoff information. For the strictness, write
$dP_\beta:=R_\beta\,dP_0$; differentiating $\sigma\mapsto\Lambda_\sigma(R_\beta)$
at the paired tilt $\sigma=s=\frac{1-\beta}{\beta}$ gives
\[
   \frac{d}{d\sigma}\Lambda_\sigma(R_\beta)\Big|_{\sigma=s}
   =\frac{\E_{P_1}[R_\beta^{-s}\log R_\beta]}{\E_{P_1}[R_\beta^{-s}]}
   =\KL(P_\beta\|P_0)\ >\ 0 ,
\]
the last inequality because $P_\beta=P_0$ would force $R^{\beta}$ to be
$P_0$-a.s.\ constant. So $\Lambda(R_\beta)\ge\Lambda_{\sigma}(R_\beta)>
\Lambda_s(R_\beta)=\lambda(\beta)$ for $\sigma$ slightly above $s$.
\end{proof}

\begin{remark}[Three named quantities]\label{rem:three}
\eqref{eq:chernoffinfo} identifies the exponent of the growth-optimal design in
closed form whenever $\KL(P_0\|P_1)<\infty$: it is the Chernoff information, the
largest rectangle $\beta\times\lambda(\beta)$ under the curve of
\Cref{fig:two}. The subject therefore has three named constants,
\[
   C(P_0,P_1)\ \le\ \min\lC\KL(P_1\|P_0),\ \KL(P_0\|P_1)\rC :
\]
the exponent $C(P_0,P_1)$ of the \epower{}-optimal design; the \epower{} ceiling
$\KL(P_1\|P_0)$, which that design attains; and the exponent ceiling
$\KL(P_0\|P_1)$, which no design attains. The last two are not comparable with
each other (\Cref{rem:notcomparable}).
\end{remark}

\begin{corollary}[Lower bound; i.i.d.\ model]\label{cor:lower}
$\displaystyle\sup_{E\in\Ev}\Lambda(E)\ \ge\ \KL(P_0\|P_1)$.
\end{corollary}

\begin{proof}
By \Cref{lem:flat}, $\sup_{E\in\Ev}\Lambda(E)\ge\sup_{\beta\in(0,1]}\lambda(\beta)
=\lim_{\beta\downarrow0}\lambda(\beta)$, the equality by monotonicity
(\Cref{lem:Zf}(ii)); and that limit is $\KL(P_0\|P_1)$ by \Cref{lem:limit}.
\end{proof}

\begin{remark}[What needs the i.i.d.\ model, and what does not]
\label{rem:whatgeneralises}
The upper half of this section is stated in the setting of \Cref{def:setting}:
\Cref{thm:upper-cond}, \Cref{prp:upper-seq}\ref{it:agg-exact}--\ref{it:agg-det}
and \Cref{cor:ceiling} use only that $\Prob_0,\Prob_1$ are laws on a filtered
space with $\Prob_1\ll\Prob_0$ on each $\Fc_t$. So do
\Cref{def:cond-evar}--\Cref{lem:order}, \Cref{lem:chernoff}, and
\Cref{sec:nfl} apart from \Cref{rem:quantifier}, \Cref{prp:powerone} and
\Cref{rem:cramer}, which name the i.i.d.\ model. What the general form gives up is that the ceiling is no longer
one number: it is the average $\frac1t\sum_{n\le t}\bar\lambda_n(0^{+})$ of
per-step conditional Chernoff--Stein exponents, and \Cref{cor:ceiling} must pass
to an essential supremum to make the per-step bounds aggregate.

Three things below genuinely use more. The lower bound needs a witness
whose per-step value is non-random, which is what lets the
per-step bounds aggregate deterministically; \Cref{prp:stein-ceiling} invokes the
Chernoff--Stein lemma, an i.i.d.\ statement; and \Cref{sec:horizon} rests on the
identity of \Cref{prp:exactfamily}, which requires $Z_{\beta,n}$ to be
deterministic and constant in $n$. The equality analysis of
\Cref{thm:upper}, and hence \Cref{prp:upper-seq}\ref{it:agg-iid} and
\Cref{prp:notattained}, is stated in the i.i.d.\ model for a different reason:
the conditional H\"older equality case carries an $\Fc_{n-1}$-measurable
proportionality factor, and identifying it is what the sections argument does.

\Cref{app:general} carries each of these out: it states the replacement for
each i.i.d.-model result of the main text, together with the hypothesis it
needs, and \Cref{tab:general} is the index; \Cref{rem:generalceiling} summarises
what the general ceiling does and does not settle. Three of them --- the conditional \epower{}
ceiling, the conditional flattened design, and the equality case --- need no
extra hypothesis at all.
\end{remark}

\subsection{Non-attainment}

Both statements of this subsection are about an arbitrary \etm{} in the i.i.d.\
model, and both are about $\Lambda^{(t)}$ of \Cref{def:exponent}.

\begin{lemma}[The supremum over tilts is attained at a finite tilt]
\label{lem:attained}
Assume in addition to \Cref{ass:standing} that $P_0\ll P_1$, so that
$P_0\sim P_1$. Let $(E_n)_{n\in\N}$ be a betting sequence and $t\in\N$ with
$\Prob_0(W_t\ne1)>0$. Then there is $s^{\ast}\in[0,\infty)$ with
$\Lambda^{(t)}=\Lambda^{(t)}_{s^{\ast}}$.
\end{lemma}

\begin{proof}
First, $\Prob_1(W_t<1)>0$. Indeed, suppose $\Prob_1(W_t<1)=0$, i.e.\ $W_t\ge1$
$\Prob_1$-a.s.; since $P_0\sim P_1$ gives $\Prob_0|_{\Fc_t}\sim\Prob_1|_{\Fc_t}$
and $W_t$ is $\Fc_t$-measurable, this gives $W_t\ge1$ $\Prob_0$-a.s., hence
$1\ge\E_{\Prob_0}[W_t]\ge1$ by the supermartingale property, so
$\E_{\Prob_0}[W_t-1]=0$ with $W_t-1\ge0$ $\Prob_0$-a.s., i.e.\ $W_t=1$
$\Prob_0$-a.s., contradicting the hypothesis.

Let $\delta:=\Prob_1(W_t<1)>0$ and pick $\epsilon\in(0,1)$ with
$\Prob_1(W_t\le1-\epsilon)\ge\delta/2$, possible since
$\{W_t<1\}=\bigcup_{k}\{W_t\le1-\frac1k\}$. Then for $s\ge0$,
\[
   \E_{\Prob_1}\lB W_t^{-s}\rB\ \ge\ (1-\epsilon)^{-s}\,\tfrac\delta2
   \ \xrightarrow[s\to\infty]{}\ \infty ,
\]
so $\Lambda^{(t)}_s\to-\infty$. Since $\Lambda^{(t)}_0=0$, there is $S<\infty$
with $\Lambda^{(t)}_s<0$ for all $s>S$, hence
$\Lambda^{(t)}=\sup_{s\in[0,S]}\Lambda^{(t)}_s$. Finally, $W_t<\infty$
$\Prob_1$-a.s.\ (\Cref{lem:supmg} and \Cref{ass:standing}), so for $\Prob_1$-a.e.\
$\omega$ the map $s\mapsto W_t(\omega)^{-s}$ is lower semicontinuous on
$[0,\infty)$: it is continuous where $W_t(\omega)\in(0,\infty)$, and where
$W_t(\omega)=0$ it equals $1$ at $s=0$ and $+\infty$ for $s>0$. By Fatou's
lemma $s\mapsto\E_{\Prob_1}[W_t^{-s}]$ is then lower semicontinuous, so
$s\mapsto\Lambda^{(t)}_s$ is upper semicontinuous on the compact $[0,S]$ and
attains its supremum there.
\end{proof}

\begin{remark}[Why not convexity]\label{rem:whyfatou}
The map $s\mapsto\log\E_{\Prob_1}[W_t^{-s}]$ is convex, but convexity alone
would not close the argument above: a convex function into $(-\infty,\infty]$
can fail to be lower semicontinuous at the endpoint of its effective domain,
which is exactly where the supremum may sit.
\end{remark}

\begin{proposition}[The ceiling is not attained; i.i.d.\ model]\label{prp:notattained}
Assume $\KL(P_0\|P_1)<\infty$ --- which with \Cref{ass:standing} already gives
$P_0\sim P_1$ --- and $P_0\ne P_1$, equivalently that $\log R$ is not
$P_0$-a.s.\ constant. Then
\[
   \Lambda^{(t)}\ <\ \KL(P_0\|P_1)
   \qquad\text{for every betting sequence and every }t\in\N ;
\]
in particular $\Lambda(E)<\KL(P_0\|P_1)$ for every $E\in\Ev$.
\end{proposition}

\begin{proof}
If $W_t=1$ $\Prob_0$-a.s.\ then $\Lambda^{(t)}_s=0$ for every $s$, so
$\Lambda^{(t)}=0<\KL(P_0\|P_1)$, the last inequality because $\log R$ is not
$P_0$-a.s.\ constant, so $P_0\ne P_1$. Otherwise \Cref{lem:attained} gives
$s^{\ast}\in[0,\infty)$ with $\Lambda^{(t)}=\Lambda^{(t)}_{s^{\ast}}$. If
$s^{\ast}=0$ then $\Lambda^{(t)}=0<\KL(P_0\|P_1)$. If $s^{\ast}>0$ then
\Cref{prp:upper-seq} gives $\Lambda^{(t)}\le \lambda(\beta^{\ast})$ with
$\beta^{\ast}=\frac1{1+s^{\ast}}\in(0,1)$, and \Cref{lem:Zf}(iv) gives
$\lambda(\beta^{\ast})<\lim_{\beta\downarrow0}\lambda(\beta)=\KL(P_0\|P_1)$ since $\lambda$ is
strictly decreasing. For the last clause take the i.i.d.\ product generated by
$E$ at $t=1$, where $W_1=E$ and $\Lambda^{(1)}=\Lambda(E)$: if $E=1$ $P_0$-a.s.\
then $\Lambda(E)=0<\KL(P_0\|P_1)$, and otherwise $\Prob_0(W_1\ne1)>0$.
\end{proof}

\subsection{The two-sided statement}

\begin{theorem}[The ceiling, matched; i.i.d.\ model]\label{thm:ceiling}
Work in the i.i.d.\ model of \Cref{def:iid}. Let $\mathcal W$ denote the class of all betting sequences (equivalently, of all
\etm s) for $\Prob_0$, and $\mathcal W_{\mathrm{iid}}\subseteq\mathcal W$ the
subclass of i.i.d.\ products generated by a single $E\in\Ev$. Under
\Cref{ass:standing}, for every $t\in\N$,
\begin{equation}\label{eq:matched}
   \sup_{\mathcal W}\ \Lambda^{(t)}
   \ =\ \sup_{\mathcal W_{\mathrm{iid}}}\ \Lambda^{(t)}
   \ =\ \sup_{E\in\Ev}\ \Lambda(E)
   \ =\ \KL(P_0\|P_1) ,
\end{equation}
the common value being approached along the flattened family
$(R_\beta)_{\beta\in(0,1]}$ of \Cref{def:flattened} as $\beta\downarrow0$. If
moreover $P_0\sim P_1$, $\KL(P_0\|P_1)<\infty$, and $\log R$ is not $P_0$-a.s.\
constant, then no member of $\mathcal W$ attains it.

Consequently, for every betting sequence, every $\alpha\in(0,1)$ and every
$t\in\N$,
\begin{equation}\label{eq:matched-error}
   \bar\gamma_t(\alpha)\ \le\ \gamma_t(\alpha)
   \ \le\ \exp\lp s\ell-t\,\Lambda^{(t)}_s\rp
   \quad\text{for all }s\ge0,
\end{equation}
and no betting sequence makes the exponent $\Lambda^{(t)}_s$ exceed
$\KL(P_0\|P_1)$, for any $s$ and any $t$.
\end{theorem}

\begin{proof}
$\sup_{\mathcal W}\Lambda^{(t)}\le\KL(P_0\|P_1)$ is \Cref{cor:ceiling}, and
$\sup_{\mathcal W_{\mathrm{iid}}}\Lambda^{(t)}=\sup_{E\in\Ev}\Lambda(E)\ge
\KL(P_0\|P_1)$ is \Cref{lem:iidexp} with \Cref{cor:lower}; since
$\mathcal W_{\mathrm{iid}}\subseteq\mathcal W$ the three suprema are squeezed
together. Non-attainment is \Cref{prp:notattained};
\eqref{eq:matched-error} is \Cref{lem:chernoff}; the final clause is
\Cref{cor:ceiling}.
\end{proof}

\begin{remark}[Where the two bounds live]\label{rem:scope}
The two halves of \eqref{eq:matched} have different scopes, and that asymmetry
is the content of the theorem. The upper bound holds for every \etm{}: the
predictable form \Cref{prp:upper-seq} applies \Cref{thm:upper-cond} once per step,
conditionally on the past, so adaptivity buys nothing at all --- and by
\Cref{prp:stein-ceiling} the same ceiling binds the type-II error itself and
not merely the Chernoff bound on it. The lower bound is attained already in the
smallest natural subclass, the i.i.d.\ products of a single \evar{}. Enlarging
the class from $\mathcal W_{\mathrm{iid}}$ to $\mathcal W$ therefore does not
raise the achievable exponent by so much as an $\epsilon$; and by the equality
case of \Cref{prp:upper-seq}, the only betting sequence achieving the bound
$\lambda(\beta)$ at a given tilt is, up to $\Prob_1$-null modification, the i.i.d.\
product of $R_\beta$.
\end{remark}

\begin{remark}[The value of the ceiling]\label{rem:stein}
$\KL(P_0\|P_1)$ is the Chernoff--Stein exponent: the optimal type-II error
exponent among all tests, sequential or not, whose type-I error is at most a
fixed $\alpha$ \cite{Chernoff56}, \cite[Thm.~11.8.3]{CT06}
(\Cref{thm:app-stein}). \Cref{thm:ceiling} therefore says that
the class of \etm s is exponent-optimal as a class, and that --- when
$\KL(P_0\|P_1)<\infty$ and $P_0\ne P_1$ --- no member of it attains the ceiling
at any finite horizon $t$. It does not say that no member
approaches it as $t\to\infty$: \Cref{prp:notattained} bounds $\Lambda^{(t)}$ for
each fixed $t$ and says nothing about $\limsup_{t}\Lambda^{(t)}$, which a single
sequence of bets with a tilt $\beta$ decreasing along blocks can drive to
$\KL(P_0\|P_1)$.
\end{remark}

\begin{remark}[Attribution]\label{rem:attribution}
\Cref{lem:Zf}(ii) and \Cref{lem:limit} are, in the R\'enyi form
$\lambda(\beta)=\Ren_{1-\beta}(P_0\|P_1)$ of \Cref{lem:Zf}(iii), the monotonicity of
R\'enyi divergence in its order and its limit at order one
\cite[Thms.~3 and 5]{vEH14}; we have given direct proofs to keep the
development self-contained. \Cref{thm:upper} with its equality case is the
$\Ev$-constrained instance of a known variational characterisation of R\'enyi
divergence \cite{Ana18,ACD15}, restated as \Cref{thm:app-variational}. What the two-sided form
\eqref{eq:matched} adds is the identification of the exact range of
$\Lambda$ on $\Ev$, together with the non-attainment.
\end{remark}

\subsection{The two ceilings}

Each of the two scalars of \Cref{rem:two-scalars} has a least upper bound
over $\Ev$. They are the same divergence in opposite arguments, and one is
attained where the other is not.

\begin{theorem}[The two ceilings; i.i.d.\ model]\label{thm:two}
Under \Cref{ass:standing}, in the i.i.d.\ model:
\begin{enumerate}[label=(\alph*),leftmargin=2.4em]
\item $\displaystyle\sup_{E\in\Ev}\ \E_{P_1}[\log E]\ =\ \KL(P_1\|P_0)$, and the
      supremum is attained, at $E=R$ and, up to $P_0$-null modification, only
      there;
\item $\displaystyle\sup_{E\in\Ev}\ \Lambda(E)\ =\ \KL(P_0\|P_1)$, and if
      $P_0\sim P_1$, $\KL(P_0\|P_1)<\infty$ and $\log R$ is not $P_0$-a.s.\
      constant, the supremum is attained by no $E\in\Ev$ --- indeed by no
      \etm{}, at any horizon.
\end{enumerate}
\end{theorem}

\begin{proof}
(a) is \Cref{lem:epower-welldef} together with its equality case; (b) is
\Cref{thm:ceiling}.
\end{proof}

\begin{remark}[The two values are not comparable]\label{rem:notcomparable}
$\KL(P_1\|P_0)$ and $\KL(P_0\|P_1)$ are the same functional evaluated in
opposite orders, and neither dominates the other: for
$P_0=\mathrm{Bern}(\frac12)$ and $P_1=\mathrm{Bern}(0.99)$ one has
$\KL(P_0\|P_1)=1.6145$ and $\KL(P_1\|P_0)=0.6371$, a ratio of $2.53$, and
exchanging the two hypotheses reverses it. (We assume here, as in \Cref{thm:two}(b), that
$\KL(P_0\|P_1)<\infty$ and $P_0\ne P_1$. Without them both suprema can be
attained by the same \evar{} --- in the model of \Cref{rem:eqcase},
$P_0=\mathrm{Bern}(\frac12)$ and $P_1=\delta_1$, one has $\Lambda(R)=\infty=
\KL(P_0\|P_1)$, and $R$ maximises both scalars.) So neither ceiling dominates
the other, and which is the larger depends on the model. They may of course also
coincide: in the Gaussian location model of \Cref{ex:gauss} both equal
$\Delta^{2}/2$. What \Cref{thm:two} asserts is not that the two values always
differ, but that the two \emph{optimisation problems} do --- different
maximisers, and one supremum attained where the other is not.
\end{remark}

\subsection{A worked instance}

\begin{example}\label{ex:ceiling}
Let $\Xc=\{0,1\}$, $P_0=\mathrm{Bern}(\frac12)$, $P_1=\mathrm{Bern}(0.6)$, so
$R(1)=1.2$ and $R(0)=0.8$. Then
\[
   \KL(P_0\|P_1)=\tfrac12\log\tfrac{0.5}{0.6}+\tfrac12\log\tfrac{0.5}{0.4}
   =0.0204110 ,
\]
and $Z_\beta=\frac12(1.2^{\beta}+0.8^{\beta})$, $\lambda(\beta)=-\beta^{-1}\log Z_\beta$:
\begin{center}
\small
\begin{tabular}{@{}lcccccc@{}}
\toprule
$\beta$ & $1$ & $0.5$ & $0.25$ & $0.1$ & $0.05$ & $\downarrow0$\\
\midrule
$\lambda(\beta)$ & $0$ & $0.0101534$ & $0.0152756$ & $0.0183561$ & $0.0193835$
   & $0.0204110$\\
$\Lambda(R_\beta)$ & $0.0050768$ & $0.0114340$ & $0.0155965$ & $0.0184075$
   & $0.0193963$ & $0.0204110$\\
$\Lambda(R_\beta)$ as \% of $\KL(P_0\|P_1)$ & $24.9$ & $56.0$ & $76.4$ & $90.2$
   & $95.0$ & $100$\\
\bottomrule
\end{tabular}
\end{center}
\Cref{lem:flat} gives $\Lambda_s(R_\beta)=\lambda(\beta)$ at the paired tilt
$s=\frac{1-\beta}{\beta}$, hence $\Lambda(R_\beta)\ge\lambda(\beta)$; the second
row shows the inequality is strict, as \Cref{cor:chernoffinfo} explains. The
columns $\beta=0.1$ and $\beta=0.05$ exhibit \evars{} whose exponent is within
$10\%$, respectively $5\%$, of the ceiling, and by \Cref{prp:notattained} no
\evar{} reaches $0.0204110$. The last column is a limit, not an \evar{}.

The likelihood ratio itself, $E=R=R_1$, has
\[
   \Lambda(R)=\sup_{s\ge0}\lC-\log\lp0.6\cdot1.2^{-s}+0.4\cdot0.8^{-s}\rp\rC
   =0.0050768 ,
\]
attained at $s^{\ast}=0.5017$, which is $24.9\%$ of the ceiling. The tilt that
\Cref{thm:upper} pairs with this $s^{\ast}$ is
$\beta^{\ast}=\frac1{1+s^{\ast}}=0.6659$, and the bound it gives is
\[
   \Lambda(R)=\Lambda_{s^{\ast}}(R)=0.0050768\ <\ \lambda(\beta^{\ast})=0.0067677 ,
\]
strictly, as it must be by the equality case of \Cref{thm:upper}: $R$ is not
equal to $R^{\beta^{\ast}}/Z_{\beta^{\ast}}$. So the \evar{} of maximal
\epower{} --- $\E_{P_1}[\log R]=\KL(P_1\|P_0)=0.0201355$, the largest possible by
\Cref{lem:epower-welldef} --- realises barely a quarter of the largest possible
exponent, and is beaten already by the single flattening
$R_{\beta^{\ast}}$. The two scalars of \Cref{rem:two-scalars} are maximised by
different \evars{}.
\end{example}

\section{What to bet at a finite horizon}
\label{sec:horizon}

\noindent
Throughout this section we work in the i.i.d.\ model of \Cref{def:iid};
\Cref{lem:surrogate-gen}, \Cref{prp:exactfamily-gen}, \Cref{lem:floor-gen},
\Cref{cor:price-gen} and \Cref{rem:betastar-gen} are what survive without it.

\Cref{thm:ceiling} says the exponent ceiling is approached as $\beta\downarrow0$
and never attained. On its own that is not advice: it does not say which
$\beta$ to use at a horizon one can actually afford. This section supplies the
missing statement. Everything rests on one identity.

A word on the range of $\beta$. In \Cref{sec:ceiling} the tilt and the design
were tied together --- $\beta=\frac1{1+s}$ with $s\in[0,\infty)$ --- so
$\beta\in(0,1]$ there was the \emph{image} of the admissible tilt range and not
a restriction: a design $R_\beta$ with $\beta>1$ is simply never the maximiser
of $\Lambda_s$, at any $s>0$. The question here is a different one. We fix a
level and a horizon and ask which member of the family has the smallest
\emph{exact} type-II error; the tilt is no longer a free parameter of the
question, and nothing recommends stopping at $\beta=1$. Throughout this section
$\beta$ therefore ranges over the whole of $\Bc$ of \Cref{def:flattened}.

Two consequences are worth recording once rather than repeatedly. For
$\beta>1$ strict convexity of $\kappa:=\log Z$ with $\kappa(1)=0$ and
$\kappa'(1)=\KL(P_1\|P_0)>0$ gives $Z_\beta>1$, hence $\lambda(\beta)<0$, so the
surrogate bound of \Cref{lem:surrogate} is vacuous there. That is beside the
point: by \Cref{prp:exactfamily} the exact error is a nondecreasing function of
$B_t(\beta)$, so $B_t$ orders the family exactly whether or not $e^{B_t}$ falls
below $1$, and minimising it remains the right thing to do. And the extension is
not idle --- below the crossover of \Cref{thm:betastar}(a) the optimum genuinely
exceeds $1$.

\subsection{The two scalars along the family}

\begin{lemma}[The two scalars along the flattened family]\label{lem:family}
For $\beta\in\Bc$ let $R_\beta$ be as in \Cref{def:flattened} and put
$c_\beta:=\E_{P_1}[\log R_\beta]$. Then
\begin{equation}\label{eq:cbeta}
   c_\beta\ =\ \beta\lp\KL(P_1\|P_0)+\lambda(\beta)\rp ,
\end{equation}
and, assuming $\KL(P_0\|P_1)<\infty$ and $\log R$ not $P_0$-a.s.\ constant:
\begin{enumerate}[label=(\roman*),leftmargin=2.4em]
\item $\beta\mapsto c_\beta$ is strictly increasing on $(0,1)$ and strictly
      decreasing on $(1,\beta_{\max})$, with $c_{0^{+}}=0$ and maximum
      $c_1=\KL(P_1\|P_0)$;
\item $\beta\mapsto\lambda(\beta)$ is strictly decreasing on
      $\Bc$, with $\lambda(0^{+})=\KL(P_0\|P_1)$, $\lambda(1)=0$
      and $\lambda(\beta)<0$ for $\beta>1$;
\item $c_\beta/\beta\uparrow\KL(P_0\|P_1)+\KL(P_1\|P_0)$, the Jeffreys
      divergence, as $\beta\downarrow0$.
\end{enumerate}
\end{lemma}

\begin{proof}
Write $\kappa(\beta):=\log Z_\beta$, so $\lambda(\beta)=-\kappa(\beta)/\beta$.
Then $\log R_\beta=\beta\log R-\kappa(\beta)$, and taking $\E_{P_1}$ gives
$c_\beta=\beta\KL(P_1\|P_0)-\kappa(\beta)=\beta\KL(P_1\|P_0)+\beta\lambda(\beta)$,
which is \eqref{eq:cbeta}. The map $\kappa$ is the cumulant generating function
of $\log R$ under $P_0$, hence convex. Note first that $\KL(P_0\|P_1)<\infty$
forces $P_0(R>0)=1$, so $Z_\beta\to P_0(R>0)=1$ and $\kappa(0^{+})=0$; also
$\kappa(1)=\log\E_{P_0}[R]=0$. Thus $\kappa(0^{+})=\kappa(1)=0$,
$\kappa'(0^{+})=\E_{P_0}[\log R]=-\KL(P_0\|P_1)$ and
$\kappa'(1)=\E_{P_0}[R\log R]=\KL(P_1\|P_0)$, the latter a left derivative,
since $Z_\beta$ may be infinite for every $\beta>1$; the interchange
$\frac{d}{d\beta}\E_{P_0}[R^{\beta}]=\E_{P_0}[R^{\beta}\log R]$ at $\beta\uparrow1$
is legitimate by monotone convergence on $\{R\ge1\}$ (where
$\E_{P_0}[(R\log R)_+]<\infty$ by \Cref{ass:standing}) and dominated convergence
on $\{R<1\}$ (where $R^{\beta}|\log R|\le1/(e\beta_0)$ for $\beta\ge\beta_0>0$).
The value $\kappa'(0^{+})=-\KL(P_0\|P_1)$ is not the same interchange: it is
$\lim_{\beta\downarrow0}\kappa(\beta)/\beta=-\lim_{\beta\downarrow0}\lambda(\beta)$
by convexity, and \Cref{lem:limit}. Finally $\kappa$ is strictly convex because
$\log R$ is not $P_0$-a.s.\ constant \emph{and}, by the previous sentence,
$P_0(R>0)=1$, so that $\kappa''(\beta)=\Var_{P_\beta}(\log R)>0$ for the tilt
$dP_\beta\propto R^{\beta}dP_0$. For (i),
$\frac{d}{d\beta}c_\beta=\KL(P_1\|P_0)-\kappa'(\beta)$, and $\kappa'$ is
strictly increasing with $\kappa'(1)=\KL(P_1\|P_0)$, so the derivative is
strictly positive on $(0,1)$, vanishes at $\beta=1$ and is strictly negative on
$(1,\beta_{\max})$; the values are $c_1=\KL(P_1\|P_0)$ and $c_{0^{+}}=0$. For
(ii), $\lambda(\beta)=-\kappa(\beta)/\beta$ is minus the slope of the chord of
$\kappa$ from $0$ to $\beta$, which is strictly increasing on all of
$\Bc$ because $\kappa$ is strictly convex with $\kappa(0^{+})=0$;
the two values are \Cref{lem:limit} and $\kappa(1)=0$, and $\lambda(\beta)<0$
for $\beta>1$ because $\kappa(\beta)>0$ there. For
(iii), $c_\beta/\beta=\KL(P_1\|P_0)+\lambda(\beta)$ by \eqref{eq:cbeta}, and
$\lambda(\beta)$ increases to $\lambda(0^{+})=\KL(P_0\|P_1)$ as
$\beta\downarrow0$ by (ii).
\end{proof}

\begin{remark}[What the family shows]\label{rem:familypicture}
\Cref{lem:family} is the two ceilings in one curve. At $\beta=1$ the design is
$R$: \epower{} maximal, exponent $\lambda(1)=0$ from this bound. As $\beta$
decreases the \epower{} falls monotonically to $0$ while the tilt-$\beta$
exponent $\lambda(\beta)$ rises monotonically to $\KL(P_0\|P_1)$. No member of the family is good at both, and
\Cref{thm:two} shows that the two objectives are maximised at different places
and that one supremum is attained where the other is not, and \Cref{fig:two}
plots the trade-off along the family for the model of \Cref{ex:ceiling}. We do
not prove a quantitative trade-off valid for arbitrary \evars{}; along the
family it can be read off the figure.
\end{remark}

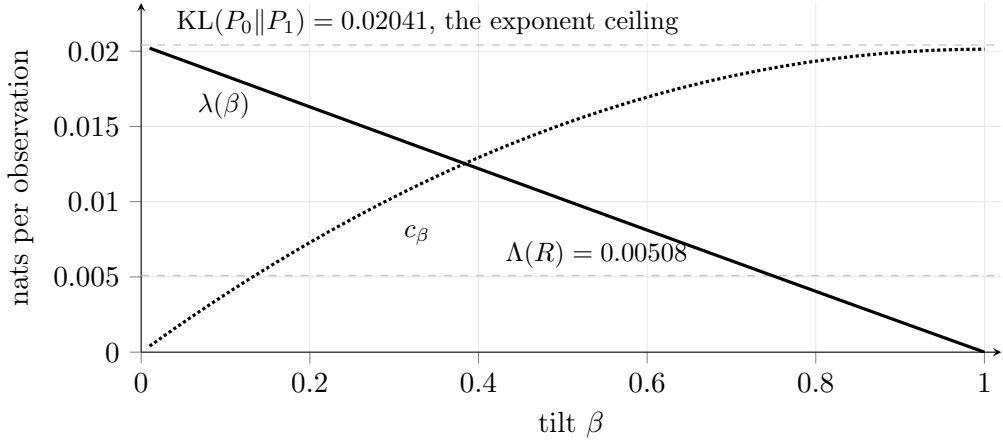
\begin{figure}[t]
\centering
\begin{tikzpicture}
\begin{axis}[
  width=0.82\textwidth, height=6.2cm,
  xlabel={tilt $\beta$}, ylabel={nats per observation},
  xmin=0, xmax=1.02, ymin=0, ymax=0.0232,
  axis lines=left, tick align=outside,
  grid=major, grid style={gray!18, very thin},
  every axis plot/.append style={very thick},
  scaled y ticks=false, yticklabel style={/pgf/number format/fixed,
     /pgf/number format/precision=3},
  xtick={0,0.2,0.4,0.6,0.8,1}, ytick={0,0.005,0.010,0.015,0.020},
]
\addplot[gray!55, dashed, thin, forget plot] coordinates {(0,0.020411)(1.02,0.020411)};
\addplot[gray!55, dashed, thin, forget plot] coordinates {(0,0.005077)(1.02,0.005077)};
\addplot[black] coordinates {
(0.010,0.020205)(0.040,0.019589)(0.070,0.018973)(0.100,0.018356)
(0.130,0.017740)(0.160,0.017124)(0.190,0.016507)(0.220,0.015891)
(0.250,0.015276)(0.280,0.014660)(0.310,0.014045)(0.340,0.013429)
(0.370,0.012815)(0.400,0.012200)(0.430,0.011586)(0.460,0.010972)
(0.490,0.010358)(0.520,0.009745)(0.550,0.009132)(0.580,0.008519)
(0.610,0.007907)(0.640,0.007296)(0.670,0.006684)(0.700,0.006074)
(0.730,0.005464)(0.760,0.004854)(0.790,0.004245)(0.820,0.003637)
(0.850,0.003029)(0.880,0.002422)(0.910,0.001815)(0.940,0.001210)
(0.970,0.000604)(1.000,0.000000)};
\addplot[black, densely dotted] coordinates {
(0.010,0.000403)(0.040,0.001589)(0.070,0.002738)(0.100,0.003849)
(0.130,0.004924)(0.160,0.005961)(0.190,0.006962)(0.220,0.007926)
(0.250,0.008853)(0.280,0.009743)(0.310,0.010596)(0.340,0.011412)
(0.370,0.012192)(0.400,0.012934)(0.430,0.013640)(0.460,0.014309)
(0.490,0.014942)(0.520,0.015538)(0.550,0.016097)(0.580,0.016620)
(0.610,0.017106)(0.640,0.017556)(0.670,0.017969)(0.700,0.018347)
(0.730,0.018687)(0.760,0.018992)(0.790,0.019261)(0.820,0.019493)
(0.850,0.019690)(0.880,0.019851)(0.910,0.019975)(0.940,0.020064)
(0.970,0.020118)(1.000,0.020136)};
\node[anchor=west, font=\small] at (axis cs:0.03,0.02195)
  {$\KL(P_0\|P_1)=0.02041$, the exponent ceiling};
\node[anchor=west, font=\small] at (axis cs:0.42,0.00650) {$\Lambda(R)=0.00508$};
\node[anchor=west, font=\small] at (axis cs:0.055,0.01640) {$\lambda(\beta)$};
\node[anchor=west, font=\small] at (axis cs:0.30,0.00780) {$c_\beta$};
\end{axis}
\end{tikzpicture}
\caption{The two scalars along the flattened family, plotted on $\beta\in(0,1]$,
for $P_0=\mathrm{Bern}(\frac12)$, $P_1=\mathrm{Bern}(0.6)$. Solid: the
tilt-$\beta$ ceiling $\lambda(\beta)=\Lambda_{(1-\beta)/\beta}(R_\beta)$, the
largest exponent available at tilt $\beta$ (\Cref{thm:upper},
\Cref{lem:flat}), rising to $\KL(P_0\|P_1)$ as $\beta\downarrow0$ but never
reaching it. Dotted: the \epower{} $c_\beta$ of the same design, rising in the
opposite direction to its own ceiling $\KL(P_1\|P_0)=0.02014$ at $\beta=1$. The
full exponent $\Lambda(R_\beta)$, optimised over $s$, is strictly larger than
$\lambda(\beta)$ (\Cref{cor:chernoffinfo}); at $\beta=1$ the solid curve is $0$
while $\Lambda(R)=0.00508$, the lower dashed line, which by
\eqref{eq:chernoffinfo} is the largest rectangle $\beta\times\lambda(\beta)$
under the solid curve --- a quarter of the ceiling here.}
\label{fig:two}
\end{figure}

\subsection{The flattened family consists of likelihood-ratio tests}

\begin{definition}[Log-likelihood-ratio sum and threshold]\label{def:GB}
Put
\[
   G_\bullet:\N_0\times\Omega\to[-\infty,\infty),\qquad
   G_t:=\sum_{n=1}^{t}\log R(X_n),\qquad G_0:=0,
\]
and
\[
   B_\bullet(\beta):\N\to\R,\qquad
   B_t(\beta):=\frac{\ell}{\beta}-t\,\lambda(\beta) .
\]
\end{definition}

\begin{lemma}[A surrogate bound at the tilt $s=\frac1\beta$]\label{lem:surrogate}
For every $\beta\in\Bc$, $\alpha\in(0,1)$ and $t\in\N$, the i.i.d.\ product
generated by $R_\beta$ satisfies $\gamma_t(\alpha)\le e^{B_t(\beta)}$.
\end{lemma}

\begin{proof}
$R_\beta^{-1/\beta}=Z_\beta^{1/\beta}R^{-1}$, and
$\E_{P_1}[R^{-1}]=\E_{P_0}[R\cdot R^{-1}]=P_0(R>0)\le1$, so
$\E_{P_1}[R_\beta^{-1/\beta}]\le Z_\beta^{1/\beta}=e^{-\lambda(\beta)}$, i.e.\
$\Lambda_{1/\beta}(R_\beta)\ge\lambda(\beta)$. \Cref{lem:chernoff} at
$s=\frac1\beta$ then gives
$\gamma_t(\alpha)\le\exp(\frac{\ell}{\beta}-t\lambda(\beta))=e^{B_t(\beta)}$.
\end{proof}

\begin{proposition}[Exact error of a flattened design]\label{prp:exactfamily}
Fix $\beta\in\Bc$ and $\alpha\in(0,1)$, and let $(W_t)_{t\in\N_0}$ be the
i.i.d.\ product generated by $R_\beta$. Then for every $t\in\N$:
\begin{enumerate}[label=(\alph*),leftmargin=2.4em]
\item $\log W_t=\beta\lp G_t-B_t(\beta)\rp+\ell$, and consequently
      \begin{equation}\label{eq:exactgamma}
         \gamma_t(\alpha)=\Prob_1\lp G_t\le B_t(\beta)\rp,
         \qquad
         \bar\gamma_t(\alpha)
            =\Prob_1\lp G_n\le B_n(\beta)\ \text{for all }n\le t\rp ;
      \end{equation}
\item the whole family $(R_\beta)_{\beta\in\Bc}$ therefore performs the
      \emph{same} test --- reject when the log-likelihood-ratio sum $G_t$
      exceeds a threshold --- and differs only in where the threshold is put;
\item $\gamma_t(\alpha)\ \ge\ \Prob_1\lp G_t\le\E_{P_1}[G_t]\rp$ whenever
      $t\,c_\beta\le\ell$ --- in particular at every $t$ when $c_\beta\le0$ ---
      with $c_\beta$ as in \Cref{lem:family}; and the
      surrogate bound $e^{B_t(\beta)}$ of \Cref{lem:surrogate} is vacuous
      unless $\lambda(\beta)>0$ --- that is, unless $\beta<1$ --- and
      $t>\ell/(\beta\lambda(\beta))$. (The Chernoff bound of
      \Cref{lem:chernoff} at $s=\frac1\beta$ is $e^{B_t(\beta)}P_0(R>0)^{t}$,
      which is smaller when $P_0(R>0)<1$; the two agree under
      $\KL(P_0\|P_1)<\infty$.)
\end{enumerate}
\end{proposition}

\begin{proof}
(a) $\log R_\beta=\beta\log R-\log Z_\beta$ and $-\log Z_\beta=\beta\lambda(\beta)$,
so $\log W_t=\beta G_t+t\beta\lambda(\beta)
=\beta\lp G_t-\frac{\ell}{\beta}+t\lambda(\beta)\rp+\ell
=\beta(G_t-B_t(\beta))+\ell$. Hence $W_t\le\frac1\alpha$, i.e.\
$\log W_t\le\ell$, holds precisely when $G_t\le B_t(\beta)$; the second identity
in \eqref{eq:exactgamma} follows by applying this at every $n\le t$ and using
$\{\tau_\alpha>t\}=\bigcap_{n\le t}\{W_n\le\frac1\alpha\}$. Part (b) restates
(a). For (c), $\E_{P_1}[G_t]=t\KL(P_1\|P_0)$, and by \eqref{eq:cbeta}
$B_t(\beta)\ge t\KL(P_1\|P_0)$ is equivalent to
$\ell\ge t\beta(\lambda(\beta)+\KL(P_1\|P_0))=t\,c_\beta$; monotonicity of
$x\mapsto\Prob_1(G_t\le x)$ then gives the stated inequality. The last claim is
$B_t(\beta)\ge0$, i.e.\ $t\lambda(\beta)\le\ell/\beta$.
\end{proof}

\subsection{The horizon-optimal tilt}

\begin{theorem}[Horizon-optimal flattening]\label{thm:betastar}
Assume $\KL(P_0\|P_1)<\infty$, $\KL(P_1\|P_0)<\infty$, and that $\log R$ is not
$P_0$-a.s.\ constant, and write
\[
   g(\beta)\ :=\ \beta\kappa'(\beta)-\kappa(\beta) ,
\]
which is strictly increasing on $\Bc$ with $g(0^{+})=0$ and
$g(1)=\KL(P_1\|P_0)$. For fixed $t$ the map $\beta\mapsto B_t(\beta)$ on
$\Bc$ is strictly decreasing and then, if $\ell/t$ lies in the
range of $g$, strictly increasing; in that case it has a unique minimiser
$\beta^{\ast}(t)$, characterised by
\begin{equation}\label{eq:betastar-root}
   g\lp\beta^{\ast}(t)\rp\ =\ \frac{\ell}{t} ,
\end{equation}
and otherwise the infimum is approached as $\beta\uparrow\beta_{\max}$. By
\Cref{prp:exactfamily}(a) the minimiser also minimises the exact type-II error
$\gamma_t(\alpha)$ over the family. Moreover:
\begin{enumerate}[label=(\alph*),leftmargin=2.4em]
\item \emph{(Crossover.)} Whenever $\beta^{\ast}(t)$ exists,
      $\beta^{\ast}(t)\gtreqless1$ according as
      $t\lesseqgtr\ell/\KL(P_1\|P_0)$. In particular the growth-optimal design
      $R$ is the horizon-optimal member of the family at exactly one, generally
      non-integer, horizon,
      \[
         t\ =\ \frac{\ell}{\KL(P_1\|P_0)} ;
      \]
      below it $R$ is beaten by \emph{sharpening} and above it by
      \emph{flattening}, with $\beta^{\ast}(t)\downarrow0$ as $t\to\infty$.
\item \emph{(Rate.)} Assume in addition that $\E_{P_0}[R^{-\delta}]<\infty$ for
      some $\delta>0$ --- equivalently, $Z_\beta<\infty$ for some $\beta<0$, so
      that $\kappa(\beta)=\log Z_\beta$ is finite, hence real-analytic, on a
      two-sided neighbourhood of $0$ --- and that
      $V_0:=\Var_{P_0}(\log R)=\kappa''(0)\in(0,\infty)$. Then
      \[
         \lambda(\beta)=\KL(P_0\|P_1)-\tfrac{V_0}{2}\,\beta+O(\beta^{2})
         \quad(\beta\downarrow0),
         \qquad
         \beta^{\ast}(t)=\sqrt{\frac{2\ell}{t\,V_0}}\,(1+o(1)) ,
      \]
      and, at that tilt,
      \begin{equation}\label{eq:capstone}
         \gamma_t(\alpha)\ \le\
         \exp\lp-t\,\KL(P_0\|P_1)+\sqrt{2\,\ell\,t\,V_0}+O(1)\rp
         \qquad(t\to\infty).
      \end{equation}
\end{enumerate}
\end{theorem}

\begin{proof}
Since $\lambda'(\beta)=-g(\beta)/\beta^{2}$,
\[
   B_t'(\beta)=-\frac{\ell}{\beta^{2}}-t\,\lambda'(\beta)
   =\frac{t\,g(\beta)-\ell}{\beta^{2}} .
\]
Now $g(0^{+})=0$ --- this is where $\KL(P_0\|P_1)<\infty$ is consumed, since it
gives both $\kappa(0^{+})=0$ (equivalently $P_0(R>0)=1$) and
$\beta\kappa'(\beta)\to0$ --- and $g'(\beta)=\beta\kappa''(\beta)>0$, since
$\kappa$ is strictly convex; so $g$ is strictly increasing from $0$, and
$g(1)=\kappa'(1)-\kappa(1)=\KL(P_1\|P_0)$. Hence $\beta\mapsto t\,g(\beta)-\ell$
changes sign at most once, from negative to positive. If $\ell/t$ lies in the
range of $g$ the sign change occurs at the $\beta$ of
\eqref{eq:betastar-root}, $B_t$ is strictly decreasing and then strictly
increasing, and with $B_t(0^{+})=+\infty$ the minimiser is unique; if not, $B_t$
is strictly decreasing throughout $\Bc$ and its infimum is approached as
$\beta\uparrow\beta_{\max}$. Since
$\Prob_1(G_t\le x)$ is nondecreasing in $x$, minimising $B_t$ minimises
$\gamma_t(\alpha)=\Prob_1(G_t\le B_t(\beta))$ over the family.

(a) Since $g$ is strictly increasing, $\beta^{\ast}(t)=g^{-1}(\ell/t)$ exceeds,
equals or falls below $1$ according as $\ell/t$ exceeds, equals or falls below
$g(1)=\KL(P_1\|P_0)$. As $g$ does not depend on $t$, the solution decreases in
$t$, tending to $0$ because $g(0^{+})=0$; and $\ell/t<g(\beta_{\max}^{-})$ holds
for all large $t$, so $\beta^{\ast}(t)$ exists there.

(b) The extra hypothesis makes $\kappa$ finite on $(-\delta,1]$, hence
real-analytic on $(-\delta,1)$, so its Taylor expansion at $0$ carries a genuine
$O(\beta^{3})$ remainder --- which is what the $O(1)$ below needs; with only
$\kappa''(0^{+})$ finite one would get an error $t\cdot o(\beta)=o(\sqrt t)$.
With $\kappa(\beta)=-\KL(P_0\|P_1)\beta+\frac{V_0}{2}\beta^{2}+O(\beta^{3})$
we get $\lambda(\beta)=-\kappa(\beta)/\beta=\KL(P_0\|P_1)-\frac{V_0}{2}\beta
+O(\beta^{2})$, and $g(\beta)=\beta\kappa'(\beta)-\kappa(\beta)
=\frac{V_0}{2}\beta^{2}+O(\beta^{3})$; solving $t\,g(\beta)=\ell$ gives
$\beta^{\ast}(t)=\sqrt{2\ell/(tV_0)}(1+o(1))$. For \eqref{eq:capstone} we do not need the asymptotics of $\beta^{\ast}(t)$ at
all: since $\beta^{\ast}(t)$ minimises $B_t$, it suffices to evaluate $B_t$ at
the explicit $b_t:=\sqrt{2\ell/(tV_0)}$, which lies in $(0,1]$ for
$t\ge2\ell/V_0$. From
$\lambda(\beta)=\KL(P_0\|P_1)-\frac{V_0}{2}\beta+O(\beta^{2})$ and
$t\,b_t^{2}=2\ell/V_0=O(1)$,
\[
   B_t(\beta^{\ast}(t))\ \le\ B_t(b_t)
   =\frac{\ell}{b_t}+\frac{tV_0}{2}b_t-t\KL(P_0\|P_1)+O(t\,b_t^{2})
   =-t\KL(P_0\|P_1)+\sqrt{2\ell tV_0}+O(1) ,
\]
because $\frac{\ell}{b_t}=\frac{tV_0}{2}b_t=\sqrt{\ell tV_0/2}$.
\Cref{lem:surrogate} gives $\gamma_t(\alpha)\le e^{B_t(\beta)}$ for every
$\beta$; take $\beta=\beta^{\ast}(t)$.
\end{proof}

\begin{remark}[Reading the crossover]\label{rem:crossover}
The threshold $\ell/\KL(P_1\|P_0)$ in \Cref{thm:betastar}(a) is exactly the
number of observations the likelihood ratio needs to accumulate $\ell=\log\frac1\alpha$
nats of evidence \emph{on average}, since $\E_{P_1}[\log R]=\KL(P_1\|P_0)$. So
within the family the growth-optimal design is the right one at exactly the
horizon at which it typically rejects, and is improved on to either side of it
--- strictly in the threshold
$B_t$, and hence strictly in the error whenever $\Prob_1$ charges the interval
$(B_t(\beta^{\ast}(t)),B_t(1)]$. When $\log R$ is lattice-valued the
improvement can lag the threshold: in the model of \Cref{ex:ceiling}, at
$\alpha=0.05$, the crossover is at $t=148.8$, but $G_t$ lives on a lattice of
spacing $\log\frac32$ and the exact errors of $R$ and of $R_{\beta^{\ast}(t)}$
coincide for $t=149,\dots,191$, the first strict improvement occurring at
$t=192$; for larger $t$ they continue to coincide at isolated horizons. That is
the
operational form of \Cref{thm:two}: which ceiling one should be aiming at
depends on how long one is prepared to wait.
\end{remark}

\begin{remark}[What sharpening is worth]\label{rem:sharpen}
Below the crossover the horizon-optimal design is a \emph{sharpened} likelihood
ratio, $\beta^{\ast}(t)>1$: one raises $R$ to a power above $1$, which lowers
both scalars of \Cref{thm:two} --- the \epower{} $c_\beta$ by
\Cref{lem:family}(i), and the exponent, since
$\Lambda(R_\beta)=\sup_{u\le1}\lC-(1-u)\kappa(\beta)/\beta-\kappa(u)\rC
\le\Lambda(R)$ because $\kappa(\beta)>0$ for $\beta>1$ --- and still reduces the
error, because at those horizons what matters is dispersion rather than drift.
The surrogate exponent $\lambda(\beta)$ itself turns negative there
(\Cref{lem:family}(ii)).
In the Gaussian model of \Cref{ex:gauss} at $\Delta=0.8$, $\alpha=0.05$
(crossover $\ell/D=9.4$), the optimum $\beta^{\ast}(t)=\sqrt{\ell/(tD)}$ gives
$\gamma_2=0.906$ against $0.981$ for $R$, and $\gamma_5=0.745$ against $0.782$;
the surrogate bound is vacuous at both horizons, $e^{B_2}=8.4$ and
$e^{B_5}=16.1$, yet $B_t$ ranks the two designs correctly, which is
\Cref{prp:exactfamily}(a) at work. Sharpening is not always available. If
$\ell/t\ge g(\beta_{\max}^{-})$ there is no minimiser at all and $B_t$ decreases
throughout $\Bc$; this happens at small $t$ whenever $g$ is bounded, as in
\Cref{ex:ceiling}, where $g(\infty)=\log2$ and no $\beta^{\ast}(t)$ exists for
$t\le\ell/\log2=4.3$. And if $\beta_{\max}=1$ --- $Z_\beta=\infty$ for every
$\beta>1$, as for heavy-tailed $R$ --- the family stops at $R$ and no sharpening
exists to be had. \Cref{fig:horizon} shows what the crossover is
worth in the Gaussian model.
\end{remark}

\begin{remark}[Is \eqref{eq:capstone} tight?]\label{rem:bahadurrao}
\eqref{eq:capstone} is an upper bound, and one may ask by how much it errs. In
the Gaussian model of \Cref{ex:gauss} the answer is exact. There
$G_t\sim\mathcal N(tD,2tD)$ under $P_1$, so \Cref{prp:exactfamily}(a) gives
$\gamma_t(\alpha)=\Phi(-z_t)$ with
$z_t=(tD-B_t(\beta^{\ast}(t)))/\sqrt{2tD}$, and the standard normal tail
estimate $-\log\Phi(-z)=\frac{z^{2}}{2}+\log\lp z\sqrt{2\pi}\rp+o(1)$ yields
\[
   -\log\gamma_t(\alpha)\ -\ \lp-B_t(\beta^{\ast}(t))\rp
   \ =\ \tfrac12\log t\ +\ O(1)\qquad(t\to\infty) .
\]
So the bound is off by exactly half a logarithm --- the Bahadur--Rao prefactor
\cite{BR60}, \cite[\S3.7]{DZ10} --- and \eqref{eq:capstone} is tight to
$O(\log t)$. Numerically, at $\Delta=0.8$ and $\alpha=0.05$ the difference minus
$\frac12\log t$ equals $3.356$, $3.592$, $3.661$, $3.682$, $3.689$, $3.691$ at
$t=10^{2},\dots,10^{7}$, converging to the predicted
$\ell+\frac12\log(4\pi D)=3.6915$.

We expect the same to hold generally: that under the hypotheses of
\Cref{thm:betastar}(b) and a non-lattice condition on $\log R$,
\[
   -\log\gamma_t(\alpha)\ =\ t\,\KL(P_0\|P_1)-\sqrt{2\ell tV_0}
      +\tfrac12\log t+O(1) ,
\]
by a Bahadur--Rao expansion of $\Prob_1(G_t\le B_t(\beta))$ along the
moving threshold. We do not prove it here; the moving threshold is what makes it
more than a citation, since $B_t(\beta^{\ast}(t))-t\KL(P_1\|P_0)$ grows like
a fixed multiple of $t$ rather than staying at a bounded deviation.
\end{remark}

\begin{figure}[t]
\centering
\begin{tikzpicture}
\begin{axis}[
  width=0.82\textwidth, height=6.4cm,
  xlabel={horizon $t$}, ylabel={$\log_{10}\gamma_t(\alpha)$},
  xmin=0, xmax=205, ymin=-19.5, ymax=1.5,
  axis lines=left, tick align=outside,
  grid=major, grid style={gray!18, very thin},
  every axis plot/.append style={very thick},
  xtick={0,50,100,150,200}, ytick={0,-4,-8,-12,-16},
]
\addplot[gray!60, dashed, thin, forget plot] coordinates {(9.36,-19.5)(9.36,0.5)};
\addplot[black] coordinates {
(2,-0.008)(6,-0.150)(10,-0.330)(14,-0.509)(18,-0.683)(22,-0.852)(26,-1.018)
(30,-1.181)(40,-1.580)(52,-2.045)(64,-2.502)(76,-2.953)(88,-3.399)(100,-3.841)
(112,-4.281)(124,-4.719)(136,-5.155)(148,-5.589)(160,-6.022)(172,-6.454)
(184,-6.885)(196,-7.315)(200,-7.458)};
\addplot[black, densely dotted] coordinates {
(2,-0.043)(6,-0.163)(10,-0.330)(14,-0.534)(18,-0.765)(22,-1.018)(26,-1.289)
(30,-1.576)(40,-2.347)(52,-3.349)(64,-4.412)(76,-5.523)(88,-6.672)(100,-7.851)
(112,-9.056)(124,-10.283)(136,-11.529)(148,-12.793)(160,-14.071)(172,-15.362)
(184,-16.666)(196,-17.981)(200,-18.421)};
\node[anchor=west, font=\small] at (axis cs:11,0.4)
  {$t=\ell/\KL(P_1\|P_0)=9.4$};
\node[anchor=west, font=\small] at (axis cs:150,-5.2) {$R$};
\node[anchor=east, font=\small] at (axis cs:150,-14.6) {$R_{\beta^{\ast}(t)}$};
\end{axis}
\end{tikzpicture}
\caption{What the horizon-optimal tilt buys, in the Gaussian location model of
\Cref{ex:gauss} with $\Delta=0.8$ and $\alpha=0.05$. Both curves are the
\emph{exact} type-II error, computed from \Cref{prp:exactfamily}(a) and the law
$G_t\sim\mathcal N(tD,2tD)$. Solid: the growth-optimal design $R$, whose
exponent is a quarter of $tD$. Dotted: the horizon-optimal member
$R_{\beta^{\ast}(t)}$ of the flattened family, with
$\beta^{\ast}(t)=\sqrt{\ell/(tD)}$. The two touch only at the crossover
$t=\ell/\KL(P_1\|P_0)=9.4$ (\Cref{thm:betastar}(a), dashed), the dotted curve
lying below by sharpening before it and by flattening after: at $t=50$ the errors are $1.1\cdot10^{-2}$ and $6.7\cdot10^{-4}$, at
$t=200$ they are $3.5\cdot10^{-8}$ and $3.8\cdot10^{-19}$. The gap widens
without bound, since the two exponents are $\frac14tD$ and $tD(1-\beta^{\ast}(t))^{2}$.}
\label{fig:horizon}
\end{figure}
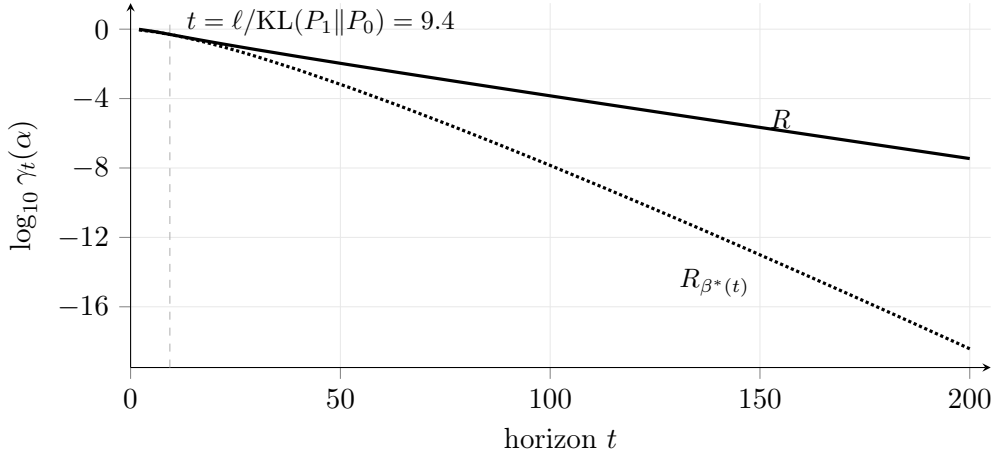

\subsection{The price of the ceiling}

Approaching the ceiling requires $\beta\downarrow0$, and \Cref{lem:family} shows
that drives the \epower{} to $0$. The following lemma converts that into a
horizon before which the test has essentially no power, for any i.i.d.\ design of
finite log-variance.

\begin{lemma}[A universal horizon floor]\label{lem:floor}
Let $E\in\Ev$ with $c:=\E_{P_1}[\log E]>0$ and $v:=\Var_{P_1}(\log E)<\infty$,
let $(W_t)_{t\in\N_0}$ be the i.i.d.\ product it generates, and suppose
$c\,\ell>2v$. Then
\[
   \bar\gamma_t(\alpha)\ \ge\ 1-\frac{2v}{c\,\ell}
   \qquad\text{for every }t\le\frac{\ell}{2c} .
\]
\end{lemma}

\begin{proof}
Write $S_n:=\log W_n=\sum_{m\le n}\log E(X_m)$ and $M_n:=S_n-nc$, so
$(M_n)_{n\in\N_0}$ is a mean-zero random walk under $\Prob_1$ with
$\Var_{\Prob_1}(M_n)=nv$. Fix $t\le\ell/(2c)$. For $n\le t$ we have
$nc\le\ell/2$, so $\{S_n>\ell\}\subseteq\{M_n>\ell/2\}$ and hence
\[
   \Prob_1\lp\tau_\alpha\le t\rp
   =\Prob_1\lp\max_{n\le t}S_n>\ell\rp
   \ \le\ \Prob_1\lp\max_{n\le t}|M_n|\ge\tfrac\ell2\rp
   \ \le\ \frac{tv}{(\ell/2)^{2}}\ \le\ \frac{2v}{c\,\ell} ,
\]
by Kolmogorov's maximal inequality (\Cref{thm:app-kolmogorov},
\cite[Thm.~22.4]{Bil95}) and then $t\le\ell/(2c)$. The hypothesis $c\ell>2v$ is
what makes the conclusion non-trivial.
\end{proof}

\begin{corollary}[The cost of a near-optimal exponent]\label{cor:price}
Assume $\Var_{P_1}(\log R)<\infty$ and apply \Cref{lem:floor} to $E=R_\beta$,
with $v_\beta=\beta^{2}\Var_{P_1}(\log R)$ and $c_\beta$ as in \eqref{eq:cbeta}.
Whenever $c_\beta\ell>2v_\beta$, the design $R_\beta$ has power at most
$2v_\beta/(c_\beta\ell)=O(\beta)$ at every horizon $t\le\ell/(2c_\beta)$, and
$\ell/(2c_\beta)\sim\ell/\lp2\beta(\KL(P_0\|P_1)+\KL(P_1\|P_0))\rp$ diverges
like $\beta^{-1}$ as $\beta\downarrow0$.
\end{corollary}

\begin{remark}[The two obstructions together]\label{rem:price}
\Cref{cor:price} is the counterweight to \Cref{thm:ceiling}. Buying an exponent
within $\epsilon$ of the ceiling costs a tilt of order $\epsilon$
(\Cref{thm:betastar}(b)) and therefore a horizon of order $\epsilon^{-1}$
before anything happens at all. In the model of \Cref{ex:ceiling}, at
$\beta=0.1$ --- an exponent $90\%$ of the ceiling --- the floor gives
$\bar\gamma_t(0.05)\ge0.9315$ for every $t\le389$, and the Chernoff bound
$e^{B_t(\beta)}$ stays above $1$ until $t>1632$. The ceiling of
\Cref{thm:ceiling} is an asymptotic statement, and \Cref{cor:price} says how
asymptotic.
\end{remark}

\subsection{The Gaussian model in closed form}

\begin{example}[Gaussian location]\label{ex:gauss}
Let $\Xc=\R$, $P_0=\mathcal N(0,1)$ and $P_1=\mathcal N(\Delta,1)$ with
$\Delta>0$, so $R(x)=e^{\Delta x-\Delta^{2}/2}$ and
$\KL(P_0\|P_1)=\KL(P_1\|P_0)=\frac{\Delta^{2}}2=:D$. Every quantity of this
paper is elementary here:
\[
   Z_\beta=e^{-\beta(1-\beta)\Delta^{2}/2},\qquad
   \lambda(\beta)=(1-\beta)\,D,\qquad
   c_\beta=\beta(2-\beta)\,D,\qquad
   V_0=\Delta^{2}=2D .
\]
Thus the fraction of the ceiling available at tilt $\beta$ is exactly $1-\beta$.
A direct computation gives moreover
\[
   \Lambda(R_\beta)=D\lp1-\tfrac\beta2\rp^{2}\ \ (0<\beta<2),
   \qquad
   \Lambda(R_\beta)=0\ \ (\beta\ge2),
   \qquad\text{so}\qquad
   \Lambda(R)=\Lambda(R_1)=\tfrac D4 :
\]
the growth-optimal \evar{} realises exactly one quarter of the ceiling, its
exponent being the Chernoff information of the pair. The crossover of
\Cref{thm:betastar}(a) is at $t=\ell/D$, and for $t\ge\ell/D$ the
horizon-optimal tilt is $\beta^{\ast}(t)=\sqrt{\ell/(tD)}$ exactly, not merely
to leading order; here $g(\beta)=D\beta^{2}$, so the same formula gives
$\beta^{\ast}(t)$ at every $t$, and returns a value above $1$ for $t<\ell/D$
(\Cref{rem:sharpen}).

Here $G_t\sim\mathcal N(tD,2tD)$ under $P_1$, so \Cref{prp:exactfamily}(a)
makes the type-II error of every $R_\beta$ an explicit normal tail, with
$-\log\gamma_t(\alpha)=\frac{(tD-B_t(\beta))^{2}}{4tD}(1+o(1))$. The two rows
below record that leading term, divided by $tD$: it is $(1-\beta^{\ast}(t))^{2}$
at the horizon-optimal tilt and $\frac14(1-\frac{\ell}{tD})^{2}$ at $\beta=1$.
(The exact $-\log\gamma_t(\alpha)/(tD)$ is larger at every finite $t$ --- at
$t=50$ it is $45.7\%$ and $28.3\%$ --- because the normal tail carries a
logarithmic prefactor; the two entries below are on one basis and are the
quantities that converge to the exponents.) Numerically, at
$\Delta=0.8$ (so $D=0.32$) and $\alpha=0.05$:

\begin{center}
\small
\begin{tabular}{@{}lrrrrr@{}}
\toprule
horizon $t$ & $50$ & $200$ & $1\,000$ & $5\,000$ & $20\,000$\\
\midrule
$\beta^{\ast}(t)$ & $0.433$ & $0.216$ & $0.097$ & $0.043$ & $0.022$\\
$(1-\beta^{\ast}(t))^{2}$, for $R_{\beta^{\ast}(t)}$
   & $32\%$ & $61\%$ & $82\%$ & $92\%$ & $96\%$\\
$\frac14(1-\frac{\ell}{tD})^{2}$, for $R$
   & $17\%$ & $23\%$ & $24.5\%$ & $24.9\%$ & $25.0\%$\\
\bottomrule
\end{tabular}
\end{center}

\noindent
The crossover $\ell/D=9.4$ is early here, and past it the horizon-optimal
design beats the growth-optimal one at every horizon: the exponent of $R$ tends
to a quarter of $tD$, while the exponent at $\beta^{\ast}(t)$ tends to $tD$
itself. Because $\log R$ has a density here, the improvement is strict at
every $t>\ell/D$.
\end{example}

\subsection{Adapting the tilt to the wealth}
\label{sec:adaptive}

\Cref{thm:betastar} fixes one $\beta$ for the whole horizon. Nothing in
\Cref{def:cond-evar} requires that: a predictable $\beta$ is admissible, and the
horizon-optimal tilt of a \emph{residual} problem is a natural choice for it.

\begin{definition}[Rolling-horizon tilt]\label{def:rolling}
Fix a horizon $T\in\N$. For $t\le T$ put
\[
   \ell_{t-1}:=\lp\ell-\log W_{t-1}\rp_+ ,
   \qquad m_t:=T-t+1 ,
\]
the \emph{residual threshold} and the number of remaining bets, and define the
predictable tilt
\begin{equation}\label{eq:rolling}
   \beta_{t-1}\ :=\ \argmin_{\beta\in\Bc}\
   \frac1\beta\lB\frac{\ell_{t-1}}{m_t}+\kappa(\beta)\rB .
\end{equation}
The associated design bets $E_t:=R_{\beta_{t-1}}$ at step $t$.
\end{definition}

\noindent
The bracket in \eqref{eq:rolling} is $B_{m_t}(\beta)/m_t$ of \Cref{def:GB} with
$\ell$ replaced by $\ell_{t-1}$ and $t$ by $m_t$: one solves the problem of
\Cref{thm:betastar} afresh at each step, for the threshold still to be cleared
and the time still available.

\begin{lemma}[The rolling-horizon tilt]\label{lem:rolling}
Let $g(\beta):=\beta\kappa'(\beta)-\kappa(\beta)$ on $\Bc$. Under
the hypotheses of \Cref{thm:betastar}:
\begin{enumerate}[label=(\alph*),leftmargin=2.4em]
\item $g$ is strictly increasing, with $g(0^{+})=0$ and
      $g(1)=\KL(P_1\|P_0)$, and, provided $\ell_{t-1}/m_t<g(\beta_{\max}^{-})$,
      the minimiser in \eqref{eq:rolling} is the unique solution of
      \begin{equation}\label{eq:rolling-root}
         g\lp\beta_{t-1}\rp\ =\ \frac{\ell_{t-1}}{m_t} ,
      \end{equation}
      while if $\ell_{t-1}/m_t\ge g(\beta_{\max}^{-})$ the infimum is approached
      as $\beta\uparrow\beta_{\max}$ and we set $\beta_{t-1}$ to the largest
      computationally admissible value;
      with the convention $\beta_{t-1}:=0$, that is $E_t:=1$, when
      $\ell_{t-1}=0$ (there the infimum is approached as $\beta\downarrow0$ and
      not attained);
\item $\beta_{t-1}$ is $\Fc_{t-1}$-measurable, so $(E_t)_{t\le T}$ is a betting
      sequence and \Cref{thm:typeI} applies unchanged;
\item $\beta_{t-1}$ is nonincreasing in $W_{t-1}$, and nondecreasing in $\ell$
      and --- at fixed $\ell_{t-1}$ --- in $t$: one bets harder when behind or
      when time runs short, and coasts once ahead;
\item $\beta_{t-1}\ge1$ if and only if
      $\ell_{t-1}/m_t\ge\KL(P_1\|P_0)$.
\end{enumerate}
\end{lemma}

\begin{proof}
(a) Writing the bracket as $\frac{\ell_{t-1}/m_t+\kappa(\beta)}{\beta}$, its
derivative in $\beta$ has the sign of $\beta\kappa'(\beta)-\kappa(\beta)
-\ell_{t-1}/m_t=g(\beta)-\ell_{t-1}/m_t$; and $g'(\beta)=\beta\kappa''(\beta)>0$
with $g(0^{+})=-\kappa(0^{+})=0$, $g(1)=\kappa'(1)-\kappa(1)=\KL(P_1\|P_0)$, as
in the proof of \Cref{lem:family}. (b) $W_{t-1}$ is $\Fc_{t-1}$-measurable and
$m_t$ is deterministic. (c) $\ell_{t-1}$ is nonincreasing in $W_{t-1}$ and nondecreasing in $\ell$,
and $m_t$ is decreasing in $t$; apply $g^{-1}$, which is increasing. (d) is
\eqref{eq:rolling-root} with $g(1)=\KL(P_1\|P_0)$.
\end{proof}

\begin{remark}[Computing $\beta_{t-1}$]\label{rem:rolling-newton}
Only \eqref{eq:rolling-root} has to be solved, and $g$ is increasing, so
bisection always works and Newton's method converges quickly from a bisection
safeguard, using
\[
   \kappa'(\beta)=\frac{\E_{P_0}\lB R^{\beta}\log R\rB}{Z_\beta} ,
   \qquad
   \kappa''(\beta)=\frac{\E_{P_0}\lB R^{\beta}\log^{2}R\rB}{Z_\beta}
      -\kappa'(\beta)^{2} ,
   \qquad
   g'(\beta)=\beta\,\kappa''(\beta) ,
\]
so that the Newton step is
$\beta\mapsto\beta-\lp g(\beta)-\ell_{t-1}/m_t\rp/\lp\beta\kappa''(\beta)\rp$.
When $\kappa$ is not available in closed form the three expectations are over
the \emph{known} null $P_0$, so they are a quadrature or Monte-Carlo problem and
not a statistical one; a single sample from $P_0$, drawn once and reused, gives
plug-in values at every $\beta$ and every step. In the Gaussian model of
\Cref{ex:gauss}, $g(\beta)=D\beta^{2}$ with $D=\Delta^{2}/2$ and
\eqref{eq:rolling-root} is explicit,
$\beta_{t-1}=\sqrt{\ell_{t-1}/(m_tD)}$.
\end{remark}

\begin{remark}[What it buys]\label{rem:rolling-numbers}
\Cref{def:rolling} minimises a \emph{bound} on the residual error, not the error
itself, so it is a heuristic and not an optimum: the exactly optimal predictable
rule is the value function of a finite-horizon control problem with state
$(t,\log W_{t-1})$, which has no closed form. Numerically it recovers almost all
of what that control problem achieves. In the Gaussian model of
\Cref{ex:gauss} with $\Delta=1$ and $\alpha=0.05$, computed on a grid:
\begin{center}
\small
\begin{tabular}{@{}lrrr@{}}
\toprule
horizon $T$ & $20$ & $50$ & $100$\\
\midrule
best constant $\beta$, from \Cref{thm:betastar}
   & $2.15\cdot10^{-2}$ & $1.89\cdot10^{-6}$ & $2.14\cdot10^{-14}$\\
rolling-horizon tilt \eqref{eq:rolling}
   & $7.12\cdot10^{-3}$ & $1.53\cdot10^{-7}$ & $3.61\cdot10^{-16}$\\
optimal predictable rule
   & $6.46\cdot10^{-3}$ & $1.10\cdot10^{-7}$ & $1.67\cdot10^{-16}$\\
\midrule
\multicolumn{4}{@{}l@{}}{\footnotesize all three restricted to
   $\beta\in(0,1]$; \Cref{rem:rolling-beyond1} lifts the cap}\\
\midrule
share of the attainable gain & $95.6\%$ & $97.6\%$ & $99.1\%$\\
\bottomrule
\end{tabular}
\end{center}
\end{remark}

\begin{remark}[Why $\beta>1$ is allowed]\label{rem:rolling-beyond1}
By \Cref{lem:rolling}(d) the rule takes $\beta_{t-1}>1$ exactly when the
residual gap per remaining bet exceeds the largest attainable \epower{}. This is
the predictable counterpart of \Cref{rem:sharpen}, and it is strictly stronger:
a constant $\beta>1$ helps only below the crossover, whereas the rule reaches
for one at any horizon once the wealth has fallen far enough behind. In the
model of \Cref{rem:rolling-numbers}, capping \eqref{eq:rolling} at $\beta=1$
multiplies the error by a factor $1.10$ to $1.11$ at every horizon computed
($T=5,10,20,50$), and the optimal predictable rule uses $\beta>1$ as well ---
at $T=20$ it improves from $6.46\cdot10^{-3}$ to $5.93\cdot10^{-3}$ once
$\beta\le2$ is allowed. The reason is that the increment of $\log W$ under
$\Prob_1$ has mean $c_\beta$, concave in $\beta$ and decreasing past $\beta=1$
(\Cref{lem:family}(i)), while its standard deviation stays exactly proportional
to $\beta$; so raising $\beta$ past $1$ buys dispersion at the cost of drift ---
which is what a design that is behind with few bets left should want. It is the
same point as \Cref{thm:nfl}, in the opposite direction: a design can be
improved by lowering its \epower{}.
\end{remark}

\begin{remark}[What is not claimed]\label{rem:rolling-limits}
Three limitations. The rule is horizon-specific: $T$ is fixed in advance and the
target is $\gamma_T$; for $\bar\gamma_T$ the residual problem acquires an
absorbing region above the threshold and \eqref{eq:rolling} is no longer the
right surrogate. The exact identity of \Cref{prp:exactfamily} is lost, since
$\log W_t=\beta G_t-t\kappa(\beta)$ requires one common $\beta$; the error must
be evaluated numerically. And the \emph{certified} exponent does not improve:
\Cref{cor:ceiling} binds every betting sequence, so no predictable choice of
$\beta_{t-1}$ can raise $\Lambda^{(t)}$ above the ceiling. The gain is in the
constant, not in the rate.
\end{remark}

\section{Discussion}
\label{sec:discussion}

\begin{remark}[Simple against simple, and the oracle]\label{rem:oracle}
Both hypotheses are simple and both laws are known throughout. This matters for
reading \Cref{thm:ceiling} and \Cref{thm:betastar}: the ceiling
$\KL(P_0\|P_1)$ is an oracle quantity, and the designs that approach it,
$R_\beta=R^{\beta}/Z_\beta$, are built from $R=\frac{dP_1}{dP_0}$ and so require
$P_1$. Not needing $P_1$ is precisely the motivation for growth-rate-optimal
\evars{} in a composite alternative \cite{GHK24,LRR25}. \Cref{thm:two} is a
statement about a simple alternative and does not by itself settle what happens
there; what it does show is that, already for a simple alternative, the scalar
those constructions maximise is not the one that governs the type-II error. What the composite analogue of \Cref{thm:betastar} looks like ---
how much of $\inf_{P_1\in\Hc_1}\KL(P_0\|P_1)$ is available to a design that does
not know which $P_1$ obtains --- is not addressed here.
\end{remark}

\begin{remark}[Products of conditional \evars{} versus \eproc{}es]
\label{rem:eproc}
\Cref{def:wealth} builds $(W_t)_{t\in\N_0}$ as a product of conditional
\evars{}. Some of the results extend beyond that class and some do not.
\Cref{prp:stein-ceiling} uses only that rejecting when $\tau_\alpha\le t$ is a
level-$\alpha$ test, so it holds for any nonnegative $\Prob_0$-supermartingale
and indeed for any \eproc{} --- a nonnegative process dominated by a
supermartingale under every null, in the sense of
\cite{RRLK23,RGVS23}. \Cref{prp:upper-seq}\ref{it:agg-exact}--\ref{it:agg-det}
extend in a different direction --- they need no product structure at all, only
a filtration (\Cref{rem:whatgeneralises}) --- while the equality analysis of
\Cref{prp:upper-seq}\ref{it:agg-iid} does use it, through the sections
$E_n^{y}$ of \eqref{eq:fubini}.
\end{remark}

\begin{remark}[Relation to sequential analysis]\label{rem:sprt}
\Cref{prp:exactfamily} says that the flattened family performs one
likelihood-ratio test at a moving threshold, which places \Cref{sec:horizon}
squarely in the territory of Wald's sequential probability ratio test
\cite{Wal45}: there too the statistic is $G_t$, there too the two error
probabilities are governed by $\KL(P_1\|P_0)$ and $\KL(P_0\|P_1)$ in opposite
roles, and there too the asymmetry between the two divergences is the source of
the design trade-off. What is different here is the constraint. The SPRT
chooses two thresholds and stops between them; a \etm{} is constrained to have
$\E_{\Prob_0}[W_t\mid\Fc_{t-1}]\le W_{t-1}$, which buys type-I validity at every
stopping time simultaneously and fixes the upper threshold at $\frac1\alpha$.
\Cref{thm:ceiling} is the price of that constraint, and \Cref{thm:betastar} says
what remains free once it is imposed.

A second point of contact is with how the safe-testing literature sizes
experiments. Sample-size rules there are typically driven by the \epower{}, via
$t\approx\ell/\E_{P_1}[\log E]$ --- which \Cref{prp:exactfamily}(c) shows is
exactly the horizon at which the threshold $B_t(\beta)$ crosses the mean of
$G_t$, so the heuristic is locating the right horizon for the wrong reason.
\Cref{thm:nfl} says the heuristic cannot be turned into a guarantee uniform over
testing problems, and \Cref{lem:floor} says what a variance hypothesis buys
instead.
\end{remark}

\begin{remark}[What is bounded, and what is not]\label{rem:scopefinal}
The finite-horizon content of this paper is \eqref{eq:matched-error},
\Cref{prp:exactfamily} and \Cref{cor:price}; \Cref{thm:ceiling} itself is a
statement about an exponent, and the rate at which that exponent is approached
is quantified only to the order of the second term. The exact second-order
behaviour --- the constant in the $O(1)$ of \eqref{eq:capstone}, the
Bahadur--Rao prefactor \cite{BR60}, and the matching converse --- is not
treated here.
Nor is misspecification. The upper bounds \Cref{thm:upper} and
\Cref{prp:upper-seq} hold for an arbitrary \evar{}, so they are untouched by
it; what a misspecified $R$ costs is the \emph{attainment} half, since the
designs $R_\beta$ are built from the true likelihood ratio.

Nor, in the general filtered setting, is the ceiling known to be sharp. The
upper bound \Cref{cor:ceiling} holds there with no structural hypothesis at all,
and \Cref{lem:flat-cond} supplies a witness step by step; what is missing is a
reason for the per-step values to aggregate, and \Cref{ass:homog} --- that the
conditional divergences be non-random --- is a sufficient condition rather than
a characterisation. Whether the essential suprema in \eqref{eq:ceiling} can be
strict is open, as is the exponent of the sequential error, for which
\Cref{prp:exponent} claims only a $\liminf$.

Finally, this paper says what the exponent \emph{can} be, not what a given
design can be \emph{certified} to achieve. The two questions are independent,
and the second is the subject of the companion paper \cite{LADDER}, which
converts hypotheses on the lower tail of $\log E_n$ into explicit
finite-horizon bounds, under an arbitrary filtration and with no product
structure. Composite hypotheses, which \Cref{rem:oracle} leaves open here, are
treated there.
\end{remark}

\phantomsection
\addcontentsline{toc}{section}{Acknowledgements}
\section*{Acknowledgements}

This note was prepared with the assistance of Claude Opus 5, a large language
model developed by Anthropic. Based on the author's notes, ideas and input, the
model was used to re-draft and restructure the exposition, to prepare the
typescript, to search for and cross-check references, to check and complete
arguments, and to provide feedback, which the author used to refine further
inputs. All references, statements and proofs have been verified by the author,
who takes full responsibility for the content, including any remaining errors.

\phantomsection

\appendix

\section{Beyond the i.i.d.\ model}
\label{app:general}

This appendix collects the statements that replace the i.i.d.-model results of
the main text, together with the hypothesis each one needs. Throughout,
\Cref{def:setting}, \Cref{ass:standing}, \Cref{def:ratios}, \Cref{def:condZ} and
\Cref{lem:Zf-cond} are in force; the i.i.d.\ model is \emph{not} assumed.
\Cref{tab:general} is the index.

\begin{table}[htbp]
\centering
\small
\begin{tabular}{@{}lll@{}}
\toprule
Main text & Replaced by & Extra hypothesis\\
\midrule
\Cref{lem:epower-welldef}, \Cref{thm:two}(a)
   & \Cref{lem:epower-cond} & none\\
\Cref{def:flattened}, \Cref{lem:flat}
   & \Cref{lem:flat-cond} & none\\
equality case of \Cref{thm:upper}
   & \Cref{prp:equality-cond} & none\\
\Cref{cor:chernoffinfo} (identity)
   & \Cref{cor:chernoffinfo-cond} & $\Prob_0(R_n>0\mid\Fc_{n-1})=1$\\
\Cref{lem:iidexp}, \Cref{cor:lower}, \Cref{thm:ceiling}, \Cref{thm:two}(b)
   & \Cref{thm:ceiling-gen} & \Cref{ass:homog}\\
\Cref{lem:limit}
   & \Cref{lem:limit-cond} & none\\
\Cref{lem:family}
   & \Cref{lem:family-cond} & none for \eqref{eq:cbeta-cond}\\
\Cref{prp:exactfamily}, \Cref{lem:surrogate}
   & \Cref{prp:exactfamily-gen}, \Cref{lem:surrogate-gen} & \Cref{ass:homog}\\
\Cref{lem:floor}
   & \Cref{lem:floor-gen} & none\\
\Cref{cor:price}
   & \Cref{cor:price-gen} & \Cref{ass:homog}, bounded conditional variance,
     $J<\infty$\\
\Cref{thm:betastar} (partially)
   & \Cref{rem:betastar-gen} & \Cref{ass:homog}, $\lambda_n$ constant in $n$\\
\Cref{prp:stein-ceiling}
   & \Cref{prp:stein-gen} & \Cref{ass:spectrum}\\
\Cref{prp:notattained}
   & \Cref{thm:ceiling-gen} & \Cref{ass:homog}, $\lambda_n(0^{+})<\infty$,
     non-degeneracy\\
\Cref{prp:exponent}
   & \Cref{prp:exponent-gen} & \Cref{ass:ge}, $\Psi'(0)>0$, interiority\\
\Cref{prp:powerone}
   & \Cref{prp:powerone-gen} & conditional variances bounded\\
\bottomrule
\end{tabular}
\caption{What each i.i.d.-model result becomes, and what it costs.}
\label{tab:general}
\end{table}

\noindent
\Cref{lem:attained} is not replaced but reused: its proof uses only
$\Prob_0|_{\Fc_t}\sim\Prob_1|_{\Fc_t}$ and $\Prob_0(W_t\ne1)>0$, and
\Cref{thm:ceiling-gen} derives the first of these. Left without a general form:
\Cref{thm:betastar}, for the reason given in \Cref{rem:betastar-gen}; the
strictness clause of \Cref{cor:chernoffinfo}; and
\Cref{prp:exactfamily}(b)--(c).

\subsection{Free of any extra hypothesis}

\begin{definition}[Conditional divergences]\label{def:kl-cond}
For $t\in\N$ put
$\KL_t:=\E_{\Prob_1}\lB\log R_t\mid\Fc_{t-1}\rB\in[0,\infty]$.
By the chain rule and \Cref{ass:standing},
$\sum_{n\le t}\E_{\Prob_1}[\KL_n]=\KL(\Prob_1|_{\Fc_t}\|\Prob_0|_{\Fc_t})<\infty$,
so $\KL_t<\infty$ $\Prob_1$-a.s. In the i.i.d.\ model $\KL_t=\KL(P_1\|P_0)$
$\Prob_1$-a.s.
\end{definition}

\begin{lemma}[Conditional \epower{} ceiling]\label{lem:epower-cond}
Let $(E_n)_{n\in\N}$ be a betting sequence and $n\in\N$. Then $\Prob_1$-a.s.
\[
   \E_{\Prob_1}\lB(\log E_n)_+\mid\Fc_{n-1}\rB\ \le\ \KL_n+1 ,
   \qquad
   \E_{\Prob_1}\lB\log E_n\mid\Fc_{n-1}\rB\ \le\ \KL_n ,
\]
with equality in the second bound if and only if $E_n=R_n$ $\Prob_1$-a.s.
\end{lemma}

\begin{proof}
By \Cref{lem:supmg}, $E_n<\infty$ $\Prob_0$-a.s., so $(\log E_n)_+\in\R$ on a
$\Prob_0$-full set. Young's inequality $xy\le x\log x-x+e^{y}$ ($x\ge0$,
$y\in\R$) applied there at $x=R_n$, $y=(\log E_n)_+$, under $\E_{\Prob_0}[\,\cdot\mid\Fc_{n-1}]$ and with
\Cref{lem:bayes}, gives
$\E_{\Prob_1}[(\log E_n)_+\mid\Fc_{n-1}]\le(\KL_n-1)+
\E_{\Prob_0}[\max(1,E_n)\mid\Fc_{n-1}]\le\KL_n+1$, using
$\max(1,E_n)\le1+E_n$ and \Cref{def:cond-evar}. For the second bound,
$R_n>0$ $\Prob_1$-a.s.\ by \Cref{lem:bayes}, and conditional Jensen for the
concave $\log$ gives
\[
   \E_{\Prob_1}\lB\log\tfrac{E_n}{R_n}\mid\Fc_{n-1}\rB
   \ \le\ \log\E_{\Prob_1}\lB\tfrac{E_n}{R_n}\mid\Fc_{n-1}\rB
   \ =\ \log\E_{\Prob_0}\lB E_n\mathbf 1_{\{R_n>0\}}\mid\Fc_{n-1}\rB
   \ \le\ 0 .
\]
Equality forces, by strict concavity, $E_n/R_n$ to be $\Prob_1$-a.s.\ equal to an
$\Fc_{n-1}$-measurable $\varkappa$, and then
$\E_{\Prob_0}[E_n\mathbf 1_{\{R_n>0\}}\mid\Fc_{n-1}]=1$ gives
$\varkappa\,\E_{\Prob_0}[R_n\mid\Fc_{n-1}]=\varkappa=1$.
\end{proof}

\begin{lemma}[The limit of $\lambda_n$ at $0$]\label{lem:limit-cond}
For every $n\in\N$, $\Prob_0$-a.s.,
\[
   \lambda_n(0^{+})\ =\ \KLrev_n\ :=\ -\E_{\Prob_0}\lB\log R_n\mid\Fc_{n-1}\rB
   \ \in[0,\infty] ,
\]
the conditional relative entropy of $\Prob_0$ with respect to $\Prob_1$ over
step $n$. In the i.i.d.\ model $\KLrev_n=\KL(P_0\|P_1)$ $\Prob_1$-a.s.
\end{lemma}

\begin{proof}
Put $h_\beta:=\frac{R_n^{\beta}-1}{\beta}$ for $\beta\in(0,1]$. For fixed
$x\in[0,\infty)$ the map $\beta\mapsto\frac{x^{\beta}-1}{\beta}$ is
nondecreasing on $(0,1]$ and decreases to $\log x$ as $\beta\downarrow0$ (for
$x>0$ it is a difference quotient of the convex $\beta\mapsto e^{\beta\log x}$;
for $x=0$ it is $-\frac1\beta$). Hence $h_\beta\downarrow\log R_n$ pointwise,
with $h_\beta\le h_1=R_n-1$ and
$\E_{\Prob_0}[|R_n-1|\mid\Fc_{n-1}]\le2$ by \Cref{lem:Zf-cond}(i). Conditional
monotone convergence for the decreasing family $(h_\beta)$ dominated above by an
integrable variable gives
$\E_{\Prob_0}[h_\beta\mid\Fc_{n-1}]\to\E_{\Prob_0}[\log R_n\mid\Fc_{n-1}]
=-\KLrev_n$ $\Prob_0$-a.s.; the conditional expectation is well defined in
$[-\infty,0]$ because $(\log R_n)_+\le R_n$.

Now $\E_{\Prob_0}[h_\beta\mid\Fc_{n-1}]=\frac{Z_{\beta,n}-1}{\beta}$ and
$\lambda_n(\beta)=-\frac{\log Z_{\beta,n}}{\beta}$. Both events below are
$\Fc_{n-1}$-measurable. On $\lC\KLrev_n<\infty\rC$ we get $Z_{\beta,n}\to1$, so
$\log Z_{\beta,n}=(Z_{\beta,n}-1)(1+o(1))$ and
$\lambda_n(\beta)\to\KLrev_n$. On $\lC\KLrev_n=\infty\rC$ we get
$\frac{Z_{\beta,n}-1}{\beta}\to-\infty$, and $\log y\le y-1$ for $y>0$ gives
$\lambda_n(\beta)\ge-\frac{Z_{\beta,n}-1}{\beta}\to\infty$.
\end{proof}

\begin{definition}[Conditional flattened design]\label{def:flat-cond}
For $\beta\in(0,1]$ and $n\in\N$ put
$R_{\beta,n}:=R_n^{\beta}/Z_{\beta,n}$ on $\lC0<Z_{\beta,n}<\infty\rC$ and
$R_{\beta,n}:=1$ elsewhere. In the i.i.d.\ model $R_{\beta,n}=R_\beta(X_n)$
$\Prob_1$-a.s.
\end{definition}

\begin{lemma}[The conditional flattened design attains the bound]
\label{lem:flat-cond}
For every $\beta\in(0,1]$ and $n\in\N$, $R_{\beta,n}$ is a conditional \evar{}
at time $n$, and with $s=\frac{1-\beta}{\beta}$,
\[
   \E_{\Prob_1}\lB R_{\beta,n}^{-s}\mid\Fc_{n-1}\rB\ =\ e^{-\lambda_n(\beta)}
   \qquad\Prob_1\text{-a.s.}
\]
\end{lemma}

\begin{proof}
$\E_{\Prob_0}[R_{\beta,n}\mid\Fc_{n-1}]=Z_{\beta,n}^{-1}
\E_{\Prob_0}[R_n^{\beta}\mid\Fc_{n-1}]=1$ $\Prob_0$-a.s.\ by
\Cref{lem:Zf-cond}(i) and \Cref{rem:Zf-cond}, so $R_{\beta,n}$ is a conditional
\evar{} in the sense of \Cref{def:cond-evar}. Since $\beta s=1-\beta$, we have $1-\beta s=\beta$, so
\Cref{lem:bayes} gives
$\E_{\Prob_1}[R_n^{-\beta s}\mid\Fc_{n-1}]
=\E_{\Prob_0}[R_n^{1-\beta s}\mid\Fc_{n-1}]=Z_{\beta,n}$, whence
$\E_{\Prob_1}[R_{\beta,n}^{-s}\mid\Fc_{n-1}]
=Z_{\beta,n}^{s}\,Z_{\beta,n}=Z_{\beta,n}^{1/\beta}=e^{-\lambda_n(\beta)}$.
\end{proof}

\begin{proposition}[Equality in \Cref{thm:upper-cond}]\label{prp:equality-cond}
Let $s\in(0,\infty)$, $\beta=\frac1{1+s}$, $n\in\N$. Then
$\E_{\Prob_1}[E_n^{-s}\mid\Fc_{n-1}]=e^{-\lambda_n(\beta)}$ $\Prob_1$-a.s.\ if
and only if $E_n=R_{\beta,n}$ $\Prob_1$-a.s. Consequently equality holds in
\Cref{prp:upper-seq}\ref{it:agg-exact} for a given $t$ if and only if
$E_n=R_{\beta,n}$ $\Prob_1$-a.s.\ for every $n\le t$.
\end{proposition}

\begin{proof}
Sufficiency is \Cref{lem:flat-cond}. For necessity, equality forces equality in
both steps of the proof of \Cref{thm:upper-cond}. Under the equality hypothesis
$A^{r}=\E_{\Prob_0}[E_n\mid\Fc_{n-1}]\in(0,1]$ and
$B^{r'}=\E_{\Prob_1}[E_n^{-s}\mid\Fc_{n-1}]=e^{-\lambda_n(\beta)}\in(0,1]$
$\Prob_1$-a.s.\ by \Cref{lem:Zf-cond}(i), so the proviso of the equality clause
of \Cref{cor:app-holder} holds and that clause gives
$G^{r}/A^{r}=H^{r'}/B^{r'}$ $\Prob_0$-a.s.; with $G=E_n^{1-\beta}$,
$r=\frac1{1-\beta}$, $H=(R_nE_n^{-s})^{\beta}$, $r'=\frac1\beta$ this reads
$E_n^{1+s}=\varkappa'R_n$ $\Prob_0$-a.s.\ with $\varkappa'$
$\Fc_{n-1}$-measurable, hence $E_n=\varkappa R_n^{\beta}$ $\Prob_0$-a.s.\ with
$\varkappa:=(\varkappa')^{\beta}$. Equality in the constraint step requires
$\E_{\Prob_0}[E_n\mid\Fc_{n-1}]=1$, that is $\varkappa Z_{\beta,n}=1$. For the last claim, the induction proving
\Cref{prp:upper-seq}\ref{it:agg-exact} runs through equalities if and only if
each of its steps does, and $M_{n-1}\in(0,\infty)$ $\Prob_1$-a.s.
\end{proof}

\begin{corollary}[The exponent of the one-step likelihood ratio]
\label{cor:chernoffinfo-cond}
Let $n\in\N$ and assume $\Prob_0(R_n>0\mid\Fc_{n-1})=1$ $\Prob_0$-a.s.\ --- the
conditional form of $P_0\ll P_1$, and implied by $\lambda_n(0^{+})<\infty$
(\Cref{lem:limit-cond}) but not conversely. Then
\[
   \sup_{s\ge0}\lC-\log\E_{\Prob_1}\lB R_n^{-s}\mid\Fc_{n-1}\rB\rC
   \ =\ \sup_{\beta\in(0,1]}\beta\,\lambda_n(\beta)
   \qquad\Prob_1\text{-a.s.},
\]
the conditional Chernoff information of the one-step kernels. Without the
hypothesis the left-hand side may be $+\infty$ while the right-hand side is
finite.
\end{corollary}

\begin{proof}
Since $R_n>0$ $\Prob_0$-a.s.\ conditionally, \Cref{lem:bayes} gives
$\E_{\Prob_1}[R_n^{-s}\mid\Fc_{n-1}]=\E_{\Prob_0}[R_n^{1-s}\mid\Fc_{n-1}]
=Z_{1-s,n}$ for \emph{every} $s\ge0$, so the left-hand side is
$\sup_{\beta\le1}\lC-\log Z_{\beta,n}\rC$ under $\beta=1-s$. By
\Cref{rem:Zf-cond} the map $\beta\mapsto\log Z_{\beta,n}$ is convex and vanishes
at $\beta=0$ and $\beta=1$, hence is $\ge0$ on $\beta\le0$; so the supremum is
attained on $\beta\in(0,1]$, where $-\log Z_{\beta,n}=\beta\lambda_n(\beta)$.
For the last sentence take $P_0=\mathrm{Bern}(\frac12)$, $P_1=\delta_1$ of
\Cref{rem:eqcase}: then $\E_{P_1}[R^{-s}]=2^{-s}$ for every $s\ge0$, so the left
side is $+\infty$, while $\sup_\beta\beta\lambda(\beta)=\log2$.
\end{proof}

\subsection{Deterministic conditional divergences}

\begin{assumption}[Homogeneity]\label{ass:homog}
For every $\beta\in(0,1]$ and every $n\in\N$ there is a constant
$z_n(\beta)\in(0,1]$ with $Z_{\beta,n}=z_n(\beta)$ $\Prob_1$-a.s.; equivalently,
$\lambda_n(\beta)$ is $\Prob_1$-a.s.\ constant. No constancy in $n$ is required.
It suffices to require this for rational $\beta$: if $\beta_k\to\beta$ then
$R_n^{\beta_k}\to R_n^{\beta}$ pointwise, dominated by $1+R_n$, so conditional
dominated convergence gives $Z_{\beta_k,n}\to Z_{\beta,n}$ $\Prob_0$-a.s. Also
$\lambda_n(0^{+})=\sup_{\beta}\lambda_n(\beta)$ is then $\Prob_1$-a.s.\
constant. The qualifier must be $\Prob_1$ and not $\Prob_0$: on
$\lC L_{n-1}=0\rC$, which \Cref{def:ratios} may charge under $\Prob_0$, the
convention $R_n:=1$ makes $Z_{\beta,n}=1$, so even in the i.i.d.\ model
$Z_{\beta,n}$ is $\Prob_0$-a.s.\ constant only when $P_0\sim P_1$. Every use
below sits inside a $\Prob_1$-expectation or an $\operatorname{ess\,sup}$ under
$\Prob_1$.
\end{assumption}

\begin{remark}[When it holds]\label{rem:homog}
\Cref{ass:homog} holds whenever $\Prob_0,\Prob_1$ are product measures with
$\Fc_t$ generated by the coordinates --- the observations independent, not
necessarily identically distributed --- since then $R_n$ agrees $\Prob_1$-a.s.\
with a function of the $n$-th coordinate alone and
$Z_{\beta,n}=\E_{P_{0,n}}[R_n^{\beta}]$ $\Prob_1$-a.s. In particular it holds
for \Cref{con:nfl}, and in the i.i.d.\ model with $z_n(\beta)=Z_\beta$. It fails as soon as the one-step conditional R\'enyi
divergences depend on the past.
\end{remark}

\begin{lemma}[The two scalars along the conditional family]\label{lem:family-cond}
For $\beta\in(0,1]$ and $n\in\N$ put
$c_{\beta,n}:=\E_{\Prob_1}[\log R_{\beta,n}\mid\Fc_{n-1}]$. Then $\Prob_1$-a.s.
\begin{equation}\label{eq:cbeta-cond}
   c_{\beta,n}\ =\ \beta\lp\KL_n+\lambda_n(\beta)\rp ,
\end{equation}
with $\KL_n$ as in \Cref{def:kl-cond}. Call $R_n$ \emph{conditionally
non-degenerate} on an event $B\in\Fc_{n-1}$ if there is no
$\Fc_{n-1}$-measurable $V$ with $R_n=V$ $\Prob_0$-a.s.\ on $B$. If moreover
$\KLrev_n<\infty$ and $R_n$ is conditionally non-degenerate, then
$\Prob_1$-a.s.\ $\beta\mapsto c_{\beta,n}$ is strictly increasing on $(0,1)$
with $c_{0^{+},n}=0$ and $c_{1,n}=\KL_n$; $\beta\mapsto\lambda_n(\beta)$ is
strictly decreasing with $\lambda_n(0^{+})=\KLrev_n$ and $\lambda_n(1)=0$; and
$c_{\beta,n}/\beta\uparrow\KL_n+\KLrev_n$ as $\beta\downarrow0$.
\end{lemma}

\begin{proof}
$\log R_{\beta,n}=\beta\log R_n-\kappa_n(\beta)$ with
$\kappa_n(\beta):=\log Z_{\beta,n}=-\beta\lambda_n(\beta)$; take
$\E_{\Prob_1}[\,\cdot\mid\Fc_{n-1}]$ and use \Cref{def:kl-cond} for
\eqref{eq:cbeta-cond}. For the rest, the proof of \Cref{lem:family} applies to
$\kappa_n$ with three substitutions: $\E_{P_0}$ by
$\E_{\Prob_0}[\,\cdot\mid\Fc_{n-1}]$, monotone and dominated convergence by
their conditional forms, and \Cref{lem:limit} by \Cref{lem:limit-cond}. Four
inputs are needed, and all are available.
$\kappa_n(1)=0$ by \Cref{lem:Zf-cond}(i).
$\kappa_n(0^{+})=0$: $\KLrev_n<\infty$ forces $\Prob_0(R_n>0\mid\Fc_{n-1})=1$
(\Cref{lem:limit-cond}), and $R_n^{\beta}\to\mathbf 1_{\{R_n>0\}}$ dominated by
$1+R_n$ gives $Z_{\beta,n}\to1$.
$\kappa_n'(1)=\E_{\Prob_0}[R_n\log R_n\mid\Fc_{n-1}]=\KL_n<\infty$
$\Prob_1$-a.s.\ by \Cref{lem:bayes} and \Cref{def:kl-cond}; the interchange at
$\beta\uparrow1$ is conditional monotone convergence on $\lC R_n\ge1\rC$ and
conditional dominated convergence on $\lC R_n<1\rC$, where
$R_n^{\beta}|\log R_n|\le\frac1{e\beta_0}$ for $\beta\ge\beta_0>0$.
Finally $\kappa_n$ is \emph{strictly} convex where $R_n$ is conditionally
non-degenerate: for $\beta_1\ne\beta_2$ and $\rho\in(0,1)$, conditional
H\"older gives
$\kappa_n(\rho\beta_1+(1-\rho)\beta_2)\le\rho\kappa_n(\beta_1)
+(1-\rho)\kappa_n(\beta_2)$, and by the equality clause of
\Cref{cor:app-holder} equality forces $R_n^{\beta_1-\beta_2}$ to equal an
$\Fc_{n-1}$-measurable variable $\Prob_0$-a.s., hence $R_n$ to do so. This
replaces the tilted-measure argument of \Cref{lem:family} and needs no regular
conditional distribution.
\end{proof}

\begin{theorem}[The ceiling, matched]\label{thm:ceiling-gen}
Under \Cref{ass:homog}, let $\mathcal W$ be the class of all betting sequences.
For every $t\in\N$,
\[
   \sup_{\mathcal W}\ \Lambda^{(t)}
   \ =\ \frac1t\sum_{n\le t}\lambda_n(0^{+})
   \ =\ \frac1t\sum_{n\le t}\KLrev_n
\]
(the second equality by \Cref{lem:limit-cond}, the right-hand side denoting the
$\Prob_1$-a.s.\ constant value of $\KLrev_n$),
approached along $(R_{\beta,n})_{n\le t}$ as $\beta\downarrow0$; and
$\Lambda^{(t)}_s=\frac1t\sum_{n\le t}\lambda_n(\beta)$ for the betting sequence
$E_n=R_{\beta,n}$, $\beta=\frac1{1+s}$. If moreover $\lambda_n(0^{+})<\infty$ for every $n\le t$, $R_n$ is conditionally non-degenerate in the
sense of \Cref{lem:family-cond} for some $n\le t$, and the betting sequence in question
has $\Prob_0(W_t\ne1)>0$, then the supremum is not attained by it.
\end{theorem}

\begin{proof}
$\le$: \Cref{cor:ceiling}, whose essential suprema are here the constants
$\lambda_n(0^{+})$ of \Cref{ass:homog}. $\ge$: for $E_n=R_{\beta,n}$, \Cref{lem:flat-cond} and
the tower property give
$\E_{\Prob_1}[W_n^{-s}]=e^{-\lambda_n(\beta)}\E_{\Prob_1}[W_{n-1}^{-s}]$,
because $\lambda_n(\beta)$ is a constant; hence
$\Lambda^{(t)}_s=\frac1t\sum_{n\le t}\lambda_n(\beta)$. Each
$\lambda_n(\cdot)$ is nonincreasing (\Cref{lem:Zf-cond}(ii)), so the finite sum
is nonincreasing in $\beta$ and its supremum over $\beta\in(0,1]$ is
$\frac1t\sum_{n\le t}\lambda_n(0^{+})$ by monotone convergence. Taking the
supremum over $s\ge0$ and over $\mathcal W$ in either order gives the
displayed identity.

For non-attainment, suppose $\Lambda^{(t)}=\frac1t\sum_{n\le t}\lambda_n(0^{+})$
for some betting sequence with $\Prob_0(W_t\ne1)>0$. First,
$\lambda_n(0^{+})<\infty$ forces $\Prob_0(R_n>0\mid\Fc_{n-1})=1$
(\Cref{cor:chernoffinfo-cond}), so by induction $\Prob_0(L_t>0)=1$ and therefore
$\Prob_0|_{\Fc_t}\sim\Prob_1|_{\Fc_t}$. \Cref{lem:attained} --- whose proof uses
only that equivalence and $\Prob_0(W_t\ne1)>0$ --- then supplies an
$s^{\ast}\in[0,\infty)$ with $\Lambda^{(t)}=\Lambda^{(t)}_{s^{\ast}}$, and
$s^{\ast}>0$ because $\Lambda^{(t)}_0=0$ while the right-hand side is positive
(strict conditional Jensen below). Hence
$\lambda_n(\beta^{\ast})=\lambda_n(0^{+})$ for every $n\le t$ with
$\beta^{\ast}=\frac1{1+s^{\ast}}$; by \Cref{lem:Zf-cond}(ii) this forces
$\lambda_n$ to be constant on $(0,\beta^{\ast}]$. But $\lambda_n(\beta)$ is
minus the slope of the chord of $\kappa_n$ from $0$ to $\beta$, and
$\kappa_n(0^{+})=0$ with $\kappa_n$ strictly convex
(\Cref{lem:family-cond}, using $\lambda_n(0^{+})<\infty$ and conditional
non-degeneracy) makes that slope strictly increasing, so $\lambda_n$ is
strictly decreasing --- a contradiction.
\end{proof}

\begin{proposition}[Exact error of a flattened design]\label{prp:exactfamily-gen}
Under \Cref{ass:homog}, fix $\beta\in(0,1]$ and let $(W_t)_{t\in\N_0}$ be
generated by $E_n=R_{\beta,n}$. Put $G_t:=\log L_t$ and
$B_t(\beta):=\frac1\beta\lp\ell+\sum_{n\le t}\log z_n(\beta)\rp$. Then
$\log W_t=\beta\lp G_t-B_t(\beta)\rp+\ell$ $\Prob_1$-a.s., and
\[
   \gamma_t(\alpha)=\Prob_1\lp G_t\le B_t(\beta)\rp ,
   \qquad
   \bar\gamma_t(\alpha)=\Prob_1\lp G_n\le B_n(\beta)\ \text{for all }n\le t\rp .
\]
\end{proposition}

\begin{proof}
$\log R_{\beta,n}=\beta\log R_n-\log z_n(\beta)$ $\Prob_1$-a.s., and
$\sum_{n\le t}\log R_n=\log L_t$; sum and rearrange. Then
$\log W_t\le\ell$ is $G_t\le B_t(\beta)$, and applying this at every $n\le t$
gives the second identity, by \Cref{def:typeII}.
\end{proof}

\begin{remark}[The horizon-optimal tilt]\label{rem:betastar-gen}
On $\beta\in(0,1]$, \Cref{thm:betastar} carries over verbatim, with
$\lambda(\beta)$ replaced by $\lambda_1(\beta)$, whenever \Cref{ass:homog} holds
\emph{and}
$\lambda_n(\beta)=\lambda_1(\beta)$ for every $n$: only then is
$B_t(\beta)=\frac1\beta(\ell-t\beta\lambda_1(\beta))$ affine in $t$, which is
what the analysis of $\beta\mapsto B_t(\beta)$ uses. Without constancy in $n$,
$B_t$ is still deterministic and \Cref{prp:exactfamily-gen} still reduces the
design problem to one scalar per horizon, but the crossover statement
\Cref{thm:betastar}(a) and the rate \Cref{thm:betastar}(b) have no closed form.
The extension to $\beta>1$ needs more still: every statement of
\Cref{app:general} is written for $\beta\in(0,1]$, and beyond it one would have
to require $Z_{\beta,n}$ to be finite and $\Prob_1$-a.s.\ constant there too,
which \Cref{ass:homog} does not provide.
\end{remark}

\begin{lemma}[Surrogate bound at the tilt $s=\frac1\beta$]\label{lem:surrogate-gen}
Under \Cref{ass:homog}, for every $\beta\in(0,1]$, $\alpha\in(0,1)$ and
$t\in\N$, the betting sequence $E_n=R_{\beta,n}$ satisfies
$\gamma_t(\alpha)\le e^{B_t(\beta)}$ with
$B_t(\beta)=\frac{\ell}{\beta}-\sum_{n\le t}\lambda_n(\beta)$, in agreement with
\Cref{prp:exactfamily-gen}.
\end{lemma}

\begin{proof}
$R_{\beta,n}^{-1/\beta}=Z_{\beta,n}^{1/\beta}R_n^{-1}$, and
$\E_{\Prob_1}[R_n^{-1}\mid\Fc_{n-1}]
=\E_{\Prob_0}[R_n R_n^{-1}\mid\Fc_{n-1}]
=\Prob_0(R_n>0\mid\Fc_{n-1})\le1$ by \Cref{lem:bayes}, so
$\E_{\Prob_1}[R_{\beta,n}^{-1/\beta}\mid\Fc_{n-1}]\le Z_{\beta,n}^{1/\beta}
=e^{-\lambda_n(\beta)}$. Since the $\lambda_n(\beta)$ are constants, the tower
induction of \Cref{prp:upper-seq}\ref{it:agg-exact} run with $\le$ in place of
$\ge$ gives
$\E_{\Prob_1}[W_t^{-1/\beta}]\le\exp(-\sum_{n\le t}\lambda_n(\beta))$;
\Cref{lem:chernoff} at $s=\frac1\beta$ then gives the claim.
\end{proof}

\begin{lemma}[Maximal inequality for square-integrable martingales]
\label{lem:doob-gen}
Let $(M_n)_{n\le t}$ be a $\Prob_1$-martingale with $M_0=0$ and
$\E_{\Prob_1}[M_t^{2}]<\infty$. Then
$\Prob_1(\max_{n\le t}|M_n|\ge x)\le\E_{\Prob_1}[M_t^{2}]/x^{2}$ for every
$x>0$.
\end{lemma}

\begin{proof}
Let $\sigma:=\inf\lC n\le t\st|M_n|\ge x\rC$ and $A_n:=\lC\sigma=n\rC\in\Fc_n$.
Conditional Jensen gives $\E_{\Prob_1}[M_t^{2}\mid\Fc_n]\ge M_n^{2}$, so
\[
   \E_{\Prob_1}\lB M_t^{2}\rB\ \ge\ \sum_{n\le t}\E_{\Prob_1}\lB M_t^{2}\mathbf 1_{A_n}\rB
   \ \ge\ \sum_{n\le t}\E_{\Prob_1}\lB M_n^{2}\mathbf 1_{A_n}\rB
   \ \ge\ x^{2}\,\Prob_1(\sigma\le t) . \qedhere
\]
\end{proof}

\begin{lemma}[A universal horizon floor]\label{lem:floor-gen}
Let $(E_n)_{n\in\N}$ be a betting sequence with
$\E_{\Prob_1}[\log E_n\mid\Fc_{n-1}]\le c$ and
$\Var_{\Prob_1}(\log E_n\mid\Fc_{n-1})\le v<\infty$ $\Prob_1$-a.s.\ for every
$n$, where $c>0$ and $c\,\ell>2v$. Then
$\bar\gamma_t(\alpha)\ge1-\frac{2v}{c\,\ell}$ for every $t\le\frac{\ell}{2c}$.
\end{lemma}

\begin{proof}
Put $S_n:=\log W_n$ and
$M_n:=S_n-\sum_{m\le n}\E_{\Prob_1}[\log E_m\mid\Fc_{m-1}]$, a $\Prob_1$-martingale
with $M_0=0$ and orthogonal increments, so
$\E_{\Prob_1}[M_t^{2}]=\sum_{m\le t}\E_{\Prob_1}[(\Delta M_m)^{2}]\le tv$. Fix
$t\le\ell/(2c)$. For $n\le t$ the drift term is at most $nc\le\ell/2$, so
$\lC S_n>\ell\rC\subseteq\lC M_n>\frac\ell2\rC$ and, by \Cref{lem:doob-gen},
\[
   \Prob_1(\tau_\alpha\le t)=\Prob_1\lp\max_{n\le t}S_n>\ell\rp
   \ \le\ \Prob_1\lp\max_{n\le t}|M_n|\ge\tfrac\ell2\rp
   \ \le\ \frac{tv}{(\ell/2)^{2}}\ \le\ \frac{2v}{c\,\ell} . \qedhere
\]
\end{proof}

\begin{corollary}[The cost of a near-optimal exponent]\label{cor:price-gen}
Assume \Cref{ass:homog}, $\Var_{\Prob_1}(\log R_n\mid\Fc_{n-1})\le\bar v<\infty$
$\Prob_1$-a.s.\ for every $n$, and
$J:=\sup_n\operatorname*{ess\,sup}_{\Prob_1}\lp\KL_n+\KLrev_n\rp<\infty$ --- the
last of which is what makes the horizon range below non-empty, and does not
follow from the others. Put
$\bar c_\beta:=\sup_n\operatorname*{ess\,sup}_{\Prob_1}c_{\beta,n}\le\beta J$.
Then $\Var_{\Prob_1}(\log R_{\beta,n}\mid\Fc_{n-1})\le\beta^{2}\bar v$, and
whenever $\bar c_\beta\ell>2\beta^{2}\bar v$ the design $E_n=R_{\beta,n}$ has
power at most $2\beta^{2}\bar v/(\bar c_\beta\ell)=O(\beta)$ at every horizon
$t\le\ell/(2\bar c_\beta)$, and $\ell/(2\bar c_\beta)\ge\ell/(2\beta J)$
diverges like $\beta^{-1}$ as $\beta\downarrow0$.
\end{corollary}

\begin{proof}
$\log R_{\beta,n}=\beta\log R_n-\kappa_n(\beta)$ with $\kappa_n(\beta)$
constant, so the conditional variance scales by $\beta^{2}$; apply
\Cref{lem:floor-gen} with $c=\bar c_\beta$ and $v=\beta^{2}\bar v$. By
\eqref{eq:cbeta-cond} and $\lambda_n(\beta)\le\lambda_n(0^{+})=\KLrev_n$,
$\bar c_\beta=\beta\sup_n\operatorname*{ess\,sup}(\KL_n+\lambda_n(\beta))
\le\beta J$.
\end{proof}

\subsection{A Stein-type converse}

\begin{assumption}[Information spectrum]\label{ass:spectrum}
There is $D\in[0,\infty)$ such that for every $\epsilon>0$,
\[
   \Prob_0\lp-\tfrac1t\log L_t\ >\ D+\epsilon\rp\ \xrightarrow[t\to\infty]{}\ 0 .
\]
\end{assumption}

\begin{proposition}[The ceiling binds the error]\label{prp:stein-gen}
Under \Cref{ass:spectrum}, for every betting sequence and every
$\alpha\in(0,1)$,
\[
   \limsup_{t\to\infty}\ -\frac1t\log\bar\gamma_t(\alpha)\ \le\ D ,
\]
and the same holds with $\gamma_t(\alpha)$ in place of $\bar\gamma_t(\alpha)$.
\end{proposition}

\begin{proof}
Fix $\epsilon>0$, put $A_t:=\{\tau_\alpha>t\}\in\Fc_t$ and
$B_t:=\lC L_t\ge e^{-t(D+\epsilon)}\rC$, so $\Prob_0(B_t^{c})\to0$ by
\Cref{ass:spectrum} and $\Prob_0(A_t^{c})\le\Prob_0(\tau_\alpha<\infty)\le\alpha$
by \Cref{thm:typeI}. Since $\Prob_1\ll\Prob_0$ on $\Fc_t$ with density $L_t$,
\[
   \bar\gamma_t(\alpha)=\E_{\Prob_0}\lB L_t\mathbf 1_{A_t}\rB
   \ \ge\ e^{-t(D+\epsilon)}\,\Prob_0(A_t\cap B_t)
   \ \ge\ e^{-t(D+\epsilon)}\lp1-\alpha-\Prob_0(B_t^{c})\rp .
\]
Hence $-\frac1t\log\bar\gamma_t(\alpha)\le D+\epsilon-\frac1t\log(1-\alpha-o(1))$,
and $\epsilon>0$ was arbitrary. For $\gamma_t(\alpha)$ replace $A_t$ by
$\lC W_t\le\frac1\alpha\rC$, whose complement has
$\Prob_0(W_t>\frac1\alpha)\le\alpha\E_{\Prob_0}[W_t]\le\alpha$ by Markov and
\Cref{lem:supmg}.
\end{proof}

\begin{remark}[Recovering \Cref{prp:stein-ceiling}]\label{rem:stein-iid}
In the i.i.d.\ model with $\KL(P_0\|P_1)<\infty$ one has
$-\frac1t\log L_t=-\frac1t\sum_{n\le t}\log R(X_n)\to\KL(P_0\|P_1)$
$\Prob_0$-a.s.\ by the strong law, so \Cref{ass:spectrum} holds with
$D=\KL(P_0\|P_1)$ and \Cref{prp:stein-gen} contains \Cref{prp:stein-ceiling}.
Under \Cref{ass:homog} and provided $\Prob_0(L_t>0)=1$ for every $t$ ---
equivalently $\lambda_n(0^{+})<\infty$ for every $n$, so that
$\log L_t=\sum_{n\le t}\log R_n$ $\Prob_0$-a.s.\ --- the quantity
$\E_{\Prob_0}[-\log R_n]$ need not be constant in $n$, and $D$ is then the limit
of $\frac1t\sum_{n\le t}\E_{\Prob_0}[-\log R_n]$ whenever a weak law applies.
\end{remark}

\subsection{Where the time average goes}

\begin{lemma}[The ceiling is the compensator of the information spectrum]
\label{lem:compensator}
Fix $t\in\N$ and assume $\lambda_n(0^{+})<\infty$ $\Prob_0$-a.s.\ for every
$n\le t$, so that $\Prob_0(L_t>0)=1$. Then $\Prob_0$-a.s.
\[
   -\log L_t\ =\ A_t+M_t ,\qquad
   A_t:=\sum_{n\le t}\KLrev_n ,\qquad
   M_t:=\sum_{n\le t}\lp-\log R_n-\KLrev_n\rp ,
\]
with $(A_t)$ predictable and nondecreasing and $(M_t)$ a $\Prob_0$-martingale
with $M_0=0$; and
\begin{equation}\label{eq:ceiling-is-divergence}
   \E_{\Prob_0}\lB\frac{A_t}{t}\rB
   \ =\ \frac1t\,\KL\lp\Prob_0|_{\Fc_t}\,\big\|\,\Prob_1|_{\Fc_t}\rp .
\end{equation}
\end{lemma}

\begin{proof}
$\log L_t=\sum_{n\le t}\log R_n$ $\Prob_0$-a.s.\ on $\lC L_t>0\rC$, and
$\KLrev_n=-\E_{\Prob_0}[\log R_n\mid\Fc_{n-1}]\in[0,\infty)$ by
\Cref{lem:limit-cond}; this is the Doob decomposition. Taking
$\E_{\Prob_0}$ and using $\E_{\Prob_0}[M_t]=0$ gives
$\E_{\Prob_0}[A_t]=\E_{\Prob_0}[-\log L_t]
=\E_{\Prob_0}[\log\frac{d\Prob_0|_{\Fc_t}}{d\Prob_1|_{\Fc_t}}]$.
\end{proof}

\begin{proposition}[The ceiling and the Stein exponent have one limit]
\label{prp:rate}
Assume $\lambda_n(0^{+})<\infty$ $\Prob_0$-a.s.\ and
$\Var_{\Prob_0}(\log R_n\mid\Fc_{n-1})\le \bar v<\infty$ $\Prob_0$-a.s.\ for
every $n$. Then $\frac1tM_t\to0$ $\Prob_0$-a.s., so that
$\frac1tA_t$ and $-\frac1t\log L_t$ have the same limits: for
$D\in[0,\infty)$,
\[
   \frac1t\sum_{n\le t}\KLrev_n\ \xrightarrow[t\to\infty]{}\ D
   \quad\Prob_0\text{-a.s.}
   \qquad\Longleftrightarrow\qquad
   -\frac1t\log L_t\ \xrightarrow[t\to\infty]{}\ D
   \quad\Prob_0\text{-a.s.}
\]
In that case \Cref{ass:spectrum} holds with this $D$, and under
\Cref{ass:homog} the ceiling $\frac1t\sum_{n\le t}\KLrev_n$ of
\Cref{thm:ceiling-gen} converges to the exponent $D$ of \Cref{prp:stein-gen}.
\end{proposition}

\begin{proof}
The increments $\Delta M_n=-\log R_n-\KLrev_n$ are $\Prob_0$-martingale
differences with $\E_{\Prob_0}[(\Delta M_n)^{2}]\le\bar v$, so the argument of
\Cref{prp:powerone-gen} --- with $\Prob_1$ replaced by $\Prob_0$ and $\log E_n$
by $-\log R_n$ --- gives $\frac1tM_t\to0$ $\Prob_0$-a.s. The equivalence and the
last two claims follow from \Cref{lem:compensator}, a.s.\ convergence implying
convergence in probability.
\end{proof}

\begin{remark}[Three regimes]\label{rem:rate}
Whether the limit exists at all depends on the model.
\begin{enumerate}[label=(\alph*),leftmargin=2.4em]
\item \emph{The i.i.d.\ model.} $\KLrev_n=\KL(P_0\|P_1)$ for every $n$
      (\Cref{lem:limit-cond}), so $D=\KL(P_0\|P_1)$ and the ceiling does not
      depend on the horizon. This is why \Cref{sec:ceiling} states one number.
\item \emph{Independent, not identically distributed.} Under \Cref{ass:homog}
      the average is the Ces\`aro mean of the deterministic sequence
      $\KLrev_n$, which need not converge. Let
      $\KLrev_n=1$ for $n\in[2^{2k},2^{2k+1})$ and $\KLrev_n=0$ for
      $n\in[2^{2k+1},2^{2k+2})$, $k\in\N_0$. Then
      $\frac1t\sum_{n\le t}\KLrev_n$ has $\liminf=\frac13$ and
      $\limsup=\frac23$. The ceiling is then genuinely a function of the
      horizon, and \Cref{thm:ceiling-gen} is the only form available.
\item \emph{Stationary ergodic.} If $\Prob_0$ is stationary and ergodic and
      $\Prob_1$ is a stationary Markov measure of finite order, then
      $-\frac1t\log L_t$ converges $\Prob_0$-a.s.\ to the relative entropy rate
      $\bar D(\Prob_0\|\Prob_1)
      =\lim_t\frac1t\KL(\Prob_0|_{\Fc_t}\|\Prob_1|_{\Fc_t})$, by the generalised
      Shannon--McMillan--Breiman theorem \cite{Bar85}, \cite{AC88}. By
      \Cref{prp:rate} the ceiling converges to the same value, and
      \eqref{eq:ceiling-is-divergence} identifies it in expectation at every
      finite $t$.
\end{enumerate}
Note the direction: \Cref{ass:standing} constrains
$\KL(\Prob_1|_{\Fc_t}\|\Prob_0|_{\Fc_t})$, while
\eqref{eq:ceiling-is-divergence} is about
$\KL(\Prob_0|_{\Fc_t}\|\Prob_1|_{\Fc_t})$ --- the same asymmetry between the two
divergences that \Cref{thm:two} records for the i.i.d.\ model.
\end{remark}

\subsection{Identifying the exponent, and power one}

\begin{assumption}[G\"artner--Ellis]\label{ass:ge}
With $S_t:=\log W_t$, the limit
$\Psi(\theta):=\lim_{t\to\infty}\frac1t\log\E_{\Prob_1}[e^{\theta S_t}]$ exists
in $(-\infty,\infty]$ for every $\theta\in\R$; $0$ lies in the interior of
$\mathcal D_\Psi:=\lC\theta\st\Psi(\theta)<\infty\rC$; and $\Psi$ is lower
semicontinuous and essentially smooth on $\operatorname{int}\mathcal D_\Psi$ ---
that is, the hypotheses of \cite[Thm.~2.3.6]{DZ10}. Write
$\Psi^{\star}(x):=\sup_{\theta\in\R}\lC\theta x-\Psi(\theta)\rC$.
\end{assumption}

\begin{proposition}[The exponent]\label{prp:exponent-gen}
Assume \Cref{ass:ge}, that $\Psi'(0)>0$, and that $0$ lies in the interior of
$\lC x\st\Psi^{\star}(x)<\infty\rC$. Then, for every $\alpha\in(0,1)$,
\[
   \lim_{t\to\infty}-\frac1t\log\gamma_t(\alpha)
   \ =\ \sup_{s\ge0}\ \lim_{t\to\infty}\Lambda^{(t)}_s
   \ =\ \Psi^{\star}(0) .
\]
Both extra hypotheses are needed: they are the analogues of
$\E_{P_1}[\log E]>0$ and of the interiority hypothesis of \Cref{prp:exponent},
and \Cref{ass:ge} alone does not imply either.
\end{proposition}

\begin{proof}
\Cref{ass:ge} is the hypothesis of the G\"artner--Ellis theorem
\cite[Thm.~2.3.6]{DZ10} for $(S_t/t)_{t\in\N}$, which therefore satisfies a
large-deviation principle with good convex rate function $\Psi^{\star}$; and
$\Psi^{\star}$ vanishes at $\Psi'(0)>0$, is nonincreasing on
$(-\infty,\Psi'(0)]$, and is continuous at $0$ by the interiority hypothesis.
The argument of \Cref{prp:exponent} now applies with $I=\Psi^{\star}$ and
$\gamma_t(\alpha)=\Prob_1(S_t/t\le\ell/t)$, giving the first equality. For the
second, $\lim_t\Lambda^{(t)}_s=-\Psi(-s)$ by definition, so
$\sup_{s\ge0}\lim_t\Lambda^{(t)}_s=\sup_{\theta\le0}\{-\Psi(\theta)\}$; and for
$\theta>0$ convexity gives $\Psi(\theta)\ge\Psi(0)+\theta\Psi'(0)>0$, so
$-\Psi(\theta)<0\le-\Psi(0)\le\Psi^{\star}(0)$ and positive $\theta$ do not
contribute to $\Psi^{\star}(0)=\sup_{\theta\in\R}\{-\Psi(\theta)\}$.
\end{proof}

\begin{proposition}[Power one]\label{prp:powerone-gen}
Let $(E_n)_{n\in\N}$ be a betting sequence with
$\E_{\Prob_1}[\log E_n\mid\Fc_{n-1}]\ge c>0$ and
$\Var_{\Prob_1}(\log E_n\mid\Fc_{n-1})\le v<\infty$ $\Prob_1$-a.s.\ for every
$n$. Then $\tau_\alpha<\infty$ $\Prob_1$-a.s.\ for every $\alpha\in(0,1)$.
\end{proposition}

\begin{proof}
Put $\Delta_n:=\log E_n-\E_{\Prob_1}[\log E_n\mid\Fc_{n-1}]$, a square-integrable
$\Prob_1$-martingale difference sequence. The martingale
$N_t:=\sum_{n\le t}\frac{\Delta_n}{n}$ has
$\E_{\Prob_1}[N_t^{2}]=\sum_{n\le t}\frac{\E_{\Prob_1}[\Delta_n^{2}]}{n^{2}}
\le v\sum_{n\ge1}n^{-2}<\infty$, so $N_t$ converges $\Prob_1$-a.s.; Kronecker's
lemma gives $\frac1t\sum_{n\le t}\Delta_n\to0$ $\Prob_1$-a.s. Hence
$\liminf_t\frac1t\log W_t\ge c>0$ and $\log W_t\to+\infty$ $\Prob_1$-a.s.
\end{proof}

\section{Results quoted from the literature}
\label{app:quoted}

Each result is stated in the notation of this paper. Nothing here is proved;
references are given for each.

\begin{theorem}[Ville's inequality; {\cite{Vil39}}, {\cite[Lem.~1]{How20}}]
\label{thm:app-ville}
Let $(U_t)_{t\in\N_0}$ be a nonnegative supermartingale on a filtered
probability space $(\Omega,\Bs{\Omega},(\Fc_t)_{t\in\N_0},\Prob)$ with $U_0\le1$
$\Prob$-a.s. Then for every $x\in(1,\infty)$,
\[
   \Prob\lp\sup_{t\in\N_0}U_t\ \ge\ x\rp\ \le\ \frac1x .
\]
\end{theorem}

\begin{theorem}[H\"older's inequality; {\cite[Thm.~3.5 and Rem.~3.6]{Rud87}}]
\label{thm:app-holder}
Let $(\Xc,\Bs{\Xc},\mu)$ be a measure space and let $r,r'\in(1,\infty)$ satisfy
$\frac1r+\frac1{r'}=1$. Then for all measurable $g,h:\Xc\to[0,\infty]$,
\[
   \int gh\,d\mu\ \le\ \lp\int g^{r}d\mu\rp^{1/r}\lp\int h^{r'}d\mu\rp^{1/r'} ,
\]
with equality if and only if there are constants $c_1,c_2\ge0$, not both zero,
with $c_1g^{r}=c_2h^{r'}$ $\mu$-a.e.; for the ``only if'' direction one needs
$0<\int g^{r}d\mu<\infty$ and $0<\int h^{r'}d\mu<\infty$.
\end{theorem}

\begin{corollary}[Conditional H\"older]\label{cor:app-holder}
Let $(\Omega,\Fc,\Prob)$ be a probability space, $\mathcal G\subseteq\Fc$ a
sub-$\sigma$-algebra, and $r,r'\in(1,\infty)$ with $\frac1r+\frac1{r'}=1$. Then
for all measurable $G,H:\Omega\to[0,\infty]$,
\[
   \E\lB GH\mid\mathcal G\rB\ \le\ \lp\E\lB G^{r}\mid\mathcal G\rB\rp^{1/r}
                            \lp\E\lB H^{r'}\mid\mathcal G\rB\rp^{1/r'}
   \qquad\Prob\text{-a.s.}
\]
Write $A:=(\E[G^{r}\mid\mathcal G])^{1/r}$ and
$B:=(\E[H^{r'}\mid\mathcal G])^{1/r'}$. On the $\mathcal G$-measurable event
$\lC0<A,B<\infty\rC$, equality holds a.s.\ if and only if
$G^{r}/A^{r}=H^{r'}/B^{r'}$ a.s.
\end{corollary}

\begin{proof}
Write $A:=(\E[G^{r}\mid\mathcal G])^{1/r}$ and $B:=(\E[H^{r'}\mid\mathcal G])^{1/r'}$, both
$\mathcal G$-measurable with values in $[0,\infty]$. On $\lC A=0\rC$ we have $G=0$
a.s.\ conditionally, hence $\E[GH\mid\mathcal G]=0$; likewise on $\lC B=0\rC$; and on
$\lC A=\infty\rC\cup\lC B=\infty\rC$ the right-hand side is $+\infty$ unless the
other factor vanishes, which is the case just treated. So it suffices to argue
on $\mathcal G$-measurable event $\lC0<A,B<\infty\rC$. There Young's inequality
$xy\le\frac{x^{r}}{r}+\frac{y^{r'}}{r'}$ applied pointwise to $x=G/A$,
$y=H/B$ and conditioned on $\mathcal G$ gives, using that $A,B$ are $\mathcal G$-measurable,
\[
   \frac{\E[GH\mid\mathcal G]}{AB}
   \ \le\ \frac{\E[G^{r}\mid\mathcal G]}{r\,A^{r}}+\frac{\E[H^{r'}\mid\mathcal G]}{r'\,B^{r'}}
   \ =\ \frac1r+\frac1{r'}\ =\ 1 .
\]
Young's inequality is an equality exactly at $x^{r}=y^{r'}$, so on
$\lC0<A,B<\infty\rC$ the displayed inequality is an equality if and only if the
conditional expectation of the nonnegative gap
$\frac{G^{r}}{rA^{r}}+\frac{H^{r'}}{r'B^{r'}}-\frac{GH}{AB}$ vanishes, i.e.\ if
and only if the gap itself vanishes a.s., i.e.\ $G^{r}/A^{r}=H^{r'}/B^{r'}$ a.s.
\end{proof}

\begin{remark}[Use in this paper]\label{rem:app-holder}
\Cref{thm:upper-cond} applies \Cref{cor:app-holder} under
$\E_{\Prob_0}[\,\cdot\mid\Fc_{n-1}]$ with $r=\frac1{1-\beta}$, $r'=\frac1\beta$,
$G=E_n^{1-\beta}$ and $H=(R_nE_n^{-s})^{\beta}$; no finiteness is needed for the
inequality, and \Cref{prp:equality-cond} checks the proviso of the equality
clause before using it.

\Cref{thm:upper} applies the unconditional \Cref{thm:app-holder} with
$r=\frac{1}{1-\beta}$,
$r'=\frac1\beta$, $g=(p_0E)^{1-\beta}$ and $h=(p_1E^{-s})^{\beta}$, so that
$\int g^{r}d\mu=\E_{P_0}[E]\in(0,1]$ and
$\int h^{r'}d\mu=\E_{P_1}[E^{-s}]\in(0,\infty]$ once the degenerate cases of
Step 1 there are excluded. The inequality itself needs no finiteness. For the
equality case the provisos do hold: if $\E_{P_1}[E^{-s}]=\infty$ then
$\Lambda_s(E)=-\infty<\lambda(\beta)$, so equality forces it to be finite, and
$\E_{P_0}[E]>0$ because $\E_{P_0}[E]=0$ would give $E=0$ $P_0$-a.s.\ and hence
$P_1$-a.s.
\end{remark}

\begin{theorem}[Kolmogorov's maximal inequality; {\cite[Thm.~22.4]{Bil95}}]
\label{thm:app-kolmogorov}
Let $(Y_n)_{n\in\N}$ be independent real random variables on
$(\Omega,\Bs{\Omega},\Prob)$ with $\E[Y_n]=0$ and $\E[Y_n^{2}]<\infty$, and put
$M_n:=\sum_{m\le n}Y_m$. Then for every $t\in\N$ and every $x>0$,
\[
   \Prob\lp\max_{n\le t}|M_n|\ \ge\ x\rp\ \le\ \frac{\E[M_t^{2}]}{x^{2}} .
\]
\end{theorem}

\begin{theorem}[Cram\'er's theorem on $\R$; {\cite{Cra38}}, {\cite[Thm.~2.2.3]{DZ10}}]
\label{thm:app-cramer}
Let $(Y_n)_{n\in\N}$ be i.i.d.\ real random variables on $(\Omega,\Fc_\infty,\Prob)$
with cumulant generating function $\kappa_Y:\R\to(-\infty,\infty]$,
$\kappa_Y(\theta):=\log\E[e^{\theta Y_1}]$, and let
$\bar Y_t:=t^{-1}\sum_{n\le t}Y_n$. Put
\[
   I:\R\to[0,\infty],\qquad
   I(x):=\sup_{\theta\in\R}\lC\theta x-\kappa_Y(\theta)\rC .
\]
Then $(\bar Y_t)_{t\in\N}$ satisfies a large deviation principle with rate
function $I$: for every closed $F\subseteq\R$ and every open $G\subseteq\R$,
\[
   \limsup_{t\to\infty}\tfrac1t\log\Prob(\bar Y_t\in F)\le-\inf_{x\in F}I(x),
   \qquad
   \liminf_{t\to\infty}\tfrac1t\log\Prob(\bar Y_t\in G)\ge-\inf_{x\in G}I(x) .
\]
\end{theorem}

\begin{theorem}[Chernoff--Stein lemma, with strong converse;
{\cite{Chernoff56}}, {\cite[Thm.~11.8.3]{CT06}}, {\cite{HK89}}]
\label{thm:app-stein}
Let $X_1,\dots,X_t$ be i.i.d.\ with one-observation law $P_0$ or $P_1$, and let
$\KL(P_0\|P_1)<\infty$. For $\alpha\in(0,1)$ set
\[
   \gamma^{\star}_t(\alpha):=\inf\lC P_1^{\otimes t}(C_t)\ :\
      C_t\in\Bs{\Xc}^{\otimes t},\ P_0^{\otimes t}(C_t^{c})\le\alpha\rC ,
\]
the least type-II error of a test with acceptance region $C_t$ for $P_0$ and
type-I error at most $\alpha$. Then
\[
   \lim_{t\to\infty}\ -\frac1t\log\gamma^{\star}_t(\alpha)\ =\ \KL(P_0\|P_1)
   \qquad\text{for every }\alpha\in(0,1) .
\]
\end{theorem}

\begin{remark}[Use in this paper, and the range of $\alpha$]\label{rem:app-stein}
\cite[Thm.~11.8.3]{CT06} proves \Cref{thm:app-stein} for
$\alpha\in(0,\frac12)$; that the value does not change for
$\alpha\in[\frac12,1)$ is the strong converse of \cite{HK89}. It is that
range-free form which \Cref{prp:stein-ceiling} uses, since nothing in this
paper restricts $\alpha$ below $\frac12$. (\cite{CT06} constrain the type-I
error by $<\alpha$ rather than $\le\alpha$; the exponent is the same for every
$\alpha\in(0,1)$, so the difference is immaterial.)
\end{remark}

\begin{theorem}[R\'enyi divergence: order monotonicity and the limit at one;
{\cite[Prop.~2, Thms.~3 and 5]{vEH14}}]
\label{thm:app-renyi}
Let $Q,P$ be probability measures on $(\Xc,\Bs{\Xc})$ and let $\Ren_a$ be as in
\Cref{def:renyi}; no absolute continuity is assumed. Then, for $a\in(0,1)$,
\[
   \Ren_a(Q\|P)=\frac{a}{1-a}\,\Ren_{1-a}(P\|Q) ;
\]
the map $a\mapsto\Ren_a(Q\|P)$ is nondecreasing on $(0,1)$; and
$\lim_{a\uparrow1}\Ren_a(Q\|P)=\KL(Q\|P)$, with no further hypothesis. (The
proviso ``$\KL(Q\|P)=\infty$ or $\Ren_b(Q\|P)<\infty$ for some $b>1$'' in
\cite{vEH14} governs the two-sided limit at order one, not the left limit used
here.)
\end{theorem}

\begin{remark}[Use in this paper]\label{rem:app-renyi}
By \Cref{lem:Zf}(iii), $\lambda(\beta)=\Ren_{1-\beta}(P_0\|P_1)$, so
\Cref{lem:Zf}(ii) and \Cref{lem:limit} are instances of the monotonicity and
the limit above, and \Cref{lem:Zf}(iii) itself is the skew-symmetry identity.
We proved all three directly in \Cref{sec:ceiling} to keep the argument
self-contained.
\end{remark}

\begin{theorem}[Variational characterisation of R\'enyi divergence;
{\cite[Thm.~1]{Ana18}}]
\label{thm:app-variational}
Let $a\in(0,1)$ and let $Q,P$ be probability measures on $(\Xc,\Bs{\Xc})$ with
$Q\ll P$. Then, with $\Ren_a$ as in \Cref{def:renyi},
\begin{equation}\label{eq:app-variational}
   \Ren_a(Q\|P)\ =\ \inf_{\mu\ll Q}
   \lC\ \KL(\mu\|P)\ +\ \frac{a}{1-a}\,\KL(\mu\|Q)\ \rC ,
\end{equation}
the infimum running over probability measures $\mu$ on $(\Xc,\Bs{\Xc})$ and
being attained at the measure $\mu_a$ with
$\frac{d\mu_a}{dP}\propto\lp\frac{dQ}{dP}\rp^{a}$.
\end{theorem}

\begin{remark}[Use in this paper: the two descriptions of $\lambda(\beta)$]
\label{rem:app-variational}
Put $a=1-\beta$ and $s=\frac{1-\beta}{\beta}$, so that $\frac{a}{1-a}=s$, and
take $Q=P_0$, $P=P_1$. Then \eqref{eq:app-variational} reads
\begin{equation}\label{eq:lambda-inf}
   \lambda(\beta)\ =\ \Ren_{1-\beta}(P_0\|P_1)
   \ =\ \inf_{\mu\ll P_0}\lC\ \KL(\mu\|P_1)+s\,\KL(\mu\|P_0)\ \rC ,
\end{equation}
attained at the $\mu$ with $\frac{d\mu}{dP_0}=R_\beta=\frac{R^{\beta}}{Z_\beta}$,
while \Cref{thm:upper} together with \Cref{lem:flat} reads
\begin{equation}\label{eq:lambda-sup}
   \lambda(\beta)\ =\ \sup_{E\in\Ev}\ \lC-\log\E_{P_1}\lB E^{-s}\rB\rC
   \ =\ \sup_{E\in\Ev}\ \Lambda_s(E) ,
\end{equation}
attained at $E=R_\beta$. The two are convex duals: \eqref{eq:lambda-inf} is an
infimum over \emph{measures} of a weighted pair of Kullback--Leibler
divergences, \eqref{eq:lambda-sup} a supremum over \emph{\evars{}} of a
log-moment, and the same $R_\beta$ solves both --- as a density in the first and
as a test function in the second. What \Cref{sec:ceiling} adds to
\eqref{eq:lambda-inf} is \eqref{eq:lambda-sup}: the constraint set is exactly
$\Ev$, so the value is exactly the largest tilted exponent available at tilt
$\beta$, and by \Cref{prp:upper-seq} the same bound survives conditioning, hence
holds for every \etm{}.

The companion duality in the opposite direction --- a variational formula for
an exponential integral in which a R\'enyi divergence appears as the penalty
rather than as the value --- is \cite[Thm.~2.1]{ACD15}: for real
$\beta'<\gamma'$, both nonzero, and bounded measurable $g:\Xc\to\R$,
\[
   \frac1{\beta'}\log\int e^{\beta' g}\,d\nu
   \ =\ \inf_{\theta}\lC\frac1{\gamma'}\log\int e^{\gamma' g}\,d\theta
      +\frac1{\gamma'}\,\Ren_{\gamma'/(\gamma'-\beta')}(\nu\|\theta)\rC ,
\]
the infimum over probability measures $\theta$ on $(\Xc,\Bs{\Xc})$. Here the
order $a=\gamma'/(\gamma'-\beta')$ lies outside $(0,1)$; the formula
$\Ren_a(Q\|P)=\frac{1}{a-1}\log\int q^{a}p^{1-a}d\mu$ of \Cref{def:renyi}
defines $\Ren_a$ for every $a\in\R\setminus\{0,1\}$ and is what is meant. The
coefficient $\frac1{\gamma'}$ is $\frac{1}{\gamma'-\beta'}$ in the normalisation
$\frac{1}{a(a-1)}\log\int q^{a}p^{1-a}d\mu$ used in \cite{ACD15}, and likewise
\Cref{thm:app-variational} is \cite[Thm.~1]{Ana18} rescaled by
$\Ren_a=a\,R_a$.
\end{remark}

\end{document}